\documentclass{amsart}
\usepackage{amsmath,amssymb,latexsym,amscd,color,array,enumerate}
\usepackage{tikz}
\usetikzlibrary{arrows,chains,matrix,positioning,scopes}
\usepackage{hyperref}

\numberwithin{equation}{section}

\newtheorem{theorem}{Theorem}[section]
\newtheorem{proposition}[theorem]{Proposition}
\newtheorem{conjecture}[theorem]{Conjecture}
\newtheorem{corollary}[theorem]{Corollary}

\newtheorem{lemma}[theorem]{Lemma}

\newtheorem{maintheorem}[theorem]{Main Theorem}
\theoremstyle{definition}

\newtheorem{remark}[theorem]{Remark}

\newtheorem{definition}[theorem]{Definition}

\def\endproof{\hfill$\square$\medskip}

\def\AA{\mathcal{A}}

\def\FF{\mathbb{F}}
\def\ZZ{\mathbb{Z}}
\def\QQ{\mathbb{Q}}
\def\CC{\mathbb{C}}

\def\bb{\mathfrak{b}}
\def\gg{\mathfrak{g}}

\def\ii{\mathbf{i}}

\def\kk{\Bbbk}

\def\UU{\mathcal{U}}

\def\PP{\mathcal{P}}

\newcommand{\bfa}{\mathbf{a}}
\newcommand{\bfB}{\mathbf{B}}
\newcommand{\bfb}{\mathbf{b}}

\newcommand{\bfd}{\mathbf{d}}
\newcommand{\bfE}{\mathbf{E}}
\newcommand{\bfe}{\mathbf{e}}

\newcommand{\bfi}{\mathbf{i}}

\newcommand{\bfS}{\mathbf{S}}

\newcommand{\bfX}{\mathbf{X}}

\newcommand{\ex}{\mathbf{ex}}

\newcommand{\cA}{\mathcal{A}}

\newcommand{\cC}{\mathcal{C}}

\newcommand{\cF}{\mathcal{F}}

\newcommand{\cH}{\mathcal{H}}

\newcommand{\cK}{\mathcal{K}}
\newcommand{\cL}{\mathcal{L}}

\newcommand{\cP}{\mathcal{P}}
\newcommand{\cQ}{\mathcal{Q}}

\newcommand{\cT}{\mathcal{T}}
\newcommand{\cU}{\mathcal{U}}

\renewcommand{\dim}{\operatorname{dim}}
\newcommand{\End}{\operatorname{End}}
\newcommand{\Ext}{\operatorname{Ext}}
\renewcommand{\ker}{\operatorname{ker}}
\newcommand{\Hom}{\operatorname{Hom}}

\newcommand{\rep}{\operatorname{rep}}

\newcommand{\half}{\frac{1}{2}}

\newcommand{\st}{\,:\,}

\def\sgn{\operatorname{sgn}}

\advance\oddsidemargin by-0.5in
\advance\evensidemargin by-0.5in
\advance\textwidth by 1in

\author{Arkady Berenstein}
\address{Department of Mathematics, University of Oregon,
Eugene, OR 97403, USA} 
\email{arkadiy@math.uoregon.edu}

\author{Dylan Rupel}
\address{\noindent Department of Mathematics, Northeastern University,
  Boston, MA 02115, USA}
\email{d.rupel@neu.edu}

\date{August 13, 2013}

\thanks{The first author was supported in part by the NSF grant DMS-1101507}

\makeatletter
\renewcommand{\@evenhead}{\tiny \thepage \hfill  A.~BERENSTEIN and  D.~RUPEL \hfill}

\renewcommand{\@oddhead}{\tiny \hfill Quantum cluster characters of Hall algebras
 \hfill \thepage}
\makeatother

\title{Quantum cluster characters of Hall algebras}

\begin{document}

\begin{abstract}
The aim of the present paper is to introduce a generalized quantum cluster character, which assigns to each object $V$ of a finitary Abelian category $\cC$ over a finite field $\FF_q$ and any sequence $\ii$ of simple objects in $\cC$ the element $X_{V,\ii}$ of the corresponding algebra $P_{\cC,\ii}$ of $q$-polynomials. We prove that if $\cC$ was hereditary, then the assignments $V\mapsto X_{V,\ii}$ define algebra homomorphisms from the (dual) Hall-Ringel algebra of $\cC$ to the $P_{\cC,\ii}$, which generalize the well-known Feigin homomorphisms from the upper half of a quantum group to $q$-polynomial algebras. 

If $\cC$ is the representation category of an acyclic valued quiver $(Q,\bfd)$ and $\ii=(\ii_0,\ii_0)$, where $\ii_0$ is a repetition-free source-adapted sequence, then we prove that the $\bfi$-character $X_{V,\ii}$ equals the quantum cluster character $X_V$ introduced earlier by the second author in \cite{rupel1} and \cite {rupel2}. Using this identification, we deduce a quantum cluster structure on the quantum unipotent cell corresponding to the square of a Coxeter element. As a corollary, we prove a conjecture from the joint paper \cite{bz5} of the first author with A.~Zelevinsky for such quantum unipotent cells. As a byproduct, we construct the quantum twist and prove that it preserves the triangular basis introduced by A.~Zelevinsky and the first author in \cite{bz6}.

\end{abstract}

\dedicatory{To the memory of Andrei Zelevinsky}

\maketitle

\tableofcontents

\section{Introduction}

The aim of this paper is two-fold:

$\bullet$ Generalize the Feigin homomorphism to Hall algebras of hereditary categories.

$\bullet$ Establish a quantum cluster structure on the image of the homomorphism. 

\medskip

The Feigin homomorphism $\Psi_\ii:U_+\to P_\ii$ was proposed by B.~Feigin in 1992 as an elementary tool for the study of quantized enveloping algebras $U_+(\gg)$, where $\gg$ is a Kac-Moody algebra, $\ii=(i_1,\ldots,i_m)$ is a sequence of simple roots of $\gg$,  and $P_\ii$ is an appropriate $q$-polynomial algebra in $t_1,\ldots,t_m$ (see \cite{ber,im,jo} or Section \ref{sec:main} below for details). Ultimately, Feigin's homomorphism assigns to each simple generator $E_i$, $i=1,\ldots,n$ of $U_+$ the linear $q$-polynomial $\sum\limits_{k:i_k=i} t_k$.  

It turns out that if $\ii$ is a reduced word for an element $w$ of the Weyl group $W$, then $I_w:=\ker\Psi_\ii$ depends only on $w$ and $\Psi_\ii$ defines an isomorphism between the skew-field of fractions of $U_w=U_+/I_w$ and the skew-field of fractions of $P_\ii$ (see e.g., \cite[Theorem 0.5]{ber}). 

This result can be viewed as an indicator of a possible quantum cluster structure on $U_w$, which, therefore, is more convenient to identify with the quotient of the dual algebra $\CC_q[N]$, where $N$ is a maximal unipotent subgroup of the corresponding Kac-Moody group $G$. For generic $q$ this identification is done for free because $\CC_q[N]\cong U_+$, but if $q=1$, then $\CC_q[N]$ specializes to the coordinate algebra $\CC[N]$ and $U_w$ -- to the coordinate algebra of the closure of the unipotent cell $N^w=B_-wB_-\cap N$,  where $B_-$ is a Borel subgroup of $G$ complementary to $N$ (see \cite{ber,bz3,bz4} for details). Since the (upper) cluster structure on $\CC[N^w]$ has been established in \cite{BFZ-cluster} by using a ``classical" analogue of $\Psi_\ii$, it is natural to expect an analogous result for each $U_w$.

One of the advantages of Feigin's homomorphism is that it allows to replace a very complicated algebra $U_w$ with the isomorphic subalgebra $\AA_\ii:=\Psi_\ii(\CC_q[N])$ of $P_\ii$ and look in a more efficient way for its quantum cluster structure inside the much simpler algebra $P_\ii$. In fact, we always have coefficients $C_1,\ldots,C_n$ in $\AA_\ii$ which are monomials in $t_1,\ldots,t_m$.  These form an Ore subset which allows to define the localization $\overline \AA_\ii$  of ${\mathcal A}_\ii$ by $C_1,\ldots,C_n$.  Our main result is the following.
 
\begin{theorem}
\label{th:cluster isomorphism intro}(Theorem \ref{th:cluster isomorphism})
For $\ii=(\ii_0,\ii_0)$, where $\ii_0=(i_1,\ldots,i_n)$ is any ordering of simple roots of $\gg$, the algebra $\overline \AA_\ii$ is 
an (acyclic) upper cluster algebra of rank $n$ with respect to an initial quantum cluster ${\bf X}=\{X_1,\ldots,X_n,C_1,\ldots,C_n\}$, where each $X_j$ is a monomial in $t_1^{-1},\ldots,t_{2n}^{-1}$.
\end{theorem}

The acyclicity of $\overline \AA_{(\ii_0,\ii_0)}$ will be obvious from the definition of the cluster structure and monomiality of coefficients $C_i$ is known for any $\ii$ (\cite[Theorem 3.1]{ber}), but the monomiality of the initial cluster variables is highly non-trivial, in particular, it implies the following ``quantum chamber ansatz" (cf \cite{bz5}).
 
\begin{corollary} 
In the assumptions of Theorem \ref{th:cluster isomorphism intro} for each $i=1,\ldots,n$ there exists an element $\tilde X_i\in \CC_q[N]$ such that $\Psi_\ii(\tilde X_i)=X_iC$, where $C$ is a monomial of $C_1,\ldots,C_n$.
\end{corollary}

Using these results, we settle a particular case of Conjecture 10.10 from \cite{bz5}. 

\begin{theorem}
\label{th:corollary bz intro}
If $\ii=(\ii_0,\ii_0)$, then the restriction of the homomorphism in \cite[Proposition 10.9]{bz5} to $\overline \AA_\ii$ is an isomorphism 
$\overline \AA_\ii\widetilde \to \CC_q[N^{c^2}]$, where $c$ is the Coxeter element in $W$ corresponding to the reduced word $\ii_0$.
\end{theorem}
We prove Theorem \ref{th:corollary bz intro} (see also Theorem \ref{th:corollary bz}) in Section \ref{subsec:proof corollary bz}.

It is natural to  expect (Conjecture \ref{conj:isomorphism Uii}) that  Theorem \ref{th:cluster isomorphism intro} and hence the  ``quantum chamber ansatz" hold without any assumptions on a reduced word $\ii$. It would be interesting to compare the conjecture  with the results of \cite{gls} where a cluster structure was established (for $\gg$ with symmetric Cartan matrix) on quantum Schubert cells $U_q(w)=\CC_q[N\cap wN_-w^{-1}]$, which are ``close relatives" of quantum unipotent cells $U_w$.

Our proof of Theorem \ref{th:cluster isomorphism intro} led us to the generalization of Feigin's homomorphism to Hall-Ringel algebras $\cH({\mathcal C})$ for most hereditary Abelian categories ${\mathcal C}$, which are natural generalizations and extensions of the quantum algebras $U_+$.  We first assign to each isomorphism class $[V]\in \cH({\mathcal C})$ a certain element $X_{V,\ii}$ of $P_\ii$ (we refer to it as the {\it quantum cluster $\ii$-character of $V$}), which is, in a sense, a generating function for flags in the object $V$ (formula  \eqref{eq:Feigin character}). We prove (Theorem \ref{th:Generalized Feigin homomorphism}) that, under certain co-finiteness conditions, for each hereditary category ${\mathcal C}$ the assignment $[V]\mapsto X_{V,\ii}$ defines a homomorphism of algebras 
\begin{equation}
\label{eq:Psi_ii intro}
\Psi_{\cC,\ii}:\cH({\mathcal C})\to P_\ii
\end{equation}
which, in the case when ${\mathcal C}=\rep_\FF(Q,\bfd)$ for an acyclic valued quiver $(Q,\bfd)$, restricts to Feigin's homomorphism from $U_+\subset \cH({\mathcal C})$ to $P_\ii$ (Corollary \ref{cor:classical_feigin}). 

Using the homomorphism \eqref{eq:Psi_ii intro}, we directly construct all non-initial cluster variables in $\overline \AA_{(\ii_0,\ii_0})$ as images $\Psi_\ii([E])$ where $[E]\in U_+$ runs over {\bf all} isomorphism classes of exceptional representations of $(Q,\bfd)$. This is achieved by identifying flags with Grassmannians of subobjects in $V$ and employing a quantum version of the famous Caldero-Chapoton formula (\cite{cc,ck}) developed by the second author in \cite{rupel1,rupel2} and independently by Qin in \cite{qin}.  It would be interesting to compare this with a similar bijection between flags and Grassmannians constructed in \cite{gls2}.

We hope to prove Conjecture \ref{conj:isomorphism Uii}, e.g., construct cluster variables in $\overline \AA_\ii$, in a similar fashion, by identifying our quantum cluster character $X_{V,\ii}$ with some quantization of a Caldero-Chapoton type character formula.

As a byproduct, we construct the {\it quantum twist} automorphism $\eta$ of $\overline \AA_\ii$ in the acyclic case and prove that it preserves the triangular basis of $\overline \AA_\ii$ (Corollary \ref{cor:twist of triangular basis}). We expect that the twist always exists (Conjecture  \ref{conj:isomorphism Uii}(c))  and preserves the canonical basis ${\bf B}_\ii$ of $\overline \AA_\ii$ (Conjecture \ref{conj:twist of basis}).

\subsection{Acknowledgments}
We are grateful to Christof Geiss, Jacob Greenstein, and Andrei Zelevinsky (who sadly passed away on April 10, 2013) for stimulating discussions. An important part of this work was done during the authors' visit to the MSRI in the framework of the ``Cluster algebras'' program and they thank the Institute and the organizers for their hospitality and support. The first author benefited from
the hospitality of Institut des Hautes \'Etudes Scientiques  and Max-Planck-Institut f\"ur Mathematik, which he
gratefully acknowledges. The authors are immensely grateful to  Gleb Koshevoy for stimulating discussions on the final stage of work on this paper.

\section{Definitions and main results}
\label{sec:main}

Let $\cC$ be a small finitary Abelian category of finite global dimension, that is,  $|\Ext^i(U,V)|<\infty$ for all $V,W\in \cC$, $i\ge 0$ and $\Ext^i(U,V)=0$ for all but finitely many $i\in \ZZ_{\ge 0}$ (where we follow the convention $\Ext^0(U,V)=\Hom(U,V)$).  

For $V,W\in \cC$ define (the square root of) the multiplicative Euler-Ringel form $\langle V,W \rangle$ by:
$$\langle V,W \rangle=\prod_{i=0}^\infty |\Ext^i(V,W)|^{\half(-1)^i}\ .$$
In fact, $\langle V,W \rangle$ depends only on Grothendieck classes $|V|$ and $|W|$ in the Grothendieck group $\cK(\cC)$ of the category $\cC$ and can be viewed as a bicharacter $\langle\cdot,\cdot\rangle:\cK(\cC)\times \cK(\cC)\to \kk^\times$, where we fix any field $\kk$  of characteristic $0$ containing all $\langle V,W\rangle^{\half}$ for $V,W\in \cC$.  

Let ${\bf S}=\{S_i\st i\in I\}$ be a set of pairwise non-isomorphic objects of ${\mathcal C}$. For any sequence  $\ii=(i_1,\ldots,i_m)\in I^m$ we define the skew-polynomial algebra $P_\ii=P_{{\mathcal C},{\bf S},\ii}$ over $\kk$ to be generated by $t_1,\ldots,t_m$ subject to the relations 
\begin{equation}
\label{eq:Pii}
t_\ell t_k=\langle S_{i_k},S_{i_\ell}\rangle\langle S_{i_\ell},S_{i_k}\rangle\cdot  t_k t_\ell\quad \text{ for $k<\ell$.}
\end{equation}

Furthermore, for each object $V\in {\mathcal C}$  we define the {\it (quantum cluster) $\ii$-character} $X_{V,\ii}$ of $V$ in (the completion of) $P_\ii$  by 
\begin{equation}
\label{eq:Feigin character}
X_{V,\ii}=\sum_{\bfa=(a_1,\ldots,a_m)\in \ZZ_{\ge 0}^m}  \prod_{1\le k<\ell\le m}\left(\frac{\langle S_{i_k},S_{i_\ell}\rangle}{\langle S_{i_\ell},S_{i_k}\rangle}\right)^{\half a_k a_\ell} \cdot |\cF_{\ii,\bfa}(V)|\cdot  t_1^{a_1}\cdots t_m^{a_m}\ ,
\end{equation}
where $\cF_{\ii,\bfa}(V)$ is the set of all flags of subobjects in $V$ of type $(\ii,\bfa)$, that is
\begin{equation}
\label{eq:flags}
\cF_{\ii,\bfa}(V)=\{0=V_m\subset V_{m-1}\subset \cdots \subset V_1 \subset V_0=V \st V_{k-1}/V_k\cong S_{i_k}^{\oplus a_k},k=1,\ldots,m\}.
\end{equation}

Note that the sum in \eqref{eq:Feigin character} is finite if $V$ has finitely many subobjects, otherwise $X_{V,\ii}$ belongs to a suitable completion of $P_\ii$. 

For each object $V$ of $\cC$ denote by $[V]$ its isomorphism class and let $\cH({\mathcal C})$ be the $\kk$-vector space freely spanned by all $[V]$ (this is the Hall-Ringel algebra, see section \ref{sec:gen_Feigin} for details).  Denote by ${\mathcal H}^*({\mathcal C})$ the finite {\it dual} Hall-Ringel algebra, which is the space of linear functions $\cH(\cC)\to \kk$ with a basis of  all delta-functions $\delta_{[V]}$ labeled by isomorphism classes $[V]$ of objects of ${\mathcal C}$. It is convenient to slightly rescale the basis of $\cH^*(\cC)$ as follows:
\begin{equation}
\label{eq:V*}
[V]^*=\langle V,V\rangle^{-\half}f(|V|)\cdot\delta_{[V]},
\end{equation} 
where $f:\cK(\cC)\to \kk^\times$ is a homomorphism of groups defined by $f(|S|)=\langle S,S\rangle^\half$ for all simple objects $S$ of $\cC$.
It turns out that the assignment 
\begin{equation}
\label{eq:V to X_Vii}
[V]^*\mapsto X_{V,\ii}
\end{equation}
often defines an algebra homomorphism from  ${\mathcal H}^*({\mathcal C})$ to $P_\ii$.
\begin{theorem}
\label{th:Generalized Feigin homomorphism}
Assume that ${\mathcal C}$ is hereditary and cofinitary and let  $\bfS=\{S_i\,|\,i\in I\}$ be a set of simple objects having no self-extensions.  Then for any sequence $\ii=(i_1,\ldots,i_m)\in I^m$ the assignment \eqref{eq:V to X_Vii} defines an algebra homomorphism 
$$\Psi_{{\mathcal C},\ii}\st{\mathcal H}^*({\mathcal C})\to P_\ii\ .$$
\end{theorem}
 
We will prove Theorem~\ref{th:Generalized Feigin homomorphism} in Section~\ref{sec:gen_Feigin} using generalities on bialgebras in braided monoidal categories presented in the Appendix (Section \ref{sec:appendix}). We will also provide some explicit formulas for $\Psi_{{\mathcal C},\ii}$ in terms of the $\cH(\cC)$-action on $\cH(\cC)^*$ (Proposition \ref{le:primitive actions proofs}). 
Note that for the category $\rep_\FF (Q,\bfd)$  of finite dimensional $\FF$-linear representations of any finite (valued) quiver $Q$ is hereditary (see e.g.,~\cite{hub}), so that 
Theorem~\ref{th:Generalized Feigin homomorphism} is applicable to such a category and to any collection ${\bf S}$ of simple objects if $Q$ has no vertex loops. 

The homomorphism $\Psi_{{\mathcal C},\ii}$ interpolates between a number of known homomorphisms.  
Denote by $U_+^*$ the subalgebra of $\cH^*(\cC)$ generated by the simple objects $[S_i]^*$, $i\in I$.

\begin{corollary}
\label{cor:classical_feigin}
The restriction of $\Psi_{{\mathcal C},\ii}$ to the subalgebra $U_+^*$ is the homomorphism $\Psi_\ii\st U_+^*\to P_\ii$ determined by (for $j\in I$):
\begin{equation}
\label{eq:Feigin homomorphism}
\Psi_\ii([S_j]^*)=\sum_{k\st i_k=j} t_k.
\end{equation}
\end{corollary}

Note that if ${\mathcal C}=\rep_\FF (Q,\bfd)$ for an acyclic valued quiver $(Q,\bfd)$ and ${\bf S}$ is the set of simple representations, then $U_+^*$ is the dual of the quantized enveloping algebra and $\Psi_\ii$ is the {\it Feigin homomorphism} used in \cite{ber,im,jo} to relate the skew-fields of fractions of $U_+^*$ and $P_\ii$. 

Also note that for ${\mathcal C}=\rep_\FF (Q,\bfd)$, if ${\bf S}$ is the set of all simple object in $\cC$, and $\ii=\ii_0$ is the repetition-free source-adapted sequence, then we obtain the following result (see \cite[Lemma 3.3]{reineke} for details).  
 
\begin{corollary}
The homomorphism $\Psi_{\rep_\FF (Q,\bfd),\ii_0}$ equals the Reineke homomorphism $\int$.
\end{corollary}

Note that for equally valued quivers $Q$ our homomorphisms $\Psi_{\rep_\FF (Q,\bfd),\ii}$ are related to the generalizations of $\int$ announced by Jiarui Fei (\cite[Proposition 6.1]{Fei}). 
 
Now we relate the $\ii$-character $X_{V,\ii}$ with the {\it quantum cluster character} $X_V$ defined by the second author in \cite{rupel1,rupel2}.  Namely, according to Definition 5.1 of \cite{rupel2}, for $V\in {\mathcal C}$, $X_V=X_V({\mathcal C})$ is an element of a certain based quantum torus ${\mathcal T}_{\mathcal C}$ given by:
\begin{equation}
\label{eq:character-integral}
X_V=\sum_{E\subset V} \langle V,V/E\rangle^{-1} \cdot X^{-|E|^*-{}^*|V/E|}\ ,
\end{equation}
where $|V|$ denotes the class of $V$ in the Grothendieck group $\cK({\mathcal C})$ and $(\cdot)^*$, $^*{(\cdot)}$ are certain dualities on $\cK({\mathcal C})$ (see Section~\ref{subsec:qcluster characters} for details).  

Given a lattice $L$ and a unitary bicharacter $\chi:L\times L\to \kk^\times$ define the {\it based quantum torus} ${\mathcal T}_\chi$ to be a $\kk$-algebra with the basis $X^e$, $e\in L$ subject to the relations:
\begin{equation}
\label{eq:based quantum torus}
X^eX^f=\chi(e,f)X^{e+f}
\end{equation}
for all $e,f\in \ZZ^m$.  Clearly, if $L=\ZZ^m$, then ${\mathcal T}_\chi$ is generated by $X_k^{\pm 1}:=X^{\pm \varepsilon_k}$, 
where $\{\varepsilon_1,\ldots,\varepsilon_m\}$ is the standard basis of $\ZZ^m$, subject to the relations:
\begin{equation}
\label{eq:q-torus presentation}
X_k X_\ell=q_{k\ell}X_\ell X_k
\end{equation}
for all $1\le k<\ell \le m$, where $q_{k\ell}=\frac{\chi(\varepsilon_k,\varepsilon_\ell)}{\chi(\varepsilon_\ell,\varepsilon_k)}$.  
And vice versa, any algebra with a presentation \eqref{eq:q-torus presentation} is isomorphic to some ${\mathcal T}_\chi$.  In what follows, we will always require that $\chi$ is {\it unitary}: $\chi(e,e)=1$ and $\chi(e,f)\chi(f,e)=1$ for all $e,f\in \ZZ^m$, so that $\chi$ is uniquely determined by the triangular array $(q_{k\ell})$.

Denote by ${\mathcal L}_\ii$ is the based quantum torus $P_\ii[t_1^{-1},\ldots,t_m^{-1}]$ with $t_k=t^{\varepsilon_k}$ for all $1\le k\le m$.  The following obvious fact will be used several times.
\begin{lemma}
\label{le:mon_change} 
For any sequence $\ii\in I^m$ and any $\ZZ$-linear automorphism $\varphi$ of $\ZZ^m$ there exists a unique unitary bicharacter $\chi_{\ii,\varphi}$ such that $\varphi$ extends to an isomorphism of based quantum tori $\hat\varphi_\ii\st {\mathcal L}_\ii\widetilde \to {\mathcal T}_{\chi_{\ii,\varphi}}$ defined by 
$$\hat\varphi(t^{\bf a})=X^{\varphi({\bf a})}$$
for ${\bf a}\in \ZZ^m$ (e.g., $X_k=X^{\varepsilon_k}=\hat \varphi(t^{\varphi^{-1}(\varepsilon_k)})$ for $k=1,\ldots,m$).
\end{lemma}

Let $A$ be any integer $I\times I$-matrix with $a_{ii}=2$ for $i\in I$.  For any sequence $\ii\in I^m$ we define $\ZZ$-linear automorphisms $\varphi_\ii,\rho_\ii$ of $\ZZ^m$ by
\begin{equation}
\label{eq:phi}
\varphi_\ii(\varepsilon_k)= -\varepsilon_k-\varepsilon_{k^-}-\sum\limits_{\ell\st \ell<k<\ell^+} a_{i_\ell,i_k}\varepsilon_\ell,\quad
\rho_\ii(\varepsilon_k)= \varepsilon_k-\sum\limits_{\ell\le k\st \ell^+=m+1} a_{i_\ell i_k}\varepsilon_\ell
\end{equation}
for $1\le k\le m$ (with the convention $\varepsilon_s=0$ if $s\notin [1,m]$). Here we used the notation from \cite{bz5}:
\begin{equation}
\label{eq:kpm}
k^+=\min\{\ell\st \ell>k, i_\ell=i_k\}\cup\{m+1\},~k^-=\max\{\ell\st \ell<k, i_\ell=i_k\}\cup\{0\}
\end{equation}
for $1\le k\le m$.  The inverse of $\varphi_\ii$ is  more complicated, see Lemma \ref{le:mon_change2}.

\begin{theorem}
\label{th:character=cluster character} 
Let $(Q,\bfd)$ be an acyclic valued quiver with $n$ vertices,  $A=(a_{ij})$ be the associated $n\times n$ Cartan matrix,  
$\ii_0$ be a source adapted sequence for $Q$, and  $\ii:=(\ii_0,\ii_0)$. Then:

(a) There exists an extension $(\tilde Q,\tilde \bfd)$ of $(Q,\bfd)$ on $2n$ vertices such that 
${\mathcal T}_{\tilde Q}={\mathcal T}_{\chi_{\ii,\rho_\ii^{-1} \varphi_\ii}}$. 

(b) For any representation $V$ of $(Q,\bfd)$ over a finite field $\FF$ of $q$ elements one has:  
\begin{equation}
\label{eq:character=cluster character in U+} 
\widehat{\rho_\ii^{-1} \varphi_\ii}(X_{V,\ii})=\tilde X_V 
\end{equation}
where $\tilde X_V$ denotes the quantum cluster character \eqref{eq:character-integral} attached to 
$\tilde {\mathcal C}=\rep_\FF(\tilde Q,\tilde \bfd)$  (here $V$ is viewed as an object of $\tilde {\mathcal C}$ under the natural embedding).    
\end{theorem}

We prove  Theorem \ref{th:character=cluster character} in Section \ref{sec:valued_quivers}. 

\begin{remark} Theorems~\ref{th:Generalized Feigin homomorphism} and \ref{th:character=cluster character} together explain observations of \cite{cc,ck} that multiplication of cluster characters  resemble the multiplication in the dual Hall-Ringel algebra of $\rep_\FF(Q,\bfd)$.  
\end{remark}
It follows from \cite[Theorem 5.1]{cx} that if $V$ is a {\it rigid} object of $\cC=\rep_\FF(Q,\bfd)$ (i.e.  $\Ext^1_\cC(V,V)=0$) and $Q$ is acyclic, then the basis vector $[V]^*$ of $\cH^*(\cC)$ belongs to the composition algebra $U_+^*$.   These observations and Theorem \ref{th:character=cluster character} imply that the Grassmannians of subobjects of rigid objects often have counting polynomials. 
To state this result, we need more notation. 

For each ${\bf e}\in \cK(\cC)$ and $V\in\cC$ denote by $Gr_\bfe(V)$ the set of all subobjects $E\subset V$ such that $|E|=\bfe$, so that the formula \eqref{eq:character-integral} reads:
\begin{equation}
\label{eq:character-Gr}
X_V=\sum_{{\bf e}\in \cK({\mathcal C})} \langle |V|,|V|-{\bf e}\rangle^{-1} \cdot |Gr_\bfe(V)| \cdot X^{-{\bf e}^*-{}^*(|V|-{\bf e})}\ .
\end{equation}
From now on, until the end of the section,  $U_+^*$ will be the generic composition algebra (see \cite[Section 5]{ringel3} for details),  that is, $U_+^*=U_+^*(A)$ is a $\kk(q^\half)$-algebra generated by $x_i=[S_i]^*$, $i\in I$  subject to the quantum Serre relations determined by a symmetrizable $I\times I$ Cartan matrix $A$. 
Accordingly, we view $\cL_\ii$ as a $\kk(q^\half)$-algebra generated by $t_k^\pm$ subject to the relations (cf \eqref{eq:Pii}):
$$t_\ell t_k=q^{c_{i_k,i_\ell}} t_kt_\ell$$
for $1\le k<\ell\le m$, where $c_{ij}=d_ia_{ij}=a_{ij}d_j$ is the $(ij)$th entry of the symmetrized Cartan matrix.
Combining this and \eqref{eq:character=cluster character in U+} with \cite[Theorem 5.1]{cx}, we generalize  \cite[Corollary 1.2]{rupel2}.

\begin{corollary} 
Let $(Q,\bfd)$ be an acyclic valued quiver and $V\in Rep_\FF(Q,\bfd)$ be rigid, i.e., has no self-extensions. Then for any $\ii\in I^m$, ${\bf a}\in \ZZ^m_{\ge 0}$ 
there exists a  polynomial $p_{\ii,{\bf a}}^V(x)\in \ZZ[x]$ such that 
$$|\cF_{\ii,{\bf a}}(V)|=p^V_{\ii,{\bf a}}(|\FF|)\ .$$
In particular, for any ${\bf e}\in \cK(\cC)$ one has  
$$|Gr_\bfe(V)|=p^V_{(\ii_0,\ii_0),{\bf a}_{\bf e}}(|\FF|)\ ,$$ 
where ${\bf a}_{\bf e}=(v_1-e_1,\ldots,v_n-e_n,e_1,\ldots,e_n)$ and  $|V|=\sum_{i\in I} v_i|S_i|$, ${\bf e}=\sum_{i\in I} e_i|S_i|$.
\end{corollary}
 
\begin{remark}  In \cite{qin} Qin  proved that for an equally valued $Q$ all counting polynomials $p^V_{(\ii_0,\ii_0),{\bf a}_{\bf e}}$  for $Gr_\bfe(V)$ belong to $\ZZ_{\ge 0}[x]$
(this was announced earlier by Caldero and Reineke in \cite{caldrein}, but their proof was incomplete). 
We expect that all $p^V_{\ii,{\bf a}}$ for rigid $V$ have nonnegative coefficients. This has been recently confirmed by the second author in \cite{rupel3} when $Q$ is any acyclic valued quiver with two vertices and $\ii=(\ii_0,\ii_0)$.  Moreover, based on numerous examples and results of \cite{efimov}, we expect that all $p_{\ii,{\bf a}}^V$ for rigid $V$ are  unimodular, i.e., $p_{\ii,{\bf a}}^V\in q^n\sum\limits_{k\ge 0} \ZZ_{\ge 0}\cdot  (k)_q$ for some $n\ge 0$, where $(k)_q:=\frac{q^k-1}{q-1}$. 
\end{remark}

For each sequence $\ii=(i_1,\ldots,i_m)$ we define the  {\it upper cluster algebra} ${\mathcal U}_\ii=\UU({\bf X}_\ii,\tilde B_\ii)\subset {\mathcal L}_\ii$ of the seed $({\bf X}_\ii,\tilde B_\ii)$, where ${\bf X}_\ii=\{t^{\varphi_\ii^{-1}(\varepsilon_1)},\ldots,t^{\varphi_\ii^{-1}(\varepsilon_m)}\}$ in the notation \eqref{eq:phi}  and $\tilde B_\ii$ is the exchange matrix defined in \cite[Section 8.2]{bz5} (see also Section \ref{sec:Special compatible pairs}). 

\begin{maintheorem}
\label{th:cluster isomorphism}
Let $A$ be a symmetrizable $I\times I$ Cartan matrix and let $\ii=(\ii_0,\ii_0)$, where $\ii_0$ is an ordering of $I$. Then 

(a)  $\Psi_\ii(U_+^*)\subset {\mathcal U}_\ii$.

(b) The  assignment $V\mapsto \Psi_\ii([V]^*)$ is a bijection between isomorphism classes of exceptional representations of the  associated valued quiver $(Q,\bfd)$ and non-initial cluster variables in ${\mathcal U}_\ii$.

(c) If $\ii$ is reduced, then the localization of $\Psi_\ii(U_+^*)$ by the cluster coefficients is equal to ${\mathcal U}_\ii$.
 
\end{maintheorem}

We prove Theorem~\ref{th:cluster isomorphism} in Section~\ref{sec:unip}.
 
\begin{theorem} 
\label{th:twist double coxeter} 
Assume that $\ii=(\ii_0,\ii_0)$ is reduced. Then there exists an isomorphism $\eta:\UU_\ii\to \UU_\ii$ such that
$$\eta(t^{\varphi_\ii^{-1}(\varepsilon_k)})=\Psi_\ii([V_k]^*),~\eta(t^{\varphi_\ii^{-1}(\varepsilon_{k+|I|})})=t^{-\varphi_\ii^{-1}(\varepsilon_{k+|I|})}$$
for $k=1,\ldots,|I|$, where $V_k$ is the exceptional object of $\rep_\FF(Q,\bfd)$ such that $|V_k|=s_{i_1}\cdots s_{i_{k-1}}|S_{i_k}|$. 
\end{theorem} 

We prove Theorem \ref{th:twist double coxeter} in Section \ref{subsec: proof of theorem twist double coxeter}.  Using these results, we settle a particular case of Conjecture 10.10 from \cite{bz5}. 
Indeed, following \cite[Section 9.3]{bz5},  for any extremal weight $\gamma$ of a fundamental simple  $U_q(\gg)$-module $V_{\omega_i}$ ($i\in I$) one defines an extremal vector $\Delta_\gamma\in V_{\omega_i}\subset  U_+^*$ 
(it is sometimes called a ``generalized minor"), where $\omega_i$ is the $i$-th fundamental weight of $\gg$.  Using  $q$-commutation relations between various $\Delta_\gamma$ (see e.g., \cite[Theorem 10.1]{bz5}), one has an injective homomorphism of algebras (for each reduced word $\ii=(i_1,\ldots,i_m)$ for an element $w$ in the Weyl group $W$ of $\gg$):
\begin{equation}
\label{eq:bz-hom}
\cL_\ii\to Frac(\kk_q[N^w])
\end{equation}
given by:
$$t^{\varphi_\ii^{-1}(\varepsilon_k)}\mapsto \underline \Delta_{s_{i_1}\cdots s_{i_k}\omega_{i_k}}$$
for $k=1,\ldots,m$, where $\underline x$ is the image of $x\in U_+^*$ under the canonical projection $U_+^*\twoheadrightarrow \kk_q[N^w]$.
 
A particular case of Conjecture 10.10 from \cite{bz5} asserts that the restriction of \eqref{eq:bz-hom} to $\UU_\ii$ is an isomorphism of algebras
$$\UU_\ii\widetilde \to \kk_q[N^w] \ .$$
It follows from the results of \cite{ber} that $\ker\Psi_\ii$ is the defining ideal of $\kk_q[N^w]$ in $U_+^*$ for any reduced word $\ii$ for $w\in W$, therefore,  $\Psi_\ii$ factors through the injective homomorphism 
$$\underline \Psi_\ii:\kk_q[N^w]\hookrightarrow \cL_\ii$$ 
(see Section \ref{subsec:proof corollary bz} for details).  The following result settles the above  conjecture for $\ii=(\ii_0,\ii_0)$.

\begin{theorem}
\label{th:corollary bz}
In the assumptions of Theorems \ref{th:cluster isomorphism}(c) and \ref{th:twist double coxeter}, the restriction of \eqref{eq:bz-hom} to $\UU_\ii$ is an isomorphism 
$$\underline \Psi_\ii^{-1}\circ \eta:\UU_\ii\widetilde \to \kk_q[N^{c^2}]\ ,$$ 
where $c$ is the Coxeter element in $W$ corresponding to the reduced word $\ii_0$. 
\end{theorem}

We prove Theorem \ref{th:corollary bz} in Section \ref{subsec:proof corollary bz}.  It is natural to expect that both Theorem~\ref{th:cluster isomorphism} and Theorem \ref{th:corollary bz} hold for all reduced words $\ii$. 

\begin{conjecture} 
\label{conj:isomorphism Uii}
Let $A$ be a symmetrizable $I\times I$ Cartan matrix and let $\ii=(i_1,\ldots,i_m)$ be a reduced word for an element $w\in W$. Then:

(a) $\Psi_\ii(U_+^*)\subset  {\mathcal U}_\ii$. Moreover, the localization of $\Psi_\ii(U_+^*)$ by the cluster coefficients is ${\mathcal U}_\ii$. 

(b) For each exceptional representation $V$ of $(Q,\bfd)$ the element $\Psi_\ii([V]^*)\in \cU_\ii$ is a  cluster variable.

(c) The assignment $t^{\varphi_\ii^{-1}(\varepsilon_k)}\mapsto \Psi_\ii(\Delta_{s_{i_1}\cdots s_{i_k}\omega_{i_k}})$ defines an isomorphism $\eta_\ii:\UU_\ii\widetilde \to \UU_\ii$.

\end{conjecture}
 
\begin{remark} 
The ``classical" version of the conjecture is known: it follows from  \cite[Theorem 2.10]{BFZ-cluster} and the existence of the ``classical" twist $\eta_w$, a certain automorphism of the unipotent cell $N^w$ (\cite[Theorem 1.2]{bz3}) which interpolates between two embeddings of tori $(\kk^\times)^m\hookrightarrow N^w$, $(\kk^\times)^m\hookrightarrow N^w$. 
\end{remark} 

\begin{remark} 
Similar to Theorem \ref{th:corollary bz}, Conjecture \ref{conj:isomorphism Uii}(c) implies \cite[Conjecture 10.10]{bz5}, i.e., 
for any reduced word $\ii$ for $w\in W$, the restriction of \eqref{eq:bz-hom} to $\UU_\ii$ is an isomorphism $\underline{\Psi}_\ii^{-1}\circ \eta_\ii:\UU_\ii\widetilde \to \kk_q[N^w]$.
\end{remark}

We conclude the section with the relationship between the twist $\eta:\UU_\ii\to \UU_\ii$ from Theorem \ref{th:twist double coxeter}  
and canonical basis in $\UU_\ii$. 

Recall that in \cite{bz6} A.~Zelevinsky and the first author  constructed a {\it triangular basis} ${\bf B}(\Sigma)$ in the upper cluster algebra $\UU(\Sigma)$ for each acyclic quantum seed $\Sigma$ in $\cU_\bfi$ and proved that ${\bf B}(\Sigma)$ does not depend on the choice of $\Sigma$ (\cite[Theorem 1.6]{bz6}). This basis consists of bar-invariant elements and has a triangular transition matrix (in terms of powers of $q^\half$) to the initial standard monomial  basis ${\mathcal E}(\Sigma)=\{E_\bfa,\bfa\in \ZZ^m\}$. Since $\eta$ commutes with the bar-involution and sends the initial standard monomial basis  ${\mathcal E}(\Sigma_\bfi)$ to ${\mathcal E}(\Sigma')$ for some other seed $\Sigma'$ of $\UU_\ii$, the following is immediate. 

\begin{corollary} 
\label{cor:twist of triangular basis}
In the assumptions of Theorem \ref{th:twist double coxeter}, one has
$$\eta({\bf B}(\Sigma_\ii))={\bf B}(\Sigma_\ii)\ .$$

\end{corollary}

The basis ${\bf B}(\Sigma)$ is an analogue of the {\it dual canonical basis} (see e.g., discussion in \cite[Remark 1.8]{bz6}). To make this analogy precise, we  define the canonical basis ${\bfB}'_\ii$ in $\Psi_\ii(U_+^*)$ and relate it to the twist $\eta_\ii$ from Conjecture  \ref{conj:isomorphism Uii}(c). 

Indeed, let ${\bf B}^*$ be the dual canonical basis of $U_+^*$. 
For each reduced $\ii\in I^m$ we define ${\bf B}'_\ii\subset \cL_\ii$ by:
$${\bf B}'_\ii=\Psi_\ii({\bf B}^*)\setminus \{0\}\ .$$
The following is immediate (part (a) follows from \cite[Proposition 4.2]{Lusztig-problems}, part (b) follows from that $\Psi_\ii$ commutes with the bar-involutions and part (c) -- from \cite[Theorem 6.18]{{kimura}}).

\begin{proposition} 
\label{pr:canonical basis}
For each reduced $\ii\in I^m$ one has:

(a) The set ${\bf B}'_\ii$ is a basis of $\Psi(U_+^*)$.

(b) For each $b\in {\bf B}'_\ii$ one has $\bar b=b$, where the bar-involution $t\mapsto \bar t$ on $\cL_\ii$ is the only $\QQ$-linear anti-automorphism 
such that $\bar t_k=t_k$, $\bar q^\half=q^{-\half}$.

(c) For any $b\in {\bf B}'_\ii$ and each $k\in [1,m]\setminus \ex$ there is $r\in \ZZ$ such that $q^{\frac{r}{2}}b\cdot X_k\in {\bf B}'_\ii$.  

\end{proposition}

Proposition \ref{pr:canonical basis}(c), in particular, implies that the set ${\bf B}_\ii$ of all bar-invariant elements of the form:
$$q^{\frac{r}{2}}\cdot  b\prod\limits_{k\in [1,m]\setminus \ex} X_k^{a_k}\ ,$$
where $r,a_k\in\ZZ$, $b\in {\bf B}'_\ii$  is a basis of the Ore localization of  $\Psi_\bfi(U_+^*)$ by all $X_k^{-1},k\in [1,m]\setminus \ex$. 

Thus, Conjecture \ref{conj:isomorphism Uii}(a) asserts that ${\bf B}_\ii$ is a basis of $\UU_\ii$.
Now we  (yet conjecturally) relate the bases with the twist $\eta_\ii$.

\begin{conjecture} 
\label{conj:twist of basis}
In the assumptions of Conjecture  \ref{conj:isomorphism Uii} one has:

(a) $\eta_\ii({\bf B}_\ii)={\bf B}_\ii$.

(b) If $\ii=(\ii_0,\ii_0)$, then ${\bf B}_\ii$ is the triangular basis ${\bf B}(\Sigma_\ii)$.

\end{conjecture}

\section{Hall-Ringel algebras and proof of Theorem~\ref{th:Generalized Feigin homomorphism}}\label{sec:gen_Feigin}
In this section we recall some basic facts about Hall-Ringel algebras of finitary categories and present some results on generalizations of Feigin homomorphisms.

Let $\cC$ be a small finitary Abelian category of finite global dimension.  For an object $V\in \cC$ we will write $[V]$ for the isomorphism class of $V$ and write $|V|$ for the class of $V$ in the Grothendieck group $\cK(\cC)$.  Let $\cH(\cC)=\bigoplus\kk\cdot[V]$ be the free $\cK(\cC)$-graded $\kk$-vector space spanned by the isomorphism classes of objects of $\cC$ with the natural grading via class in $\cK(\cC)$.  For $U,V,W\in \cC$ define the (finite) set ${\mathcal F}_{U,W}^V$ by:
$${\mathcal F}_{U,W}^V=\{R\subset V\,|\, R\cong W,V/R\cong U\}\ .$$
\begin{proposition}[\cite{ringel2}] 
The assignment
$$[U][W]:=\langle|U|,|W|\rangle \sum_{[V]} |\mathcal F_{UW}^V|\cdot [V]$$
defines an associative multiplication on $\cH(\cC)$.

\end{proposition}
The algebra $\cH(\cC)$ is known as the \emph{Hall-Ringel algebra}.  The following formula for iterated multiplication is well-known and follows easily from the definitions:
\begin{equation}
\label{eq:iter_mult}
[U_1]\cdots[U_m]=\sum_{[V]} \prod\limits_{k< \ell} \langle |U_k|,|U_\ell| \rangle \cdot |{\mathcal F}_{U_1,\ldots, U_m}^V|\cdot [V] \ ,
\end{equation}
for $U_1,\ldots,U_m,V\in \cC$, where
$${\mathcal F}_{U_1,\ldots,U_m}^V=\{0=V_m\subset V_{m-1}\subset \cdots \subset V_1\subset V_0=V\st V_{k-1}/V_k\cong U_k\}\ .$$  

We will consider $\cH(\cC)\otimes\cH(\cC)$ as an algebra with the twisted multiplication:
$$([U_1]\otimes [U_2])([V_1]\otimes[V_2])=\langle U_2,V_1\rangle \langle V_1,U_2\rangle  [U_1][V_1]\otimes[U_2][V_2]\ .$$

\begin{proposition}[\cite{green}]
\label{pr:green}
The map $\Delta\st \cH(\cC)\to\cH(\cC)\otimes\cH(\cC)$ given by
$$\Delta([V])=\sum_{[U],[W]} \langle|U|,|W|\rangle \cdot|{\mathcal F}_{UW}^V|\cdot \frac{|Aut(U)|\cdot |Aut(W)|}{|Aut(V)|}\cdot  [U]\otimes[W]$$
defines a coalgebra structure on $\cH(\cC)$.  Moreover, if $\cC$ is hereditary and cofinitary, then $\Delta$ is an algebra homomorphism $\cH(\cC)\to \cH(\cC)\otimes\cH(\cC)$.
\end{proposition}

Note that if the objects of $\cC$ do not have finitely many subobjects it will be necessary to consider the codomain of $\Delta$ to be the completion $\cH(\cC)\hat{\bigotimes}\cH(\cC)$, see \cite{schiffmann} for more details.

For each simple object $S\in \cC$ define the exponential 
\begin{equation}
\label{eq:Es}
E_S:=\exp_{q}([S])=\sum\limits_{n=0}^\infty \frac{1}{(n)_q!}[S]^n\in\hat\cH(\cC)\ ,
\end{equation}
where $q=\langle S,S\rangle^2$ and  $(n)_q!=(1)_q\cdots (n)_q$ for $(k)_q=\frac{q^k-1}{q-1}$.  

Recall that for a given (complete) coalgebra $C$ with the coproduct $\Delta:C\to C\hat \bigotimes C$ a non-zero element $c\in C$ is called {\it grouplike} if $\Delta(c)=c\otimes c$. 

\begin{proposition}
\label{pr:Sgrouplike}
For each simple object $S$ with no self-extensions one has:

(a) $E_S$ is a grouplike element of $\hat\cH(\cC)$.

(b)  $\frac{1}{(n)_q!}[S]^n=\langle S,S\rangle^{\frac{n(n-1)}{2}} [S^{\oplus n}]$, where $q=|\End(S)|$.
 
\end{proposition}

\begin{proof} To prove (a) note that $S$ is primitive, i.e., $\Delta([S])=[S]\otimes 1+1\otimes [S]$, hence the assertion follows from Lemma~\ref{le:primitive}.
 
For (b) we proceed by induction in $n$.  Indeed, if $n=0$, we have nothing to prove. Let $n\ge 1$. Then: 
$$[S^{\oplus n-1}][S]=\langle S^{\oplus n-1},S\rangle F_{S^{\oplus {n-1}},S}^{S^{\oplus n}}[S^{\oplus n}]=\langle S,S\rangle^{n-1}(q^{n-1}+\cdots+q+1)[S^{\oplus n}]\ .$$
Using the inductive hypothesis, this is equivalent to:
$$\frac{1}{(n-1)_q!}\langle S,S\rangle^{-\frac{(n-1)(n-2)}{2}}[S]^{n-1}\cdot [S]=\langle S,S\rangle^{n-1}(n)_q[S^{\oplus n}] \ .$$
This proves (b).  The proposition is proved. \end{proof}

We are now ready to prove Theorem~\ref{th:Generalized Feigin homomorphism}.  Assume in addition that $\cC$ is hereditary and cofinitary.  The algebra $\UU:=\cH(\cC)$ is naturally graded by $\Gamma:=\cK(\cC)$ and according to Proposition \ref{pr:green}, $\UU$ is a braided bialgebra in $\cC_\chi$ (see Section \ref{sec:appendix}), where $\chi:\Gamma\times \Gamma\to \kk^\times$ is the bicharacter given by the symmetrized Euler-Ringel form:
\begin{equation}
\label{eq:chi ringel}
\chi(|V|,|W|)=\langle V,W\rangle\cdot \langle W,V\rangle \ .
\end{equation}

Let $\bfS=\{S_i\st  i\in I\}$ be the set of all (up to isomorphism) simple objects of $\cC$.  Let $\ii=(i_1,\ldots, i_m)\in I^m$ be such that each $S_{i_k}$ has no self-extensions.  We set $\cP:=\cL_\ii$, where $\cL_\ii$ is the quantum torus as in Section~\ref{sec:main}.  Write $\bfE=(E_{i_1},\ldots,E_{i_m})$ for the exponentials $E_i:=E_{S_i}$ (see \eqref{eq:Es}). By Proposition \ref{pr:Sgrouplike}(a), each member of the family $\bfE$ is a grouplike element in $\UU\hat \otimes \UU$.  It is easy to see that ${\bf E}$  is ${\mathcal P}$-adapted (in terminology of Section~\ref{sec:appendix}) via the homomorphisms $\tau_k\st \ZZ\cdot|S_{i_k}|\to {\mathcal P}$ defined by $\tau_k(r|S_{i_k}|)=t_k^r$, $k=1,\ldots,m$, $r\in \ZZ$.  Thus all hypotheses of Theorem \ref{th:abstract_Feigin} are satisfied and the assignment
\begin{equation}
\label{eq:Feigin explicit via pairing}
x\mapsto \sum\limits_{\bfa\in\ZZ_{\ge0}^m} x\left([S_{i_1}]^{(a_1)}\cdots [S_{i_m}]^{(a_m)}\right)t_1^{a_1}\cdots t_m^{a_m}
\end{equation}
for $x\in\cH^*(\cC)$
is an algebra homomorphism $\Psi_{\cC,\ii}:\cH^*(\cC)\to {\mathcal L}_\ii$, where we abbreviate $[S]^{(n)}:=\frac{1}{(n)_q!} [S]^n$, and $q=|\End(S)|$. It remains to show that 
\begin{equation}
\label{eq:Psi=XVi}
\Psi_{\cC,\ii}([V]^*)=X_{V,\ii}
\end{equation}
for all objects $V$ of $\cC$.
 
To verify \eqref{eq:Psi=XVi}, we need the following fact.
\begin{lemma}
\label{le:iter_mult}
For $\bfa=(a_1,\ldots,a_m)\in\ZZ_{\ge0}^m$ one has:
$$[S_{i_1}]^{(a_1)}\cdots[S_{i_m}]^{(a_m)}=\sum\limits_{[V]} \left(\prod_{k=1}^m\langle S_{i_k},S_{i_k}\rangle^{\frac{a_k(a_k-1)}{2}}\right)\cdot \left(\prod\limits_{k<\ell}\langle S_{i_k},S_{i_\ell}\rangle^{a_ka_\ell}\right)\cdot  |\cF_{\ii,\bfa}(V)|\cdot [V]\ ,$$
where $ \cF_{\ii,\bfa}(V)$ is defined in \eqref{eq:flags}.
\end{lemma}

\begin{proof} 
Indeed, by Proposition \ref{pr:Sgrouplike}(b) and \eqref{eq:iter_mult}, we obtain:
$$[S_{i_1}]^{(a_1)}\cdots[S_{i_m}]^{(a_m)}=\left(\prod_{k=1}^m\langle S_{i_k},S_{i_k}\rangle^{\frac{a_k(a_k-1)}{2}}\right)[S_{i_1}^{\oplus a_1}]\cdots [S_{i_m}^{\oplus a_m}]$$
$$=\sum_{[V]} \left(\prod_{k=1}^m\langle S_{i_k},S_{i_k}\rangle^{\frac{a_k(a_k-1)}{2}}\right)\cdot\left(\prod\limits_{k< \ell} \langle S_{i_k}^{\oplus a_k},S_{i_\ell}^{\oplus a_\ell} \rangle\right)\cdot  \left|\cF_{S_{i_1}^{\oplus a_1},\ldots,S_{i_m}^{\oplus a_m}}^V\right|\cdot[V]\ .$$
This proves the lemma because ${\mathcal F}_{S_{i_1}^{\oplus a_1},\ldots,S_{i_m}^{\oplus a_m}}^V=\cF_{\ii,\bfa}(V)$ and $\langle S_{i_k}^{\oplus a_k},S_{i_\ell}^{\oplus a_\ell} \rangle=\langle S_{i_k},S_{i_\ell}\rangle^{a_ka_\ell}$.
\end{proof}
 
Using definition \eqref{eq:V*} of $[V]^*$ and Lemma \ref{le:iter_mult}, we compute:
$$\Psi_{\cC,\ii}([V]^*)=\langle V,V\rangle^{-\half}f(|V|) \Psi_{\cC,\ii}(\delta_{[V]})=\langle V,V\rangle^{-\half}f(|V|) \sum\limits_{\bfa\in\ZZ_{\ge0}^m} \delta_{[V]}\left([S_{i_1}]^{(a_1)}\cdots [S_{i_m}]^{(a_m)}\right)\cdot t_1^{a_1}\cdots t_m^{a_m}$$
$$=\langle V,V\rangle^{-\half}f(|V|) \sum\limits_{\bfa\in\ZZ_{\ge0}^m}  \left(\prod_{k=1}^m\langle S_{i_k},S_{i_k}\rangle^{\frac{a_k(a_k-1)}{2}}\right)\cdot \left(\prod\limits_{k<\ell}\langle S_{i_k},S_{i_\ell}\rangle^{a_ka_\ell}\right)\cdot  |\cF_{\ii,\bfa}(V)|\cdot t_1^{a_1}\cdots t_m^{a_m}\ .$$
Furthermore, note that  $\cF_{\ii,\bfa}(V)\ne \emptyset$ implies that $|V|=\sum\limits_{k=1}^m a_k\cdot |S_{i_k}|$, hence 
$$f(|V|)=\prod_{k=1}^m f(|S_{i_k}|)^{a_k}=\prod_{k=1}^m  \langle S_{i_k},S_{i_k}\rangle^{\half a_k}\ .$$
Also, 
$$\langle V,V\rangle^{-\half}=\prod\limits_{k,\ell}\langle S_{i_k},S_{i_\ell}\rangle^{-\half a_ka_\ell}=\prod\limits_{k=1}^m\langle S_{i_k}, S_{i_k}\rangle^{-\half a_k^2}\prod\limits_{k<\ell}\langle S_{i_k}, S_{i_\ell}\rangle^{-\half a_ka_\ell}\prod\limits_{k<\ell}\langle S_{i_\ell}, S_{i_k}\rangle^{-\half a_ka_\ell}.$$
Therefore, 
$$\Psi_{\cC,\ii}([V]^*)= \sum\limits_{\bfa\in\ZZ_{\ge0}^m}\langle V,V\rangle^{-\half}f(|V|)  \left(\prod_{k=1}^m\langle S_{i_k},S_{i_k}\rangle^{\frac{a_k(a_k-1)}{2}}\right)\cdot \left(\prod\limits_{k<\ell}\langle S_{i_k},S_{i_\ell}\rangle^{a_ka_\ell}\right)\cdot  |\cF_{\ii,\bfa}(V)|\cdot t_1^{a_1}\cdots t_m^{a_m}$$
$$=\sum\limits_{\bfa\in\ZZ_{\ge0}^m}  \prod\limits_{1\le k<\ell\le m}\left(\frac{\langle S_{i_k},S_{i_\ell}\rangle}{\langle S_{i_\ell},S_{i_k}\rangle}\right)^{\half a_ka_\ell}\cdot  |\cF_{\ii,\bfa}(V)|\cdot t_1^{a_1}\cdots t_m^{a_m}=X_{V,\ii} \ .$$
This verifies \eqref{eq:Psi=XVi} and thus finishes the proof of Theorem \ref{th:Generalized Feigin homomorphism}. \endproof

We conclude the section with a  formula for $\Psi_{\cC,\ii}$ in terms of   $\cH(\cC)$-actions on $\cH^*(\cC)$. We need some notation.

Given a finitary Abelian category $\cC$, for  each $U\in \cC$ define linear maps $\partial_U,\partial_U^{op}:\cH^*(\cC)\to \cH^*(\cC)$ 
(in the notation \eqref{eq:V*}) by:
\begin{equation}
\label{eq:partial derivative}
\partial_U(\delta_{[V]})=\sum_{[W]} \langle W,U\rangle  |\cF_{W,U}^V|\cdot  \delta_{[W]},~\partial_U^{op}(\delta_{[V]})=\sum_{[W]} \langle U,W\rangle |\cF_{U,W}^V|\cdot  \delta_{[W]} \ .
\end{equation}

The following fact is obvious. 

\begin{lemma} For any $U,V\in \cC$, $x\in \cH^*(\cC)$ one has:
$$\partial_U(x)([W])=x([W][U]),~\partial_U^{op}(x)([W])=x([U][W])$$
hence $[U]\otimes x\mapsto \partial_U(x)$ and  $[U]\otimes x\mapsto \partial^{op}_U(x)$ are respectively left and right  $\cH(\cC)$-actions on $\cH^*(\cC)$.
\end{lemma}

For each $U\in \cC$ define a linear transformation $K_U:\cH^*(\cC)\to \cH^*(\cC)$ by 
$$K_U(\delta_{[V]})=\chi(|U|,|V|)^{-1}\cdot \delta_{[V]}$$ 
for any $V\in \cC$, where $\chi(\cdot,\cdot)$ is as in \eqref{eq:chi ringel}.

Clearly, $K_U=K_{U'}\circ K_{U/U'}$ for any sub-object $U'$ of $U$. It is also easy to see that 
\begin{equation}
\label{eq:K_U E_U commutation}
\partial_UK_{U'}=\chi(|U|,|U'|)^{-1}K_{U'} \partial_U,~ \partial_U^{op}K_{U'} =\chi(|U|,|U'|)^{-1}K_{U'} \partial_U^{op}
\end{equation}

Then denote 
\begin{equation}
\label{eq:partial underlined}
\underline \partial_U:=K_U^\half \partial_U,~\underline \partial_U^{op}:=K_U^\half \partial_U^{op}
\end{equation}
for all $U\in\cC$.

The following is an immediate corollary of Lemma \ref{le:primitive actions}.

\begin{lemma}
\label{le:primitive actions proofs} Let $\cC$ be a hereditary cofinitary category. Then for any simple object $S\in \cC$ and any  $x,y\in \cH^*(\cC)$  one has: 
$$\partial_S(xy)=  \partial_S(x)K_S^{-1}(y)+  x \partial_S(y),~ \partial_S^{op}(xy)=\partial_S^{op}(x)y+ K_S^{-1}(x)\partial_S^{op}(y)\ ,$$
$$\underline \partial_S(xy)= \underline \partial_S(x)K_S^{-\half}(y)+ K_S^\half(x)\underline \partial_S(y),~\underline \partial_S^{op}(xy)=\underline \partial_S^{op}(x)K_S^\half(y)+ K_S^{-\half}(x)\underline \partial_U^{op}(y)\ .$$
\end{lemma}

Using this we can refine \eqref{eq:abstract_Feigin via left action}, \eqref{eq:abstract_Feigin via right action} as follows.
\begin{proposition}
\label{pr:primitive actions proofs} In the assumptions of Theorem \ref{th:Generalized Feigin homomorphism},  
for any sequence $\ii=(i_1,\ldots,i_m)\in I^m$ and any homogeneous $x\in  \cH^*(\cC)$ of degree $\gamma$ one has:
\begin{equation}
\label{eq:concrete_Feigin via left action}
\Psi_{\cC,\ii}(x)=\sum \underline \partial_{S_{i_1}}^{[a_1]}\cdots \underline \partial_{S_{i_m}}^{[a_m]}(x)\cdot t^{\bf a}=
\sum (\underline \partial^{op}_{S_{i_m}})^{[a_m]}\cdots  (\underline \partial^{op}_{S_{i_1}})^{[a_1]}(x) \cdot t^{\bf a}\ ,
\end{equation}
where the summation is over all ${\bf a}=(a_1,\ldots,\gamma_m)\in \ZZ_{\ge 0}^m$ such that 
$a_1|S_{i_1}|+\cdots+a_m|S_{i_m}|=\gamma$.
(where we abbreviated $X_i^{[\ell]}:=|End(S_i)|^{\frac{\ell(\ell-1)}{4}}X_i^{(\ell)}$).
\end{proposition}

\begin{proof} Indeed, taking into account that 
$$(K_U^\half \partial_U)^a=\chi(|U|,|U|)^{-\frac{a(a-1)}{4}}K_U^\frac{a}{2}\partial_U^a,~(K_U^\half \partial^{op}_U)^a=\chi(|U|,|U|)^{-\frac{a(a-1)}{4}}K_U^\frac{a}{2}(\partial_U^{op})^a$$ 
for all $U\in \cC$, $a\ge 0$, we obtain after repeatedly applying \eqref{eq:K_U E_U commutation}:
$$\underline \partial_{S_{i_1}}^{[a_1]}\cdots \underline \partial_{S_{i_m}}^{[a_m]}
=K_{S_{i_1}}^{\frac{a_1}{2}}\partial_{S_{i_1}}^{(a_1)}\cdots K_{S_{i_m}}^{\frac{a_m}{2}}\partial_{S_{i_m}}^{(a_m)}=
\chi\cdot K_U\partial_{S_{i_1}}^{(a_1)}\cdots \partial_{S_{i_m}}^{(a_m)}\ ,$$
$$(\underline \partial_{S_{i_m}}^{op})^{[a_m]}\cdots (\underline \partial_{S_{i_1}}^{op})^{[a_1]}
=K_{S_{i_m}}^{\frac{a_m}{2}}(\partial_{S_{i_m}}^{op})^{(a_m)}\cdots K_{S_{i_1}}^{\frac{a_1}{2}}(\partial_{S_{i_1}}^{op})^{(a_1)}=
\chi\cdot K_U(\partial_{S_{i_m}}^{op})^{(a_m)}\cdots (\partial_{S_{i_1}}^{op})^{(a_1)}\ ,$$
where $U=a_1S_{i_1}\oplus \cdots \oplus a_mS_{i_m}$ and $\chi=\prod_{1\le k<\ell\le m} \chi(|S_{i_k}|,|S_{i_\ell}|)^{\frac{-a_ka_\ell}{2}}$. Therefore, 
$$\underline \partial_{S_{i_1}}^{[a_1]}\cdots \underline \partial_{S_{i_m}}^{[a_m]}(x)=\chi\cdot \partial_{S_{i_1}}^{(a_1)}\cdots \partial_{S_{i_m}}^{(a_m)}(x),~(\underline \partial_{S_{i_m}}^{op})^{[a_m]}\cdots (\underline \partial_{S_{i_1}}^{op})^{[a_1]}(x)
=\chi\cdot (\partial_{S_{i_m}}^{op})^{(a_m)}\cdots (\partial_{S_{i_1}}^{op})^{(a_1)}(x)$$
for all $x\in \cH^*$ of degree $\gamma=a_1|S_{i_1}|+\cdots+a_m|S_{i_m}|$.

Note that \eqref{eq:Pii} in the form $t_\ell t_k=\chi(|S_{i_k}|,|S_{i_\ell}|)t_kt_\ell$ for $k<\ell$ and \eqref{eq:based quantum torus} imply that
$$t^{\bf a}= \prod_{1\le k<\ell\le m} \chi(|S_{i_k}|,|S_{i_\ell}|)^{\frac{a_ka_\ell}{2}}t_1^{a_1}\cdots t_m^{a_m}$$
for any ${\bf a}=(a_1,\ldots,a_m)\in \ZZ^m$.
Therefore, in the notation \eqref{eq:Feigin explicit via pairing}, we have:
$$x\left([S_{i_1}]^{(a_1)}\cdots [S_{i_m}]^{(a_m)}\right)t_1^{a_1}\cdots t_m^{a_m}=\partial_{S_{i_1}}^{(a_1)}\cdots \partial_{S_{i_m}}^{(a_m)}(x)\cdot t_1^{a_1}\cdots t_m^{a_m}=\underline \partial_{S_{i_1}}^{[a_1]}\cdots \underline \partial_{S_{i_m}}^{[a_m]}(x)\cdot t^{\bf a}$$
$$=(\partial_{S_{i_m}}^{op})^{(a_m)}\cdots (\partial_{S_{i_1}}^{op})^{(a_1)}(x)\cdot t_1^{a_1}\cdots t_m^{a_m}=(\underline \partial_{S_{i_m}}^{op})^{[a_m]}\cdots  (\underline \partial_{S_{i_1}}^{op})^{[a_1]}(x) \cdot t^{\bf a} \ .$$
This verifies \eqref{eq:concrete_Feigin via left action}. Proposition \ref{pr:primitive actions proofs}  is proved.

\end{proof}

\section{Special compatible pairs}
\label{sec:Special compatible pairs}

Following and generalizing \cite[Section 8]{bz5}, here we construct compatible pairs $(\Lambda_\ii,\tilde B_\ii)$ for all sequences $\ii\in I^m$ and certain symmetrizable $I\times I$ matrices.  These will later serve as the exchange data for the  initial seed in ${\mathcal L}_\ii$.  

Let $A=(a_{ij})$ be a symmetrizable $I\times I$ integer  matrix such that $a_{ii}=2$ for $i\in I$ and let  $D=diag(d_i,i\in I)$ be a diagonal matrix with all $d_i\in \ZZ_{>0}$ such that  $C=DA=(d_ia_{ij})=(c_{ij})$ is symmetric.  Denote by $\cQ^\vee$ the lattice with the free basis $\alpha_i^\vee$, $i\in I$.  For each $i\in I$ define $\alpha_i:=d_i\alpha_i^\vee$ and denote by $\cQ$ the sublattice of $\cQ^\vee$ (freely) spanned by $\alpha_i$, $i\in I$.

Let $\cP:=\cQ\oplus\Hom(\cQ^\vee,\ZZ)$ be the corresponding weight lattice.  By definition, $\cP$ has a free basis $\{\alpha_i,\omega_i\,|\,i\in I\}$, where $\omega_i:\cQ^\vee\to \ZZ$ determined by $\omega_i(\alpha_j^\vee)=\delta_{ij}$.  We extend the canonical evaluation pairing $\cQ^\vee\times\Hom(\cQ^\vee,\ZZ)\to \ZZ$ to the pairing $(\cdot,\cdot):\cQ^\vee\times\cP\to\ZZ$ by 
\begin{equation}
\label{eq:extended pairing}
(\alpha_i^\vee,\alpha_j)=a_{ij}\quad\text{for $i,j\in I$.}
\end{equation} 

In particular, since $\cQ\subset \cQ^\vee$, one has a paring $\cQ\times\cP\to\ZZ$ whose restriction to $\cQ\times \cQ$ is symmetric.  Note that modulo the right kernel of $(\cdot,\cdot)$ we may identify $\alpha_j$, $j\in I$ with $\sum\limits_{i\in I}a_{ij}\omega_i$.

For $i\in I$ define a linear endomorphism $s_i$ of $\cP$ by
$$s_i\alpha_j=\alpha_j-a_{ij}\alpha_i\quad\text{and}\quad s_i\omega_j=\omega_j-\delta_{ij}\alpha_i\quad\text{for $j\in I$.}$$

Denote by $W$ the group of automorphisms of $\cP$ generated by $s_i$, $i\in I$. Clearly, $s_i^2=1$ for all $i\in I$ and $W(\cQ)\subset \cQ$.  The following result is easy to check.

\begin{lemma} 
For any symmetrizable $I\times I$ matrix $A$ as above one has:

(a)  There is a unique action of $W$ on $\cQ^\vee$ such that:
$$s_i\alpha_j^\vee=\alpha_j^\vee-a_{ji}\alpha_i^\vee\quad\text{for all $i,j\in I$.}$$ 

(b) The pairing \eqref{eq:extended pairing} is $W$-invariant.

\end{lemma}

To each sequence $\ii=(i_1,\ldots,i_m)\in I^m$ we assign a sequence $w_k\in W$, $k=1,\ldots, m$ by $w_k=s_{i_1}s_{i_2}\cdots s_{i_k}$ (with the convention that $w_0=1$).  

\begin{proposition}
\label{le:mon_change2}
For each $\ii\in I^m$ the inverse of $\varphi_\ii:\ZZ^m\to\ZZ^m$ defined in \eqref{eq:phi} is  given by
\begin{equation}
\label{eq:phi_inv}
\varphi_\ii^{-1}(\varepsilon_\ell)=-\sum\limits_{k=1}^\ell(w_k\alpha_{i_k}^\vee,w_\ell\omega_{i_\ell})\varepsilon_k
\end{equation}
for $\ell=1,\ldots,m$.
\end{proposition}

\begin{proof} 
Define the $\ii$-derivative and $\ii$-integral $\ZZ$-linear maps $\partial_\ii,\int_\ii:\ZZ^m\to\ZZ^m$ respectively by 
\begin{equation}
\label{eq:derint}
\partial_\ii\varepsilon_k=\varepsilon_k-\varepsilon_{k^-}\quad\text{and}\quad{\textstyle\int_\ii}\varepsilon_k=\varepsilon_k+\varepsilon_{k^-}+\cdots \ .
\end{equation}
for $1\le k \le m$, where  $k^-$ is defined in \eqref{eq:kpm}. Clearly, these maps are mutually inverse.  We need the following fact.
\begin{lemma} 
For any sequence $\ii\in I^m$ one has:
\begin{equation}
\label{eq:derint1}
{\textstyle\int_\ii}\varphi_\ii(\varepsilon_k)=-\varepsilon_k-\sum\limits_{\ell=1}^{k-1}a_{i_\ell i_k}\varepsilon_\ell
\end{equation}
for $k=1,\ldots,m$,
\begin{equation}
\label{eq:derint2}
\varphi_\ii^{-1}(\partial_\ii \varepsilon_\ell)=-\varepsilon_\ell-\sum\limits_{k=1}^{\ell-1} (w_k\alpha_{i_k}^\vee,w_\ell\alpha_{i_\ell})\varepsilon_k
\end{equation}
for $k=1,\ldots,m$.
\end{lemma}
 
\begin{proof} We obtain \eqref{eq:derint1} by a direct computation:
$${\textstyle\int_\ii}\varphi_\ii(\varepsilon_k)=-(\varepsilon_k+\varepsilon_{k^-}+\cdots)-(\varepsilon_{k^-}+\varepsilon_{k^{--}}\cdots)-\sum\limits_{\ell:\ell<k<\ell^+}a_{i_\ell i_k}(\varepsilon_\ell+\varepsilon_{\ell^-}+\cdots)$$
$$=-\varepsilon_k-2\varepsilon_{k^-}-2\varepsilon_{k^{--}}-\cdots-\sum\limits_{\ell:\ell<k,i_\ell\ne i_k}a_{i_\ell i_k}\varepsilon_\ell=-\varepsilon_k-\sum\limits_{\ell:\ell<k}a_{i_\ell i_k}\varepsilon_\ell\ .
$$

To prove \eqref{eq:derint2} apply $\int_\ii\varphi_\ii$ to the right hand side of \eqref{eq:derint2} for $\ell:=p$, we obtain for all $p\in [1,m]$:
$${\textstyle\int_\ii}\varphi_\ii(-\varepsilon_p-\sum\limits_{k=1}^{p-1} (w_k\alpha_{i_k}^\vee,w_p\alpha_{i_p})\varepsilon_k)=-{\textstyle\int_\ii}\varphi_\ii(\varepsilon_p)-\sum\limits_{k=1}^{p-1} (w_k\alpha_{i_k}^\vee,w_p\alpha_{i_p}){\textstyle\int_\ii}\varphi_\ii(\varepsilon_k)$$
$$=\varepsilon_p+\sum\limits_{\ell=1}^{p-1}a_{i_\ell i_p}\varepsilon_\ell+\sum\limits_{k=1}^{p-1}(w_k\alpha_{i_k}^\vee,w_p\alpha_{i_p})\varepsilon_k+\sum\limits_{k=1}^{p-1}\sum\limits_{\ell=1}^{k-1} (w_k(a_{i_\ell,i_k}\alpha_{i_k}^\vee),w_p\alpha_{i_p})\varepsilon_\ell$$
$$=\varepsilon_p+\sum\limits_{\ell=1}^{p-1}a_{i_\ell i_p}\varepsilon_\ell+\sum\limits_{\ell=1}^{p-1}(w_\ell\alpha_{i_\ell}^\vee,w_p\alpha_{i_p})\varepsilon_\ell+\sum\limits_{\ell=1}^{p-2}\sum\limits_{k=\ell+1}^{p-1} (w_k\alpha_{i_\ell}^\vee-w_{k-1}\alpha_{i_\ell}^\vee,w_p\alpha_{i_p})\varepsilon_\ell$$
$$=\varepsilon_p+\sum\limits_{\ell=1}^{p-1}a_{i_\ell i_p}\varepsilon_\ell+\sum\limits_{\ell=1}^{p-1}(w_\ell\alpha_{i_\ell}^\vee,w_p\alpha_{i_p})\varepsilon_\ell+\sum\limits_{\ell=1}^{p-2} (w_{p-1}\alpha_{i_\ell}^\vee-w_\ell\alpha_{i_\ell}^\vee,w_p\alpha_{i_p})\varepsilon_\ell=\varepsilon_p\ ,$$
where we used $w_k(a_{i_\ell,i_k}\alpha_{i_k}^\vee)=w_k\alpha_{i_\ell}^\vee-w_{k-1}\alpha_{i_\ell}^\vee$ in the third equality and $(w_{p-1}\alpha_{i_\ell}^\vee,w_p\alpha_{i_p})=-(\alpha_{i_\ell}^\vee,\alpha_{i_p})=-a_{i_\ell,i_p}$ for the last equality.  The lemma is proved.

\end{proof}

Denote by $\varphi_\ii'$ the $\ZZ$-linear map given by \eqref{eq:phi_inv}. To prove \eqref{eq:phi_inv}, it suffices to check that $\varphi_\ii'\partial_\ii$ is given by \eqref{eq:derint2}. Indeed, 
$$\varphi_\ii'(\partial_\ii\varepsilon_\ell)=\varphi_\ii'(\varepsilon_\ell-\varepsilon_{\ell^-})=-\sum\limits_{k=1}^\ell(w_k\alpha_{i_k}^\vee,w_\ell\omega_{i_\ell})\varepsilon_k+\sum\limits_{k=1}^{\ell^-}(w_k\alpha_{i_k}^\vee,w_{\ell^-}\omega_{i_\ell})\varepsilon_k $$  $$=-\varepsilon_\ell-\sum\limits_{k=1}^{\ell-1}(w_k\alpha_{i_k}^\vee,w_\ell\omega_{i_\ell}-w_{\ell^-}\omega_{i_\ell})\varepsilon_k=-\varepsilon_\ell-\sum\limits_{k=1}^{\ell-1}(w_k\alpha_{i_k}^\vee,w_\ell\alpha_{i_\ell})\varepsilon_k\ ,
$$
where the second to last equality used $(w_k\alpha_{i_k}^\vee,w_{\ell^-}\omega_{i_\ell})=(\alpha_{i_k}^\vee,\omega_{i_\ell})=0$ for $\ell^-<k\le\ell$ and the last equality used the identities  $w_{\ell^-}\omega_{i_\ell}=w_{\ell-1}\omega_{i_\ell}$,  $w_\ell\omega_{i_\ell}-w_{\ell-1}\omega_{i_\ell}=w_\ell\alpha_{i_\ell}$.  The proposition is proved. 
  
\end{proof}

Define a skew-symmetric bilinear form $\Lambda_\ii(\cdot,\cdot):\ZZ^m\times\ZZ^m\to\ZZ$ by
\begin{equation}
\label{eq:Lambdaii general}
\Lambda_\ii(\varepsilon_k,\varepsilon_\ell)=\sgn(k-\ell)c_{i_k,i_\ell}\ .
\end{equation}

Denote $\ex=\ex_\ii:=\{k\in [1,m]\,|\, k^+\le m\}$ and call it the set of \emph{exchangeable indices} of $\ii$.  Following \cite[Section 8]{bz5}, for $k\in\ex$ define the vector $\bfb^k=\bfb^k(\ii)=\sum\limits_{p=1}^m b_{pk}\varepsilon_p\in \ZZ^m$ by:
\begin{equation}
\label{eq:Bii}
b_{pk}=\begin{cases} 
-1 & \text{if $p=k^-$}\\ 
-a_{i_pi_k} & \text{if $p<k<p^+<k^+$}\\ 
a_{i_pi_k} & \text{if $k<p<k^+<p^+$}\\ 
1 & \text{if $p=k^+$}\\ 
0 & \text{otherwise}
\end{cases}.
\end{equation}
Denote by $\tilde B_\ii$ the $m\times \ex$ matrix with columns ${\bf b}^k$, $k\in \ex$.

Our objective in this section is to prove the following compatibility condition which refines and generalizes \cite[Theorem 8.3]{bz5} for single words $\ii$ (the double word version can be stated verbatim).

\begin{theorem}
\label{th:compatibility}
For any $k\in\ex$ and $1\le\ell\le m$, one has
\begin{equation}
\label{eq:compatibility}
\Lambda_\ii(\varphi_{\ii}^{-1}({\bf b}^k),\varphi_\ii^{-1}(\varepsilon_\ell))=2d_{i_k}\delta_{k\ell}.
\end{equation}
\end{theorem}

\begin{proof} The  proof is rather technical computational so we have split all computations into the series of more manageable results (Lemmas \ref{le:intb0}, \ref{le:form versus map}, \ref{le:Phiii explicit}  and Proposition \ref{prop:pairing}).

\begin{lemma} 
\label{le:intb0}
For any sequence $\ii\in I^m$ and $k\in\ex$ one has:
\begin{equation}
\label{eq:intb} 
{\textstyle\int_\ii}\bfb^k=\varepsilon_k+\varepsilon_{k^+}+\sum\limits_{p=k+1}^{k^+-1}a_{i_pi_k}\varepsilon_p\ .
\end{equation}
\end{lemma}

\begin{proof} By definition, for any ${\bf a}=\sum_{p=1}^m a_p\varepsilon_p$, one has 
$${\textstyle\int_\ii}\bfa=\sum\limits_{p=1}^m a_p(\varepsilon_p+\varepsilon_{p^-}+\cdots)=\sum\limits_{p=1}^m (a_p+a_{p^+}+\cdots)\varepsilon_p\ .$$
The result follows easily from this and the characterization of $b_{pk}$ from \cite[Remark 8.8]{bz5}:
$$b_{pk}=s_{pk}-s_{p,k^+}-s_{p^+,k}+s_{p^+,k^+}$$
for $p\in [1,m]$ and $k\in \ex$, where $s_{pk}=\frac{1}{2}\sgn(p-k)a_{i_p,i_k}$ (with the convention $s_{p^+,k}=a_{i_p,i_k}$ if $p^+=m+1$). Indeed, we obtain:
$${\textstyle\int_\ii}\bfb^k=\sum\limits_{p=1}^m (b_{pk}+b_{p^+,k}+\cdots)\varepsilon_p=\sum\limits_{p=1}^m (s_{pk}-s_{p,k^+})\varepsilon_p=\varepsilon_k+\varepsilon_{k^+}+\sum\limits_{p=k+1}^{k^+-1}a_{i_pi_k}\varepsilon_p$$
because 
$s_{pk}-s_{p^+,k}=\frac{1}{2}(\sgn(p-k)-\sgn(p-k^+))a_{i_p,i_k}=
\begin{cases} 
1 & \text{if $p=k$ or $p=k^+$}\\ 
a_{i_p,i_k} & \text{if $k<p<k^+$}
\\ 0 & \text{otherwise}
\end{cases}.$
\end{proof}

\begin{proposition}
\label{prop:pairing} 
For any $\ii\in I^m$ and $1\le k,\ell\le m$ one has:
\begin{equation}
\label{eq:pairing} 
\Lambda_\ii(\varphi_\ii^{-1}(\partial_\ii\varepsilon_k),\varphi_\ii^{-1}(\partial_\ii\varepsilon_\ell))=\sgn(\ell-k)(w_k\alpha_{i_k},w_\ell\alpha_{i_\ell})
\end{equation}
\begin{equation}
\label{eq:pairing phi} 
\Lambda_\ii(\varphi_\ii^{-1}(\varepsilon_k),\varphi_\ii^{-1}(\varepsilon_\ell))=
\begin{cases}
(w_k\omega_{i_k}-\omega_{i_k},\omega_{i_\ell}+w_\ell\omega_{i_\ell}) & \text{if $k\le \ell$}\\
(\omega_{i_\ell}-w_\ell\omega_{i_\ell},w_k\omega_{i_k}+\omega_{i_k}) & \text{if $k>\ell$}\\
\end{cases}
\end{equation}

\end{proposition}

\begin{proof} To prove \eqref{eq:pairing} we need the following easy but powerful fact. 
\begin{lemma}
\label{le:form versus map}
Let $\Phi$ be a bilinear form on $\kk^m$, $\{\varepsilon_1,\ldots, \varepsilon_m\}$ be a basis of $\kk^m$, and $c_1,\ldots,c_m\in \kk\setminus\{0\}$. Define  a linear map $\varphi:\kk^m\to \kk^m$ by:
$$\varphi(\varepsilon_k)=\sum_{\ell=1}^m c_\ell \Phi(\varepsilon_k,\varepsilon_\ell)\varepsilon_\ell\ .$$
If $\varphi$ is invertible, then:
$$\varphi^{-1}(\varepsilon_\ell)=\sum_{k=1}^m c_k \Phi(\varphi^{-1}(\varepsilon_k),\varphi^{-1}(\varepsilon_\ell))\varepsilon_k\ .$$
\end{lemma}

\begin{proof} Since $\varphi$ is invertible we have
$$\varphi(\varphi^{-1}(\varepsilon_k))=\sum_{\ell=1}^m c_\ell \Phi(\varphi^{-1}(\varepsilon_k),\varepsilon_\ell)\varepsilon_\ell=\varepsilon_k\ .$$
In particular, this implies $\Phi(\varphi^{-1}(\varepsilon_k),\varepsilon_\ell)=c_k^{-1}\delta_{k\ell}$.  Now composing $\varphi$ with the map 
$$\varphi'(\varepsilon_\ell):=\sum_{k=1}^m c_k \Phi(\varphi^{-1}(\varepsilon_k),\varphi^{-1}(\varepsilon_\ell))\varepsilon_k$$
gives
$$\varphi'(\varphi(\varepsilon_\ell))=\sum_{k=1}^m c_k \Phi(\varphi^{-1}(\varepsilon_k),\varepsilon_\ell)\varepsilon_k=\varepsilon_\ell \ .$$
Thus $\varphi'=\varphi^{-1}$.  The lemma is proved.

\end{proof}

We apply Lemma \ref{le:form versus map} with $\varphi:=\int_\ii\varphi_\ii$ and $c_\ell:=d_{i_\ell}^{-1}$ for $\ell=1,\ldots,m$.  That is, by \eqref{eq:derint1} and \eqref{eq:derint2}, the bilinear form $\Phi_\ii$ on $\ZZ^m$  defined by
\begin{equation}
\label{eq:another Phi def} 
\Phi_\ii(\varepsilon_k,\varepsilon_\ell)=
\begin{cases} -d_{i_k} & \text{if $k=\ell$}\\
-c_{i_\ell,i_k} & \text{if $\ell<k$}\\
 0 & \text{otherwise}\\
\end{cases}
\end{equation}
satisfies 
\begin{equation}
\label{eq:another Phi}
\Phi_\ii(\varphi_\ii^{-1}(\partial_\ii \varepsilon_k),\varphi_\ii^{-1}(\partial_\ii \varepsilon_\ell)) 
=\begin{cases} 
-d_{i_k} & \text{if $k=\ell$}\\
-(w_k\alpha_{i_k},w_\ell\alpha_{i_\ell}) & \text{if $k<\ell$}\\
0 & \text{otherwise}\\
\end{cases}.
\end{equation}

The identity \eqref{eq:pairing} follows from this and the fact that $\Lambda_\ii$ is a skew-symmetrization of $\Phi_\ii$, i.e., 
\begin{equation}
\label{eq:Lambda Phi}
\Lambda_\ii({\bf a},{\bf b})=\Phi_\ii({\bf b},{\bf a})-\Phi_\ii({\bf a},{\bf b})
\end{equation}
for all ${\bf a},{\bf b}\in \ZZ^m$.  

To prove \eqref{eq:pairing phi}, we need the following result. 

\begin{lemma} 
\label{le:Phiii explicit}
The form $\Phi_\ii$ given by \eqref{eq:another Phi def} satisfies:
$$\Phi_\bfi(\varphi_\bfi^{-1}(\varepsilon_k),\varphi_\bfi^{-1}(\varepsilon_\ell))=
\begin{cases} 
(\omega_{i_k}-w_k\omega_{i_k},w_\ell\omega_{i_\ell}) & \text{if $k<\ell$}\\
(w_\ell\omega_{i_\ell}-\omega_{i_\ell},\omega_{i_k}) & \text{if $k\ge \ell$}
\end{cases}.$$
\end{lemma}

\begin{proof} Indeed, let us compute using  \eqref{eq:another Phi}:
$$\Phi_\bfi(\varphi_\bfi^{-1}(\varepsilon_k),\varphi_\bfi^{-1}(\partial_\bfi \varepsilon_\ell))=\Phi_\bfi(\varphi_\bfi^{-1}(\partial_\bfi{\textstyle\int_\bfi}\varepsilon_k),\varphi_\bfi^{-1}(\partial_\bfi \varepsilon_\ell))=\Phi_\bfi\big(\varphi_\bfi^{-1}(\partial_\bfi(\varepsilon_k+\varepsilon_{k^-}+\cdots)),\varphi_\bfi^{-1}(\partial_\bfi\varepsilon_\ell)\big)$$
$$=\sum\limits_{\substack{p\le\min(k,\ell):\\i_p=i_k}}
\Phi_\bfi\big(\varphi_\bfi^{-1}(\partial_\bfi\varepsilon_p),\varphi_\bfi^{-1}(\partial_\bfi\varepsilon_\ell)\big)=-\varepsilon_{k,\ell}d_{i_k}-\sum\limits_{\substack{p\le\min(k,\ell-1):\\i_p=i_k}}
(w_p\alpha_{i_k},w_\ell\alpha_{i_\ell})$$
$$=-\varepsilon_{k,\ell}d_{i_k}-\sum\limits_{\substack{p\le \min(k,\ell-1):\\i_{k'}=i_k}} (w_p\omega_{i_k}-w_{p^-}\omega_{i_k},w_\ell\alpha_{i_\ell})=-\varepsilon_{k,\ell}d_{i_k}+
(\omega_{i_k}-w_{k_0}\omega_{i_k},w_\ell\alpha_{i_\ell})\ ,$$
where $\varepsilon_{k,\ell}=1$ if $k\ge \ell$, $i_k=i_\ell$ and $0$ otherwise and $k_0=\max\{p\le \min(k,\ell-1)\st i_p=i_k\}$. 

Note that if $k\ge \ell$, then $w_{k_0}\omega_{i_k}=w_{\ell-1}\omega_{i_k}$ and so $-\varepsilon_{k,\ell}d_{i_k}+(\omega_{i_k}-w_{k_0}\omega_{i_k},w_\ell\alpha_{i_\ell})=(w_\ell\alpha_{i_\ell},\omega_{i_k})$. Therefore, 
$$\Phi_\bfi(\varphi_\bfi^{-1}(\varepsilon_k),\varphi_\bfi^{-1}(\partial_\bfi \varepsilon_\ell))=(w_\ell\alpha_{i_\ell},\omega_{i_k})-\begin{cases} 
(w_\ell\alpha_{i_\ell},w_k\omega_{i_k}) & \text{if $k<\ell$}\\
0 & \text{if $k\ge \ell$}
\end{cases}.
$$
Using this, we obtain:
$$\Phi_\bfi(\varphi_\bfi^{-1}(\varepsilon_k),\varphi_\bfi^{-1}(\varepsilon_\ell))=\Phi_\bfi(\varphi_\bfi^{-1}(\varepsilon_k),\varphi_\bfi^{-1}(\partial_\bfi{\textstyle\int_\bfi} \varepsilon_\ell))=\Phi_\bfi\big(\varphi_\bfi^{-1}(\varepsilon_k),\varphi_\bfi^{-1}(\partial_\bfi(\varepsilon_\ell+\varepsilon_{\ell^-}+\cdots))\big)$$
$$=\sum\limits_{\substack{p\le\ell:\\i_p=i_\ell}}
\Phi_\bfi\big(\varphi_\bfi^{-1}(\varepsilon_k),\varphi_\bfi^{-1}(\partial_\bfi\varepsilon_p)\big)
=\sum\limits_{\substack{p\le\ell:\\i_p=i_\ell}} \left( (w_p\alpha_{i_\ell},\omega_{i_k})-\begin{cases} 
(w_p\alpha_{i_\ell},w_k\omega_{i_k}) & \text{if $k<p$}\\
0 & \text{if $k\ge p$}
\end{cases} \right)$$
$$=\sum\limits_{\substack{p\le\ell:\\i_p=i_\ell}}
(w_p\alpha_{i_\ell},\omega_{i_k})-\begin{cases} 
\sum\limits_{\substack{k<p\le\ell:\\i_p=i_\ell}}(w_p\alpha_{i_\ell},w_k\omega_{i_k}) & \text{if $k<\ell$}\\
0 & \text{if $k\ge \ell$}
\end{cases}$$
$$=(w_\ell\omega_{i_\ell}-\omega_{i_\ell},\omega_{i_k})-\begin{cases} 
(w_\ell\omega_{i_\ell}-w_{\ell_0}\omega_{i_\ell},w_k\omega_{i_k}) & \text{if $k<\ell$}\\
0 & \text{if $k\ge \ell$}
\end{cases}=\begin{cases} 
(\omega_{i_k}-w_k\omega_{i_k},w_\ell\omega_{i_\ell}) & \text{if $k<\ell$}\\
(w_\ell\omega_{i_\ell}-\omega_{i_\ell},\omega_{i_k}) & \text{if $k\ge \ell$}
\end{cases}$$
where $\ell_0=\max\{p\le \min(k,\ell)\st i_p=i_\ell\}$. Here we used the identities $w_p\alpha_{i_\ell}=w_p\omega_{i_\ell}-w_{p^-}\omega_{i_\ell}$ whenever $i_p=i_\ell$ and $w_{\ell_0}\omega_{i_\ell}=w_k\omega_{i_\ell}$.  The lemma is proved.

\end{proof}

Finally, using \eqref{eq:Lambda Phi} and Lemma \ref{le:Phiii explicit}, we obtain for all $k,\ell\in [1,m]$:
$$\Lambda_\ii\big(\varphi_\bfi^{-1}(\varepsilon_k),\varphi_\bfi^{-1}(\varepsilon_\ell)\big)=\Phi_\ii\big(\varphi_\bfi^{-1}(\varepsilon_\ell),\varphi_\bfi^{-1}(\varepsilon_k)\big)-\Phi_\ii\big(\varphi_\bfi^{-1}(\varepsilon_k),\varphi_\bfi^{-1}(\varepsilon_\ell)\big)$$
$$=\begin{cases} 
(\omega_{i_\ell}-w_\ell\omega_{i_\ell},w_k\omega_{i_k}) & \text{if $\ell<k$}\\
(w_k\omega_{i_k}-\omega_{i_k},\omega_{i_\ell}) & \text{if $\ell\ge k$}
\end{cases}-\begin{cases} 
(\omega_{i_k}-w_k\omega_{i_k},w_\ell\omega_{i_\ell}) & \text{if $k<\ell$}\\
(w_\ell\omega_{i_\ell}-\omega_{i_\ell},\omega_{i_k}) & \text{if $k\ge \ell$}
\end{cases}.$$

This proves \eqref{eq:pairing phi} and thus completes the proof of Proposition~\ref{prop:pairing}.
\end{proof}

To prove the compatibility conditions \eqref{eq:compatibility},  we compute using \eqref{eq:intb} and \eqref{eq:pairing}:
$$\Lambda_\ii(\varphi_\ii^{-1}(\bfb^k),\varphi_\ii^{-1}(\partial_\ii\varepsilon_\ell))=\Lambda_\ii\left(\varphi_\ii^{-1}(\partial_\ii{\textstyle\int_\ii}\bfb^k),\varphi_\ii^{-1}(\partial_\ii\varepsilon_\ell)\right)$$
$$=\Lambda_\ii\left(\varphi_\ii^{-1}(\partial_\ii\varepsilon_k+\partial_\ii\varepsilon_{k^+}+\sum\limits_{p:k<p<k^+}a_{i_pi_k}\partial_\ii\varepsilon_p),\varphi_\ii^{-1}(\partial_\ii\varepsilon_\ell)\right)$$
$$=\Lambda_\ii\left(\varphi_\ii^{-1}(\partial_\ii\varepsilon_k),\varphi_\ii^{-1}(\partial_\ii\varepsilon_\ell))+\Lambda_\ii(\varphi_\ii^{-1}(\partial_\ii\varepsilon_{k^+}),\varphi_\ii^{-1}(\partial_\ii\varepsilon_\ell)\right)+\sum\limits_{p=k+1}^{k^+-1}a_{i_pi_k}\Lambda_\ii(\varphi_\ii^{-1}(\partial_\ii\varepsilon_p),\varphi_\ii^{-1}(\partial_\ii\varepsilon_\ell))$$
$$=\left(\sgn(\ell-k)w_k\alpha_{i_k}+\sgn(\ell-k^+)w_{k^+}\alpha_{i_k}+\sum\limits_{p=k+1}^{k^+-1}\sgn(\ell-p)a_{i_pi_k}w_p\alpha_{i_p},w_\ell\alpha_{i_\ell}\right)\ .$$

The relation $a_{i_p i_k}\alpha_{i_p}=\alpha_{i_k}-s_{i_p}\alpha_{i_k}$ implies $a_{i_pi_k}w_p\alpha_{i_p}=w_p\alpha_{i_k}-w_{p-1}\alpha_{i_k}$, hence the sum above telescopes and we obtain:
$$\sum\limits_{p=k+1}^{k^+-1}\sgn(\ell-p)a_{i_pi_k}w_p\alpha_{i_p}=
\begin{cases} 
w_k\alpha_{i_k}+w_{k^+}\alpha_{i_k} & \text{if $\ell\le k$}\\ 
-w_k\alpha_{i_k}-w_{k^+}\alpha_{i_k} & \text{if $\ell\ge k^+$}\\ 
-w_k\alpha_{i_k}+w_{\ell-1}\alpha_{i_k}+w_\ell\alpha_{i_k}+w_{k^+}\alpha_{i_k}& \text{if $k<\ell<k^+$}
\end{cases}.$$

Therefore, taking into account that $(w_{\ell-1}\alpha_{i_k}+w_\ell\alpha_{i_k},w_\ell\alpha_{i_\ell})=(\alpha_{i_k},s_{i_\ell}\alpha_{i_\ell}+\alpha_{i_\ell})=0$, we obtain:
$$\Lambda_\ii(\varphi_\ii^{-1}({\bf b}^k),\varphi_\ii^{-1}(\partial_\ii\varepsilon_\ell))
=\begin{cases} 
\delta_{k\ell}(w_k\alpha_{i_k},w_\ell\alpha_{i_\ell}) & \text{if $\ell\le k$}\\
-\delta_{k^+,\ell}(w_{k^+}\alpha_{i_k},w_\ell\alpha_{i_\ell}) & \text{if $\ell\ge k^+$}\\
0 & \text{if $k<\ell<k^+$}\\
\end{cases}
=2d_{i_k}(\delta_{k\ell}-\delta_{k^+,\ell})\ .$$

Using this, we obtain the desired result:
$$\Lambda_\ii(\varphi_\ii^{-1}({\bf b}^k),\varphi_\ii^{-1}(\varepsilon_\ell))=\Lambda_\ii(\varphi_\ii^{-1}({\bf b}^k),\varphi_\ii^{-1}(\partial_\ii\varepsilon_\ell))+\Lambda_\ii(\varphi_\ii^{-1}({\bf b}^k),\varphi_\ii^{-1}(\partial_\ii\varepsilon_{\ell^-}))+\cdots$$
$$=2d_{i_k}(\delta_{k\ell}-\delta_{k^+,\ell}+\delta_{k,\ell^-}-\delta_{k^+,\ell^-}+\cdots)=2d_{i_k}(\delta_{k\ell}-\delta_{k^+\ell}+\delta_{k^+\ell}-\delta_{k^{++},\ell}+\cdots)=2d_{i_k}\delta_{k\ell}\ .$$
This completes the proof of \eqref{eq:compatibility} and therefore Theorem~\ref{th:compatibility} is proved.
\end{proof}

Generalizing second equation \eqref{eq:phi}, we define $\ZZ$-linear automorphisms $\rho_\ii^+,\rho_\ii^-$ of $\ZZ^m$ by:
\begin{equation}
\label{eq:rhopm}
\rho_\ii^+(\varepsilon_k)=\varepsilon_k+\sum\limits_{\ell<k\st \ell^+=m+1} a_{i_\ell i_k}\varepsilon_\ell,\quad\rho_\ii^-(\varepsilon_k)=\varepsilon_k-\sum\limits_{\ell\le k\st \ell^+=m+1} a_{i_\ell i_k}\varepsilon_\ell
\end{equation}
for $1\le k\le m$. By definition, $\rho_\ii=\rho_\ii^-$. 

\begin{proposition}
\label{pr:compatibility prime} 
Let $\ii\in I^m$ be such that  $\ex_\ii=\{1,\ldots,m-r\}$ for some $r\le |I|$. Then for any $k\in\ex$ and $1\le\ell\le m$, one has:
\begin{equation}
\label{eq:compatibility prime}
\Lambda_\ii(((\rho_\ii^\pm)^{-1}\varphi_\ii)^{-1}( {\bf b}^k_\pm),((\rho_\ii^\pm )^{-1}\varphi_\ii)^{-1}(\varepsilon_\ell))=2d_{i_k}\delta_{k\ell}\ , 
\end{equation}
where ${\bf b}^k_+$ and ${\bf b}^k_-$, $k\in \ex$, are truncations of ${\bf b}^k$ given by 
$${\bf b}^k_\pm=\sum_{p\in {\bf ex}} b_{pk}\varepsilon_p \pm \delta_{k^+}  \varepsilon_{k^+}$$
and
$\delta_{\ell}=
\begin{cases} 
0 & \text{if $\ell\in \ex$}\\
1 & \text{if $\ell\notin \ex$}
\end{cases}$.  
\end{proposition}

\begin{proof} The following fact is obvious. 
\begin{lemma} 
\label{le:rho phi}
Under the hypotheses of Proposition \ref{pr:compatibility prime} one has:  
$$\rho_\ii^+(\varepsilon_k)=\varepsilon_k+\delta_k\sum\limits_{\ell<k:\ell\notin \ex} a_{i_\ell i_k}\varepsilon_\ell,
\quad\rho_\ii^-(\varepsilon_k)=\varepsilon_k-\delta_k\sum\limits_{\ell\le k:\ell\notin \ex} a_{i_\ell i_k}\varepsilon_\ell\ .$$
\end{lemma} 

%
%

We also need the following fact. 

\begin{lemma} 
\label{le:rho of b truncated}
Under the hypotheses of Proposition \ref{pr:compatibility prime} one has $\rho_\ii^\pm(\bfb_\pm^k)={\bf b}^k$ for $k\in \ex$.

\end{lemma}

\begin{proof} 
Indeed, using Lemma \ref{le:rho phi}, for each $k\in \ex$ we have: 
$$\rho_\ii^+({\bf b}^k_+)=\rho_\ii^+({\bf b}^k_+ -\delta_{k^+}\varepsilon_{k^+}+\delta_{k^+}\varepsilon_{k^+})
={\bf b}^k_+-\delta_{k^+}\varepsilon_{k^+}+\delta_{k^+}\rho_\ii^+(\varepsilon_{k^+})={\bf b}^k_++\delta_{k^+}\cdot\sum\limits_{\ell<k^+:\ell\notin \ex} a_{i_\ell i_k}\varepsilon_\ell={\bf b}^k$$
$$\rho_\ii^-({\bf b}^k_-)=\rho_\ii^-({\bf b}^k_-+\delta_{k^+}\varepsilon_{k^+}-\delta_{k^+}\varepsilon_{k^+})
={\bf b}^k_-+\delta_{k^+}\varepsilon_{k^+}-\delta_{k^+}\rho_\ii^-(\varepsilon_{k^+})={\bf b}^k_-+\delta_{k^+}\cdot\sum\limits_{\ell\le k^+:\ell\notin \ex} a_{i_\ell i_k}\varepsilon_\ell={\bf b}^k$$
because, by definition \eqref{eq:Bii},  
$b_{\ell k}=\begin{cases} 
a_{i_\ell i_k} & \text{if $k<\ell<k^+$}\\
1 & \text{if $\ell=k^+$}\\
0 & \text{otherwise}
\end{cases}
$
whenever $\ell\in [1,m]\setminus \ex$. 
\end{proof}

Now we are ready to prove \eqref{eq:compatibility prime}. Indeed,  using Lemmas \ref{le:rho phi} and \ref{le:rho of b truncated}, we obtain for all $k\in \ex$, $\ell\in [1,m]$:
$$\Lambda_\ii(((\rho_\ii^\pm)^{-1}\varphi_\ii)^{-1}( {\bf b}^k_\pm),((\rho_\ii^\pm )^{-1}\varphi_\ii)^{-1}(\varepsilon_\ell))=\Lambda_\ii(\varphi_{\ii}^{-1}\rho_\ii^\pm ({\bf b}^k_\pm),\varphi_\ii^{-1}\rho_\ii^\pm(\varepsilon_\ell))=\Lambda_\ii(\varphi_{\ii}^{-1}({\bf b}^k_\pm),\varphi_\ii^{-1}\rho_\ii(\varepsilon_\ell))$$
$$=\Lambda_\ii(\varphi_{\ii}^{-1}({\bf b}^k),\varphi_\ii^{-1}(\varepsilon_\ell)+\varphi_\ii^{-1}(\rho_\ii^\pm(\varepsilon_\ell)-\varepsilon_\ell))=\Lambda_\ii(\varphi_{\ii}^{-1}({\bf b}^k),\varphi_\ii^{-1}(\varepsilon_\ell))=2d_{i_k}\delta_{k\ell}$$
by Theorem \ref{th:compatibility} because $\Lambda_\ii(\varphi_{\ii}^{-1}({\bf b}^k),\varphi_\ii^{-1}(\varepsilon_{\ell'}))=0$ whenever $\ell'\in [1,m]\setminus \ex$.  The proposition is proved.

\end{proof}

For each sequence $\ii\in I^m$ and $\lambda\in \cP$ define the vector ${\bf a}_\lambda\in \ZZ^m$ by 
\begin{equation}
\label{eq:alambda}
{\bf a}_\lambda:=-\sum_{k=1}^m (w_k\alpha_{i_k}^\vee,w_m\lambda)\varepsilon_k\ .
\end{equation}
This notation is justified by the following collection of easily verifiable facts.

\begin{lemma} 
\label{le:alambda}
For any $\ii\in I^m$ and $\lambda\in \cP$ one has:

(a) ${\bf a}_{\lambda}=\{0\}$ for all $\lambda$ in the right kernel of the pairing  $(\cdot,\cdot):\cQ^\vee\times \cP\to \ZZ$.  

(b) ${\bf a}_{\alpha_j}=\sum\limits_{i\in I} a_{ij} {\bf a}_{\omega_i}$ for all $j\in I$.

(c) The set $L_0:=\{{\bf a}_\lambda\,|\,\lambda\in \cP\}$ is a sub-lattice of $\ZZ^m$ spanned by ${\bf a}_{\omega_i}$, $i\in I$.

(d) For any $w\in W$ the lattice $L_0$ is spanned by ${\bf a}_{w\omega_i}$, $i\in I$.

(e) $\varphi_\ii^{-1}(\varepsilon_\ell)={\bf a}_{\omega_{i_\ell}}$ for each $\ell\in [1,m]\setminus\ex$.

(f) $|{\bf a}_\lambda|=w_m\lambda-\lambda$ for $\lambda\in \cP$, where where we abbreviated $|{\bf a}|:=\sum\limits_{k=1}^m a_k\alpha_{i_k}$. 
\end{lemma}

\begin{proposition} 
\label{pr:commute alambda}
For each $\ii\in I^m$ and ${\bf a}=(a_1,\ldots,a_m)\in \ZZ^m$ one has:
$$\Lambda_\ii({\bf a},{\bf a}_\lambda)=(|{\bf a}|,w_m\lambda+\lambda)\ ,$$
\end{proposition}
\begin{proof}
It suffices to verify the result for $\bfa=\varepsilon_k$:
$$\Lambda_\bfi(\varepsilon_k,\bfa_\lambda)=-\sum\limits_{\ell=1}^m\sgn(k-\ell)c_{i_ki_\ell}(w_\ell\alpha_{i_\ell}^\vee,w_m\lambda)
=-\sum\limits_{\ell=1}^m\sgn(k-\ell)(w_\ell\alpha_{i_k}-w_{\ell-1}\alpha_{i_k},w_m\lambda)$$
$$=(\alpha_{i_k},w_m\lambda)-(w_{k-1}\alpha_{i_k},w_m\lambda)-(w_k\alpha_{i_k},w_m\lambda)+(w_m\alpha_{i_k},w_m\lambda)
=(\alpha_{i_k},w_m\lambda+\lambda)\ .$$
Since $|\varepsilon_k|=\alpha_{i_k}$ this completes the proof.
\end{proof}

\section{Valued Quivers and Proof of Theorem~\ref{th:character=cluster character}}
\label{sec:valued_quivers}
\subsection{Quantum cluster characters} 
\label{subsec:qcluster characters}
Let ${\mathcal C}$ be a finitary Abelian category with finitely many simple objects ${\bf S}=\{S_i\,|\,i\in I\}$ with no self-extensions. 
Denote by $\cK({\mathcal C})$ the Grothendieck group of ${\mathcal C}$.  Clearly, $\cK({\mathcal C})=\bigoplus\limits_{i\in I}\ZZ\cdot |S_i|$, where  $|S_i|\in\cK({\mathcal C})$ is the class of $S_i$.

Note that $\FF_i:=\End_{\mathcal C}(S_i)$ is a finite field for all $i\in I$. 

Define the $I\times I$ Cartan matrix $A=A_{\bf S}=(a_{ij})$ of $\cC$ by
\begin{equation}
\label{eq:cartan categorical}
 a_{ij}=2\delta_{ij}  -  \dim_{\FF_i} (\Ext^1(S_i,S_j)) - \dim_{\FF_i}(\Ext^1(S_j,S_i))\ .
\end{equation}

From now on, we assume that  for each $i\ne j$: $Ext^1(S_i,S_j)=0$ or $Ext^1(S_j,S_i)=0$ and define an $I\times I$ matrix $B=B_{\mathcal C}=(b_{ij})$ by:
\begin{equation}
\label{eq:bij categorical}
b_{ij}=
\begin{cases} 
 \dim_{\FF_i} (\Ext^1(S_i,S_j)) & \text{if $\Ext^1(S_i,S_j)\ne 0$}\\ 
- \dim_{\FF_i} (\Ext^1(S_j,S_i)) & \text{if $\Ext^1(S_j,S_i)\ne 0$}\\ 
  0 & \text{otherwise}\\
\end{cases}.
\end{equation}
Clearly,  $a_{ij}=-|b_{ij}|$ for $i\ne j$.

Following \cite{rupel2},  for $\bfe \in \cK({\mathcal C})$ denote by  ${\bf e}\mapsto {}^*\bfe$ and ${\bf e}\mapsto \bfe^*$ respectively the  endomorphisms of $\cK({\mathcal C})$ given by 
\begin{equation}
\label{eq:star e}
{}^*|S_j|=|S_j|-\sum_{k\in I}[b_{kj}]_+\cdot |S_k|,~|S_j|^*=|S_j|-\sum_{k\in I}[-b_{kj}]_+\cdot |S_k|.
\end{equation}
Define ${\bf b}^j$, $j\in I$ by 
\begin{equation}
\label{eq:bjrep}
{\bf b}^j:=|S_j|^*-{}^*|S_j|
\end{equation} 
Clearly, ${\bf b}^j= \sum_{k\in I}b_{kj}\cdot |S_k|$ for all $j\in J$, so it can be identified with the $j$-th column of $B$.

Fix a subset $\ex$ of $I$ and a unitary bicharacter $\chi=\chi_{\mathcal C}$ on $\cK({\mathcal C})$ such that 
\begin{equation}
\label{eq:categorical compatibility}
\chi({\bf b}^j,|S_i|)=|\kk_i|^{\delta_{ij}}
\end{equation}
for all $i\in I$, $j\in \ex$.

Denote by ${\mathcal T}_{\chi_{\mathcal C}}$ the based quantum torus  with the basis $X^{\bf a}$, ${\bf a}\in \cK({\mathcal C})$ subject to the relation \eqref{eq:based quantum torus}:
$$ X^{\bf a}X^{\bf b}=\chi_{\mathcal C}({\bf a},{\bf b})X^{{\bf a}+{\bf b}}$$
for ${\bf a},{\bf b}\in \cK({\mathcal C})$. 
%
 
Then the quantum cluster character  $X_V\in {\mathcal T}_{\chi_{\mathcal C}}$ of $V\in \cC$ is given by \eqref{eq:character-Gr}:
\begin{equation}
\label{eq:character-Gr refined}
X_V=\sum_{{\bf e}\in \cK({\mathcal C})} \langle |V|,|V|-{\bf e}\rangle^{-1} \cdot |Gr_\bfe(V)| \cdot X^{-{\bf e}^*-{}^*(|V|-{\bf e})}\ .
\end{equation}

\subsection{Valued quivers, flags, and Grassmannians}
Let $\FF$ be a finite field with $q$ elements.  Fix an algebraic closure $\overline{\FF}$ of $\FF$ and for each $a>0$ denote by $\FF_a$ the degree $a$ extension of $\FF$ in $\overline{\FF}$.  Note that the largest subfield of $\overline{\FF}$ contained in both $\FF_a$ and $\FF_b$ is $\FF_{gcd(a,b)}$.  If $a|b$, then  $F_a\subset \FF_b$ and $F_b\cong (\FF_a)^{\frac{b}{a}}$ as a vector space over $\FF_a$.

Fix an integer $n$ and let $Q=\{Q_0, Q_1, s, t\}$ be a quiver with vertices $Q_0=\{1,2,\ldots, n\}$ and arrows $Q_1$, where we denote by $s(\alpha)$ and $t(\alpha)$ the source and target of an arrow $\alpha:s(\alpha)\to t(\alpha)$. Given a map 
$\bfd: Q_0\to \ZZ_{>0}$ ($i\mapsto d_i$), we refer to the pair $(Q,\bfd)$ as a \textit{valued quiver}.

Define a representation $V=(\left\{V_{i}\right\}_{i\in Q_0},\left\{\varphi_{\alpha}\right\}_{\alpha\in Q_1})$ of $(Q,\bfd)$ by assigning an $\FF_{d_i}$-vector space $V_i$ to each vertex $i\in Q_0$ and an $\FF_{gcd(d_{s(\alpha)}, d_{t(\alpha)})}$-linear map $\varphi_\alpha: V_{s(\alpha)}\to V_{t(\alpha)}$ to each arrow $\alpha\in Q_1$.  Morphisms, direct sums, kernels, and co-kernels are defined as in the well-known representation theory of quivers.  Denote by $\rep_{\FF} (Q,\bfd)$ the  (Abelian) category of all finite dimensional representations of $(Q,\bfd)$. 


Following \cite{rupel2} to each skew-symmetrizable $n\times n$ matrix $B$, i.e., such that  $(DB)^T=-B^TD$, for some diagonal matrix $D=diag(d_1,\ldots,d_n)$, $d_i\in \ZZ_{>0}$  we assign  a valued quiver $(Q_B,\bfd)$ having no loops or two-cycles with:

$\bullet$ $Q_0=\{1,\ldots,n\}$, ${\bf d}=(d_1,\ldots,d_n)$;

$\bullet$  $gcd(|b_{ij}|,|b_{ji}|)$ arrows in $Q_1$ from $i$ to $j$ for distinct $i,j\in Q_0$ whenever $b_{ij}>0$.

\medskip
One can easily recover the matrix $B$ from a valued quiver $(Q,\bfd)$ having no loops or two-cycles  by $B=B_Q:=B_{\rep_\FF(\QQ,\bfd)}$ via \eqref{eq:bij categorical}, where ${\bf S}=\{S_1,\ldots,S_m\}$ is the set of all simple representations of $(Q,\bfd)$ (see also \cite{rupel2} for details).  

The following fact is obvious.
 
\begin{lemma} 
\label{le:double quiver}
Let $(Q,\bfd)$, $(Q',\bfd')$ be valued quivers with $Q_0=\{1,2,\ldots,n\}$ and  $Q'_0=\{1,2,\ldots,r\}$) respectively and $C$ be an integer $r\times n$ matrix 
such that 
$$D'C= C D,$$
where $D=diag(d_1,\ldots,d_n)$, $D'=diag(d'_1,\ldots,d'_r)$. 

Then there exists  a valued quiver $(\tilde Q,\tilde\bfd)$ with 
$\tilde Q_0=\{1,2,\ldots,n+r\}$ such that
$$\tilde\bfd=(\bfd,\bfd'),
~B_{\tilde Q}=
\begin{pmatrix}
B_Q &-D^{-1}C^TD'\\
C & B_{Q'}
\end{pmatrix}.$$

 


\end{lemma} 

We will call $Q$ the \emph{principal subquiver} of $\tilde Q$.  Our default choice of extension $(\tilde Q,\tilde\bfd)$ in Lemma \ref{le:double quiver} is with $r=n$,  $(Q',\bfd')=(Q,\bfd)$ and $C=\pm I_n$ (where $I_n$ is the identity $n\times n$ matrix) so that  
\begin{equation}
\label{eq:default tilde Q}
\tilde \bfd=(\bfd,\bfd),
~B_{\tilde Q}^\pm =
\begin{pmatrix} 
B_Q & \mp I_n\\ 
\pm I_n & B_Q
\end{pmatrix}. 
\end{equation}

For a vertex $i\in Q_0$ write $\mu_i Q$ for the quiver obtained from $Q$ by reversing all arrows $a\in Q_1$ with $s(a)=i$ or $t(a)=i$.  
Let $\ii_0=(i_1,\ldots,i_n)\in Q_0^n$ be a repetition-free source adapted sequence in $Q$, i.e. $i_k$ ($k\ge 1$) is a source in the quiver $\mu_{i_{k-1}}\cdots\mu_{i_1}Q$.  

Recall that for  $\bfa\in\ZZ_{\ge0}^m$ and any sequence $\ii\in Q_0^m$ the flag variety $\cF_{\ii,\bfa}(V)$ is given by  
$$\cF_{\ii,\bfa}(V)=\{0=V_m\subset V_{m-1}\subset \cdots \subset V_1 \subset V_0=V \st  V_{k-1}/V_k\cong a_kS_{i_k}\}$$
and for each $k\in [1,m]$ denote by $\pi_k:\cF_{\ii,\bfa}(V)\to Gr(V)$ the map which assigns to each flag 
$(0=V_m\subset V_{m-1}\subset \cdots \subset V_1 \subset V_0=V)\in \cF_{\ii,\bfa}(V)$ the subobject $V_k$ of $V$.

In what follows, we identify the Grothendieck group $\cK(Q,\bfd)$ with $\ZZ^n$ via $|S_{i_k}|\mapsto \varepsilon_k$, $k=1,\ldots,n$.

\begin{theorem}
\label{th:flag-grass} Let $V\in \rep_\FF(Q,\bfd)$, $\ii=(\ii_0,\ii_0)$. Then for any ${\bf e}\in\ZZ_{\ge 0}^n$  
the  map $\pi_n$ defines a bijection 
$$\cF_{\ii,(|V|-{\bf e},{\bf e})}(V)\widetilde \to Gr_\bfe(V)\ .$$ 
\end{theorem}

\begin{proof} We will construct the inverse $\iota\st Gr_\bfe(V)\to \cF_{\ii,(|V|-{\bf e},{\bf e})}(V)$ of $\pi_n$. Since $\ii_0$ is a source adapted sequence in $Q$, given $W\subset V\in Gr_\bfe(V)$ there is a unique flag in $\cF_{\ii_0,\bfe}(W)$ which we denote by $(0=V_{2n}\subset V_{2n-1}\subset \cdots \subset V_n=W)$.  This gives the lower half of the flag associated to $W$ by $\iota$. 

 Now we recursively construct the remaining subobjects $V_k$, $k=0,\ldots,n-1$.  Set $V_0=V$ and suppose that $V_0\supset \cdots \supset V_{k-1}$ are already constructed for some $k\ge 1$.  Since $i_k$ is a source in the support of $V_{k-1}/V_n$, there is a unique surjective map from $V_{k-1}/V_n$ to $(v_{i_k}-e_{i_k})S_{i_k}$. Denote by $V_k\subset V_{k-1}$ the kernel of the composition $V_{k-1}\to V_{k-1}/V_n\to (v_{i_k}-e_{i_k})S_{i_k}$, so that $V_{k-1}/V_k\cong (v_{i_k}-e_{i_k})S_{i_k}$. 

Therefore, we have constructed a unique flag $0=V_{2n}\subset V_{2n-1}\subset \cdots \subset V_1\subset V_0=V$ 
in $\cF_{(\ii_0,\ii_0),(|V|-{\bf e},{\bf e})}(V)$ with $V_n=W$.  
Clearly $\pi_n\circ\iota$ and $\iota\circ\pi_n$ are identity maps respectively.
\end{proof}
 
\begin{remark}
In \cite{gls2}, Geiss, Leclerc, and Schr\"oer constructed a bijection  between flags in representations of the preprojective algebra and Grassmannians for certain quivers.  We will study such generalized bijections in terms of quantum cluster characters in our future work.
\end{remark}
 
\begin{corollary}
\label{cor:Feigin_grass}
For $\ii=(\ii_0,\ii_0)$ and $V\in \cC$, one has:
\begin{equation}
\label{eq:Feigin_grass}
X_{V,\ii}=\sum\limits_{\bfe\in\ZZ_{\ge0}^n} \langle \bfe, |V|-\bfe\rangle^{-1} |Gr_\bfe(V)| \cdot t^{(|V|-{\bf e},{\bf e})}
  \end{equation}
\end{corollary}

\begin{proof}
By definition \eqref{eq:Feigin character} of $X_{V,\ii}$ we have
\begin{equation}
\label{eq:Feigin_flag}
X_{V,\ii}=\sum\limits_{\bfa\in\ZZ_{\ge0}^{2n}} \prod\limits_{k<\ell}\left(\frac{\langle S_{i_k},S_{i_\ell}\rangle}{\langle S_{i_\ell},S_{i_k}\rangle}\right)^{\half a_ka_\ell}|\cF_{\ii,\bfa}(V)|\cdot t_1^{a_1}\cdots t_m^{a_m}=\sum_{\bfa\in\ZZ_{\ge0}^m} \prod\limits_{k<\ell}\langle S_{i_\ell},S_{i_k}\rangle^{-a_ka_\ell}|\cF_{\ii,\bfa}(V)|\cdot t^{\bfa}\ ,
\end{equation}
where $t^\bfa=\prod\limits_{k<\ell}(\langle S_{i_k},S_{i_\ell}\rangle\langle S_{i_\ell},S_{i_k}\rangle)^{\half a_k a_\ell}
t_1^{a_1}\cdots t_m^{a_m}$ is a basis element of the based quantum torus ${\mathcal L}_\ii$.  Note that $|\cF_{\ii,\bfa}(V)|=0$ 
unless $|V|=\sum\limits_{k=1}^{2n} a_k |S_{i_k}|$ and thus we will restrict the sum to these $\bfa\in\ZZ_{\ge0}^{2n}$.

Since $\ii=(\ii_0,\ii_0)$, we have $\langle S_{i_\ell},S_{i_k}\rangle=1$ if $1\le k<\ell\le n$ or $1\le k-n<\ell-n\le n$. Therefore,
$$\prod\limits_{k<\ell}\langle S_{i_\ell},S_{i_k}\rangle^{-a_ka_\ell}=\prod\limits_{1\le k<\ell\le 2n}\langle a_\ell S_{i_\ell},a_k S_{i_k}\rangle^{-1}=\langle \sum\limits_{\ell=n+1}^{2n} a_\ell S_{i_\ell}, \sum\limits_{k=1}^n a_k S_{i_k}\rangle^{-1}=\langle \bfe, |V|-\bfe\rangle^{-1}\ ,$$
where ${\bf e}=\sum\limits_{\ell=n+1}^{2n} a_\ell S_{i_\ell}$ determines $\bfa=(|V|-\bfe,\bfe)$.  Combining with Theorem \ref{th:flag-grass} we obtain: 
$$X_{V,\ii}=\sum\limits_{{\bf e}\in\ZZ_{\ge0}^n} \langle \bfe, |V|-\bfe\rangle^{-1}|\cF_{\ii,(|V|-\bfe,\bfe)}(V)|\cdot t^{(|V|-\bfe,\bfe)}=\sum\limits_{\bfe\in\ZZ_{\ge0}^n} \langle \bfe, |V|-\bfe\rangle^{-1} |Gr_\bfe(V)|\cdot  t^{(|V|-{\bf e},{\bf e})}\ .$$
This proves \eqref{eq:Feigin_grass}. \end{proof}


\subsection{{\texorpdfstring{$({\bf i}_0,\ii_0)$}{({\bf i}_0,\ii_0)}}-Character is Quantum Cluster Character: conclusion of the proof}
\label{sec:char=char}  
Denote by $(\tilde Q,\tilde \bfd)$ the valued quiver on $2n$ vertices given by \eqref{eq:default tilde Q} with the ``minus" sign, that is,  
$B_{\tilde Q}^- 
=\begin{pmatrix}
B_Q & I_n\\
-I_n & B_Q\\
\end{pmatrix}$, where $I_n$ is the $n\times n$ identity matrix. 

\begin{proposition}\label
{pr:Feigin_qcc}
Let $\ii=(\ii_0,\ii_0)$ be twice a source adapted sequence.  Then
\begin{equation}
\label{eq:Feigin qcc}
\widehat{\rho_\ii^{-1}\varphi_\ii}(X_{\ii,V})=\tilde X_V\ ,
\end{equation} 
for any $V\in \rep_\FF(Q,\bfd)$,  where $\tilde X_V$ stands for the quantum cluster character attached to $\tilde {\mathcal C}=\rep_\FF(\tilde Q,\tilde \bfd)$ with $\ex=\{1,\ldots,n\}$ 
(and $V$ is regarded as an object of $\tilde {\mathcal C}$). 
\end{proposition}

\begin{proof} In what follows, we identify the Grothendieck group $\cK(\tilde Q,\tilde \bfd)$ with $\ZZ^{2n}$ via $|S_i|\mapsto \varepsilon_i$ for $i=1,\ldots,2n$.
We need the following result.

\begin{lemma} 
\label{le:phii0i0}
Let $\ii=(\ii_0,\ii_0)$. Then, in the notation \eqref{eq:phi} one has:   
$$\rho_\ii^{-1}\varphi_\ii({\bf f},\bfe)=-(\bfe,0)^*-{}^*({\bf f},0)$$
for all ${\bf f},{\bf e}\in \ZZ^n$.

\end{lemma}

\begin{proof} It suffices to show that 
\begin{equation}
\label{eq:rhoek}
\rho_\ii^{-1}\varphi_\ii(\varepsilon_k)=
\begin{cases}
-{}^*\varepsilon_k & \text{if $1\le k\le n$}\\ 
-\varepsilon_{k-n}^* & \text{if $n+1\le k\le 2n$}
\end{cases}
\end{equation}
for $k=1,\ldots,2n$, 
where the elements ${}^*\varepsilon_k$, $\varepsilon_k^*$ are defined by \eqref{eq:star e}. 

Indeed, the choice of $B_{\tilde Q}^-$ implies that the elements ${}^*\varepsilon_k$, $\varepsilon_k^*$, $k=1,\ldots,n$ are equal to:
$${}^*\varepsilon_k=\varepsilon_k-\sum_{\ell=1}^{2n} [b_{\ell k}]_+\varepsilon_\ell=\varepsilon_k+\sum_{\ell=1}^{k-1} a_{i_\ell i_k} \varepsilon_\ell,~\varepsilon_k^*=\varepsilon_k-\sum_{\ell=1}^{2n} [-b_{\ell k}]_+\varepsilon_\ell=\varepsilon_k-\varepsilon_{k+n}+\sum_{\ell=k+1}^n a_{i_\ell i_k} \varepsilon_\ell$$
because $b_{\ell k}=\sgn(\ell-k)a_{i_\ell i_k}$ if $1\le k,\ell\le n$ and 
$b_{\ell,k}=-\delta_{\ell-n,k}$ if $n<\ell\le 2n$, $k=1,\ldots,n$. Therefore, using Lemma \ref{le:rho phi}, we obtain $\rho_\ii({}^*\varepsilon_k)={}^*\varepsilon_k$ and:
$$\rho_\ii(\varepsilon_k^*)=\varepsilon_k-\rho_\ii(\varepsilon_{k+n})+\sum_{\ell=k+1}^n a_{i_\ell i_k} \varepsilon_\ell=\varepsilon_k-\varepsilon_{k+n}+\sum_{\ell=n+1}^{k+n} a_{i_\ell i_k} \varepsilon_\ell+\sum_{\ell=k+1}^n a_{i_\ell i_k} \varepsilon_\ell$$
$$=\varepsilon_k+\varepsilon_{k+n}+\sum_{\ell=k+1}^{k+n-1} a_{i_\ell i_k} \varepsilon_\ell \ .
$$

Furthermore, by definition \eqref{eq:phi} one has:
$$\varphi_\ii(\varepsilon_k)=
\begin{cases}
-\varepsilon_k -\sum\limits_{\ell=1}^{k-1} a_{i_\ell i_k} \varepsilon_\ell & \text{if $1\le k\le n$}\\ 
-\varepsilon_k-\varepsilon_{k-n} -\sum\limits_{\ell=k-n+1}^{k-1} a_{i_\ell i_k} \varepsilon_\ell & \text{if $n+1\le k\le 2n$}
\end{cases}=\begin{cases}
-\rho_\ii({}^*\varepsilon_k) & \text{if $1\le k\le n$}\\ 
-\rho_\ii(\varepsilon_{k-n}^*) & \text{if $n+1\le k\le 2n$}
\end{cases}$$
which proves \eqref{eq:rhoek}.

The lemma is proved.
\end{proof}

Now we are ready to finish the proof of Proposition \ref{pr:Feigin_qcc}. Applying $\widehat{\rho_\ii^{-1}\varphi_\ii}$ to  \eqref{eq:Feigin_grass} and taking into account that $\widehat{\rho_\ii^{-1}\varphi_\ii}(t^{\bf a})=X^{\rho_\ii^{-1}\varphi_\ii({\bf a})}$ for ${\bf a}\in \ZZ^{2n}$ by Lemma \ref{le:mon_change}, we obtain using Lemma \ref{le:phii0i0}: 
$$\widehat{\rho_\ii^{-1}\varphi_\ii}(X_{V,\ii})=\sum\limits_{\bfe\in\ZZ_{\ge0}^n} \langle \bfe, |V|-\bfe\rangle^{-1} |Gr_\bfe(V)|\cdot X^{\rho_\ii^{-1}\varphi_\ii(|V|-\bfe,\bfe)}$$
$$=\sum\limits_{\bfe\in\ZZ_{\ge0}^n} \langle \bfe, |V|-\bfe\rangle^{-1} |Gr_\bfe(V)| \cdot X^{-(\bfe,0)^*-{}^*(|V|-\bfe,0)}=\tilde X_V\ .$$
Proposition~\ref{pr:Feigin_qcc} is proved.
\end{proof}


Therefore, Theorem \ref{th:character=cluster character} is proved. \endproof



\section{Quantum groups and quantum cluster algebras}

\subsection{Quantum groups, representations, and generalized minors}
\label{subsec:Quantum groups, representations, and generalized minors}

Let  $A=(a_{ij})$ be an $I\times I$ symmetrizable Cartan matrix with symmetrizing matrix $D=diag(d_i,i\in I)$. Denote by $U_q(\gg)$ the quantized enveloping algebra associated to $\gg$. It is generated over $\kk$ by $E_i,F_i,K_i,K_i^{-1}$, $i\in I$ subject to the relations (we retain notation of section \ref{sec:Special compatible pairs}):

$\bullet$ For each $i\in I$ the quadruple $E_i,F_i,K_i,K_i^{-1}$ are standard generators of $U_{q_i}(sl_2)$, where $q_i=q^{d_i}$;

$\bullet$ For each $i\in I$  the element $K_i$ acts by conjugation on each graded component $U_q(\gg)_\gamma$, $\gamma\in \cQ$ with the eigenvalue $q_i^{(\alpha_i^\vee,\gamma)}=q^{(\alpha_i,\gamma)}$;

$\bullet$ Quantum Serre relations for $i\ne j$: 
$\sum\limits_{k=0}^{1-a_{ij}} (-E_i)^{[k]}E_jE_i^{[1-a_{ij}-k]}=\sum\limits_{k=0}^{1-a_{ij}} (-F_i)^{[k]}F_jF_i^{[1-a_{ij}-k]}=0$
(Here  $x_i^{[\ell]}:=\frac{1}{[\ell]_{q_i}!}x_i^\ell$ is an $i$-th divided power, where $[\ell]_{q_i}!=[1]_{q_i}\cdots [\ell]_{q_i}$, and $[k]_{q_i}=\frac{q_i^k-q_i^{-k}}{q_i-q_i^{-1}}$).

Let $U_q(\bb_+)^*$ be the algebra generated by $U_+^*$ and $v_\lambda$, $\lambda\in  \cP$ (in the notation of Section \ref{sec:Special compatible pairs}) subject to the relations:
\begin{equation}
\label{eq:basis relations Uqb*}
v_\lambda x_i=q^{-(\alpha_i,\lambda)}x_iv_\lambda\quad\text{and}\quad v_\lambda v_\mu=v_{\lambda+\mu}
\end{equation}
for $i\in I$,  $\lambda,\mu\in \cP$, where we abbreviated $x_i:=[S_i]^*\in U_+^*$. 

It is well-known that $U_q(\bb_+)^*$ is a finite dual Hopf algebra of the quantized universal enveloping algebra $U_q(\bb_+)$, where $\bb_+$ is a Borel subalgebra of $\gg$. The algebra $U_q(\bb_+)^*$ is naturally graded by $\cP$ via $|v_\lambda|=\lambda$ and $|x_i|=-\alpha_i$ for $i\in I$.  Therefore, for a homogeneous $x\in U_+^*$ and $\lambda\in \cP$ one has:
\begin{equation}
\label{eq:basis relations Uqb* general}
v_\lambda x=q^{(|x|,\lambda)}xv_\lambda\ .
\end{equation}

The following result is well-known (see e.g., \cite[Section 3]{ber}).

\begin{lemma}
\label{le:action}
There is an action of $U_q(\gg)$ on $U_q(\bb_+)^*$ given by the following formulas (for $i\in I$):
\begin{itemize}
\item $K_i(y)=q^{(\alpha_i,|y|)}y=q_i^{(\alpha_i^\vee,|y|)}y$ for all homogeneous elements $y\in U_q(\bb_+)^*$.
\item $F_i(y)=\frac{x_i y-K_i^{-1}(y)x_i}{q_i-q_i^{-1}}$ for all $y\in U_q(\bb_+)^*$.
\item $E_i$ is a $K_i$-derivation of $U_q(\bb_+)^*$ determined by:
\begin{itemize}
\item $E_i(x_j)=\delta_{ij}$ and $E_i(v_\lambda)=0$ for all $i,j\in I$, $\lambda\in \cP$.
\item the Leibniz rule: $E_i(xy)=E_i(x)K_i(y)+xE_i(y)$ for all $x,y\in U_q(\bb_+)^*$.
\end{itemize}
\end{itemize}
\end{lemma}

It is easy to see that $V_\lambda:=U_q(\gg)(v_\lambda)\subset U_q(\bb_+)^*$ is a simple $U_q(\gg)$-module with the highest weight $\lambda$ 
for each $\lambda\in \cP$.  Denote by $\cP^+\subset \cP$ the cone of {\it dominant integral weights}, i.e., all $\lambda\in \cP$ such that 
$(\alpha_i^\vee,\lambda)\in \ZZ_{\ge 0}$ for all $i\in I$.

In particular, 
if $\lambda\in \cP^+$,  then for each $w\in W$ the module $V_\lambda\subset U_q(\bb_+)^*$ 
contains an extremal vector $v_{w\lambda}$ which is the unique (up to a scalar) element of $V_\lambda$ such that $|v_{w\lambda}|=w\lambda$. This $v_{w\lambda}$ can be computed recursively by:
\begin{equation}
\label{eq:extremal recursion}
v_{w\lambda}=(F_iK_i^\half)^{[N]}(v_{s_iw\lambda})
\end{equation}
for any $i$ such that $\ell(s_iw)=\ell(w)-1$, where $N=(\alpha_i^\vee,s_iw\lambda)$ and $\ell(w)$ is the the Coxeter length $w$.

Assume that $\kk$ has an involutive automorphism $c\mapsto \overline c$ such that $\overline {q^\half}=q^{-\half}$.  We extend it uniquely to $U_+^*$ by $\overline x_i=x_i$ for $i\in I$ and 
$$\overline {x\cdot y}=\overline y\cdot \overline x$$
for $x,y\in U_+^*$.  We further extend the bar-involution to $U_q(\bb_+)^*$ uniquely by
$$\overline {v_\lambda\cdot x}=x\cdot v_\lambda$$
for each $\lambda\in \cP$, $x\in U_+^*$.

The following fact is a direct consequence of Lemma~\ref{le:action}. 
\begin{lemma} 
\label{lemma:bar commutation}
For each $y\in U_q(\bb_+)^*$ and $i\in I$ one has:
$$\overline {F_iK_i^{\half}(y)}=F_iK_i^{\half}(\overline y)\quad\text{and}\quad\overline {K_i^{-\half}E_i(y)}=K_i^{-\half}E_i(\overline y)\ .$$
\end{lemma}

Furthermore, we define the {\it generalized quantum minor} $\Delta_{w\lambda}\in U_+^*$ by 
$$\Delta_{w\lambda}:=q^{-\half (w\lambda-\lambda,\lambda)}v_{w\lambda}\cdot v_\lambda^{-1}\ .$$

\begin{lemma}
For any dominant $\lambda\in\cP$ and $w\in W$, both $v_{w\lambda}$ and $\Delta_{w\lambda}$ are bar-invariant.
\end{lemma}
\begin{proof}
We will proceed by induction in the length $\ell(w)$ of $w\in W$. Indeed, for $\ell(w)=0$, i.e. $w=1$, we have nothing to prove as $\overline{v_\lambda}=v_\lambda$ by definition.  Suppose that $\ell(w)>1$, fix $i\in I$ such that $\ell(s_iw)<\ell(w)$. 
Since $\overline{[N]_{q_i}!}=[N]_{q_i}!$, we obtain by \eqref{eq:extremal recursion} and Lemma \ref{lemma:bar commutation}: 
$$\overline {v_{w\lambda}}=\overline{(F_iK_i^\half)^{[N]}(v_{s_iw\lambda})}=(F_iK_i^\half)^{[N]}(\overline{v_{s_iw\lambda}})=(F_iK_i^\half)^{[N]}(v_{s_iw\lambda})=v_{w\lambda}\ .$$
Finally, 
$$\overline {\Delta_{w\lambda}}=q^{\half (w\lambda-\lambda,\lambda)}\overline{v_{w\lambda}\cdot v_\lambda^{-1}}=q^{\half (w\lambda-\lambda,\lambda)}v_\lambda^{-1}\cdot v_{w\lambda}=q^{-\half (w\lambda-\lambda,\lambda)}v_{w\lambda}\cdot v_\lambda^{-1}= \Delta_{w\lambda}$$
because $v_\lambda^{-1}=v_{-\lambda}$, $|v_{w\lambda}|=w\lambda$, and  $v_{-\lambda}\cdot v_{w\lambda}=q^{(w\lambda-\lambda,-\lambda)} v_{w\lambda}\cdot v_\lambda^{-1}$ by \eqref{eq:basis relations Uqb* general}.
\end{proof}


The following is an analogue of \cite[(10.2)]{bz5}, it also follows from \cite[Lemma 10.3]{bz5}.
\begin{lemma}
\label{le:v_wlambda commute}
If $w,w'\in W$ such that $\ell(ww')=\ell(w)+\ell(w')$, then for any $\lambda,\mu\in \cP^+$ one has: 
$$v_{w\mu}\cdot v_{w w'\lambda}=q^{(w'\lambda-\lambda,\mu)}v_{w w'\lambda}\cdot v_{w\mu}\ .$$
\end{lemma}

We conclude the section with an explicit recursion on generalized quantum minors.

\begin{proposition} For each $\lambda\in \cP^+$, $w\in W$, and $i\in I$ such that $\ell(s_iw)<\ell(w)$, one has:
\begin{equation}
\label{eq:recursion extremal explicit}
v_{w\lambda}=\frac{q_i^\frac{N}{2}}{(q_i-q_i^{-1})^N}\sum_{k=0}^N x_i^{[N-k]}\cdot  v_{s_iw\lambda}\cdot (-q_i^{-1}x_i)^{[k]}\ ,
\end{equation}
\begin{equation}
\label{eq:recursion minor explicit}
\Delta_{w\lambda}=\frac{q_i^\frac{N}{2}}{(q_i-q_i^{-1})^N}\sum_{k=0}^N (q_i^{\frac{(\alpha_i^\vee,\lambda)}{2}}x_i)^{[N-k]} \Delta_{s_iw\lambda}\cdot (-q_i^{-1-\frac{(\alpha_i^\vee,\lambda)}{2}}x_i)^{[k]}\ ,
\end{equation}
where $N=-(\alpha_i^\vee,w\lambda)$.
\end{proposition}

\begin{proof} 
This is a direct consequence of the following result.
\begin{lemma} 
For each $i\in I$, $y\in U_q(\bb_+)^*$ and $N\ge 0$ one has:
$$F_i^{[N]}(y)=(q_i-q_i^{-1})^{-N}\sum_{k=0}^N x_i^{[N-k]}\cdot K_i^{-k}(y)\cdot (-q_i^{N-1}x_i)^{[k]}\ .$$
In particular, if $y$ is a homogeneous element of $U_q(\bb_+)^*$, then 
\begin{equation}
\label{eq:FiN homogeneous}
F_i^{[N]}(y)=(q_i-q_i^{-1})^{-N}\sum_{k=0}^N x_i^{[N-k]}\cdot  y\cdot (-q_i^{N-M-1}x_i)^{[k]}\ ,
\end{equation}
where $M=(\alpha_i^\vee,|y|)$. 
\end{lemma}
\begin{proof}
We proceed by induction in $N$. If $N=0$ we have nothing to prove. If $N=1$, then the assertion follows from Lemma~\ref{le:action}.  The induction step follows easily from the binomial identity
$$q_i^{-k}\begin{bmatrix}
N\\
k
\end{bmatrix}_{q_i}+q_i^{N+1-k}\begin{bmatrix}
N\\
k-1
\end{bmatrix}_{q_i}=\begin{bmatrix}
N+1\\
k
\end{bmatrix}_{q_i}\ $$
where
$\begin{bmatrix}
N\\
k
\end{bmatrix}=\frac{[N]_{q_i}!}{[k]_{q_i}![N-k]_{q_i}!}$.
\end{proof}
  
In particular, if $\ell(s_iw)<\ell(w)$, then taking $y=v_{s_iw\lambda}$  and $N=M=(\alpha_i^\vee,s_iw\lambda)=-(\alpha_i^\vee,w\lambda)$ in  \eqref{eq:FiN homogeneous}, we obtain:
\begin{equation}
\label{eq:extremal recursion enhanced}
F_i^{[N]}(v_{s_iw\lambda})=(q_i-q_i^{-1})^{-N}\sum_{k=0}^N x_i^{[N-k]}\cdot v_{s_iw\lambda}\cdot (-q_i^{-1}x_i)^{[k]}\ .
\end{equation}

Let $F'_i:=F_iK_i^\half$. Taking into account that $F_i^N=(F'_iK_i^{-\half})^N=q_i^{\frac{N(N-1)}{2}}{F'_i}^N K_i^{-\frac{N}{2}}$, \eqref{eq:extremal recursion} becomes:
$$F_i^{[N]}(v_{s_iw\lambda})=q_i^{\frac{N(N-1)}{2}}{F'_i}^{[N]} K_i^{-\frac{N}{2}}(v_{s_iw\lambda})=q_i^{-\frac{N}{2}}{F'_i}^{[N]}(v_{s_iw\lambda})=q_i^{-\frac{N}{2}}v_{w\lambda}\ .$$
Combining this with \eqref{eq:extremal recursion enhanced}, we obtain \eqref{eq:recursion extremal explicit}. 

Prove \eqref{eq:recursion minor explicit} now. Indeed, using \eqref{eq:recursion extremal explicit} we obtain

$$\Delta_{w\lambda} =q^{-\half (w\lambda-\lambda,\lambda)}v_{w\lambda}\cdot v_\lambda^{-1}=\frac{q_i^\frac{N}{2}}{(q_i-q_i^{-1})^N}\sum_{k=0}^N x_i^{[N-k]}\cdot  v_{s_iw\lambda}\cdot (-q_i^{-1}x_i)^{[k]}\cdot q^{-\half (w\lambda-\lambda,\lambda)}v_\lambda^{-1}$$
$$=\frac{q_i^\frac{N}{2}}{(q_i-q_i^{-1})^N}\sum_{k=0}^N x_i^{[N-k]}\cdot  q^{-\half (w\lambda-\lambda,\lambda)} v_{s_iw\lambda}\cdot (-q_i^{-1}x_i)^{[k]}v_\lambda^{-1}$$
$$=\frac{q_i^\frac{N}{2}}{(q_i-q_i^{-1})^N}\sum_{k=0}^N x_i^{[N-k]}\cdot  q^{-\half (w\lambda-s_iw\lambda,\lambda)} \Delta_{s_iw\lambda}\cdot v_\lambda (-q_i^{-1}x_i)^{[k]}v_\lambda^{-1}$$
$$=\frac{q_i^\frac{N}{2}}{(q_i-q_i^{-1})^N}\sum_{k=0}^N x_i^{[N-k]}\cdot  q^{-\half (w\lambda-s_iw\lambda+2k\alpha_i,\lambda)} \Delta_{s_iw\lambda}\cdot (-q_i^{-1}x_i)^{[k]}$$
$$=\frac{q_i^\frac{N}{2}}{(q_i-q_i^{-1})^N}\sum_{k=0}^N (q_i^{\frac{(\alpha_i^\vee,\lambda)}{2}}x_i)^{[N-k]} \Delta_{s_iw\lambda}\cdot (-q_i^{-1}\cdot q_i^{-\frac{(\alpha_i^\vee,\lambda)}{2}}x_i)^{[k]}$$
where we used the fact that $w\lambda-s_iw\lambda+N\alpha_i=0$.  This proves \eqref{eq:recursion minor explicit}. 

The proposition is proved.
\end{proof}

\subsection{Quantum Cluster Algebras}
\label{sec:qca} 
Fix a field $\kk$  of characteristic zero, a natural number $m$ and an $n$-element subset $\ex$ of $\{1,\ldots,m\}$ ($n\le m$) and consider a pair $(\chi,\tilde B)$, where $\chi:\ZZ^m\times \ZZ^m\to \kk^\times$ is a unitary bicharacter of $\ZZ^m$, and $B=({\bf b}^k\,|\,k\in \ex)$ is an $m\times \ex$ matrix. 
Following \cite{bz5}, we say that the pair $(\chi,\tilde B)$ is {\it compatible} if (cf. \eqref{eq:categorical compatibility}):
$$\chi({\bf b}^k,\varepsilon_i)=q_k^{\delta_{ik}}$$
for all $k\in \ex$, $i=1,\ldots,m$ where $q_k\in \kk^\times$, $k\in \ex$ are some non-roots of unity.


Fix a skew-field $\cF$.  A \emph{quantum seed} in $\cF$ is a pair $({\bf X}, \tilde B)$,  where

$\bullet$ ${\bf X}=\{X_1,\ldots,X_m\}$ is (an algebraically independent) generating set of $\cF$ such that: 
$$X_iX_j=q_{ij}X_jX_i$$
for all $i,j=1,\ldots,m$ for some $q_{ij}\in \kk^\times$ satisfying $q_{ii}=1$, $q_{ij}q_{ji}=1$ for all $i,j=1,\ldots,m$;

$\bullet$ $\tilde B=(b_{ij})$ is an  integer $m\times n$ matrix such that the pair $(\chi,\tilde B)$  is a compatible, 
where $\chi$  is the unique unitary bicharacter of $\ZZ^m$ such that $\chi(\varepsilon_i,\varepsilon_j)=q_{ij}$ for all $i,j\in [1,m]$. 

Furthermore, we refer to ${\bf X}$ as a {\it quantum cluster}, and for each ${\bf a}=(a_1,\ldots,a_m)\in \ZZ^m$ denote: 
\begin{equation}\label{eq:cluster monomial}
X^{\bf a}:=\prod_{1\le i<j\le m} q_{ji}^{\frac{1}{2}a_ia_j}X_1^{a_1}\cdots X_m^{a_m}\ .
\end{equation} 
If $a_1,\ldots,a_n$ are nonnegative integers, we refer to $X^{\bf a}$ as {\it cluster monomial}. By construction, 
\begin{equation}
\label{eq:q-based torus}
X^{\bf a}\cdot X^{\bf b}=\chi({\bf a},{\bf b})X^{{\bf a}+{\bf b}}\ ,
\end{equation}
for all ${\bf a},{\bf b}\in \ZZ^m$, 
that is, the subalgebra of $\cF$ generated by $X_1,\ldots,X_m$ is isomorphic to the based quantum torus 
${\mathcal T}_{\chi}$ defined in \eqref{eq:based quantum torus}.

To each quantum seed we associate  the upper quantum  cluster algebra $\UU({\bf X}, \tilde B)$ by the formula:
\begin{equation}\label{eq:qupper-bound}
\UU({\bf X}, \tilde B) = \bigcap_{j\in \ex} \kk[X_1^{\pm 1}, \ldots, X_{j-1}^{\pm 1}, X_j, X'_j, X_{j+1}^{\pm 1}, \ldots, X_m^{\pm 1}] \ ,
\end{equation}
(here $\kk[S]$ denotes the subalgebra of ${\mathcal F}$ generated by $S$) where $X_i=X^{\varepsilon_i}$ for all $i$ and \begin{equation}
\label{eq:neighboring clusters}
X'_k=X^{[{\bf b}^k]_+-\varepsilon_k}+X^{[-{\bf b}^k]_+-\varepsilon_k}
\end{equation}
for $k\in \ex$.

Given a quantum seed $({\bf X}, \tilde B)$, a $k$-th {\it mutation} $\mu_k({\bf X}, \tilde B)$  is a pair $({\bf X}',\tilde B')$, where  

$\bullet$ ${\bf X}'={\bf X}\setminus \{X_k\}\cup \{X'_k\}$, where $X'_k$ is given by \eqref{eq:neighboring clusters}; 

$\bullet$ $\tilde B'=\mu_k(\tilde B)$ is given by $b'_{ij} =
\begin{cases} 
-b_{ij} & \text{if $i=k$ or $j=k$} \\
\,\,~b_{ij} + \displaystyle\frac{|b_{ik}| b_{kj} + b_{ik} |b_{kj}|}{2} & \text{otherwise} 
\end{cases}.$

It is easy to show that each mutation $\mu_k({\bf X},\tilde B)$ is also a quantum seed. 

\begin{theorem}[\text{\cite[Theorem 5.1]{bz5}}] 
\label{th:upper}
For any quantum seed $({\bf X},\tilde B)$ and $k\in \ex$ one has: 
$$\UU(\mu_k({\bf X}, \tilde B))=\UU({\bf X}, \tilde B)\ ,$$
i.e., the upper quantum cluster algebra does not depend on the choice of a quantum seed in a mutation-equivalence class. 
\end{theorem} 

In particular, one immediately obtains the {\it quantum Laurent phenomenon}: each cluster ${\bf X}'$ of each quantum seed $\mu_{j_1}\cdots \mu_{j_\ell}(\Sigma)$ belongs to $\UU$.

In fact, there is another way to modify a seed $({\bf X},B)$ such that the upper cluster algebra does not change. The following is obvious. 

\begin{lemma}
\label{le:coeff_trans}
Let $\rho$ be a $\ZZ$-linear automorphism of $\ZZ^m$ such that $\rho(\varepsilon_k)=\delta_{k\ell}\varepsilon_\ell$ for all $k\in\ex$ and $\ell\in[1,m]$.  Then for any quantum seed $(\bfX,\tilde B)$ one has:

(a) The pair $(\rho(\bfX),\rho^{-1}(\tilde B))$ is a quantum seed, where $\rho(\bfX):=\{X^{\rho(\varepsilon_1)},\ldots,X^{\rho(\varepsilon_m)}\}$ and $\rho^{-1}(\tilde B)$ is obtained by applying $\rho^{-1}$ to each column ${\bf b}^k$ of $\tilde B$.

(b) $\UU(\rho(\bfX),\rho^{-1}(\tilde B))=\UU(\bfX,\tilde B)$.

 \end{lemma}

Given a quantum seed $\Sigma=({\bf X},\tilde B)$, define the {\it lower} cluster algebra $\underline {\mathcal A}_\Sigma$ to be the subalgebra of $\UU$  generated by ${\bf X}$ and by the neighboring cluster variables $X'_j$, $j\in \ex$ given by \eqref{eq:neighboring clusters}.

\begin{theorem}[\cite{bz5}]\label{th:lower}
$\underline {\mathcal A}_\Sigma=\UU(\Sigma)$ if and only if the seed $\Sigma$ is acyclic, i.e., there exists an ordering $(j_1,\ldots,j_n)$ of $\ex$ such that $b_{j_p,j_{p'}}\ge 0$ for all $1\le p<p'\le n$.
\end{theorem}

Although the quantum Laurent phenomenon guarantees that each cluster variable is an element of $\cT_{\chi}$, it is a non-trivial task to compute their initial cluster expansions. 
  
The main result from \cite{rupel2} solves this Laurent problem in some cases 
via the quantum cluster characters \eqref{eq:character-Gr refined}. To state the result, we need the following obvious fact. 

\begin{lemma} 
\label{le:seed categorical}
Given a finitary abelain category ${\mathcal C}$, a collection ${\bf S}=\{S_i,i\in 1,\ldots,m\}$ of simple objects without self-extensions, a subset $\ex$ of $\{1,\ldots,m\}$, and a compatible unitary bicharacter $\chi_\cC$ of $\cK(\cC)$ in the sense of \eqref{eq:categorical compatibility}. Then  
the pair $\Sigma_{\mathcal C}=({\bf X},\tilde B)$ is a quantum seed in ${\mathcal T}_{\chi_{\cC}}$, where: 

$\bullet$ ${\bf X}=\{X^{|S_1|},\ldots,X^{|S_m|}\}$ is an initial cluster in the quantum torus ${\mathcal T}_{\chi_{\cC}}$,

$\bullet$ $\tilde B$ is the restriction of the matrix $B_{\mathcal C}$ given by \eqref{eq:bij categorical} to $m\times \ex$.

\end{lemma}

By definition, for each $V\in \cC$ its cluster character $X_V$ given by \eqref{eq:character-Gr refined} belongs to ${\mathcal T}_{\chi_{\cC}}$. The following result is essentially \cite[Theorem 1.1]{rupel2}.
  
\begin{theorem}
\label{th:qcc} 
In the notation of Lemma \ref{le:seed categorical},  for each  valued quiver $(\tilde Q,\tilde \bfd)$ whose restriction to $\ex$ is acyclic, the quantum cluster character 
$$V\mapsto X_V$$ 
is a bijection between indecomposable exceptional  representations 
$V\in \cC=\rep_\FF(\tilde Q,\tilde\bfd)$ supported in $\ex$ and non-initial quantum cluster variables of  $\UU(\Sigma_\cC)$.
\end{theorem}

\subsection{Quantum  {\texorpdfstring{$\ii$}{ii}}-seeds and {\texorpdfstring{$\ii$}{ii}}-characters}
We construct quantum $\ii$-seeds in ${\mathcal L}_\ii$ for symmetrizable $I\times I$ Cartan matrices $A=(a_{ij})$ with $a_{ii}=2$ and all sequences $\ii\in I^m$. 
Indeed, given a non-root of unity $q^\half\in \kk$, define the unitary bicharacter $\chi_\ii:\ZZ^m\times \ZZ^m\to \kk^\times$ by:
$$\chi_\ii({\bf a},{\bf b})=q^{\half\Lambda_\ii({\bf a},{\bf b})},$$
where $\Lambda_\ii$ is the bilinear form given by \eqref{eq:Lambdaii general}. The following fact is obvious (cf. \eqref{eq:Pii}).

\begin{lemma}For any $A$ as above and any $\ii\in I^m$ the algebra ${\mathcal L}_\ii$ generated by $t_1,\ldots,t_m$ subject to the relations 
$$t_\ell t_k=q^{c_{i_ki_\ell}}t_kt_\ell$$
for all $1\le k<\ell\le m$, is a based quantum torus with the bicharacter $\chi_\ii$. In particular, 
$$t^{\bf a}t^{\bf b}=q^{\Lambda_\ii({\bf a},{\bf b})}t^{\bf b} t^{\bf a}$$
for all ${\bf a},{\bf b}\in \ZZ^m$.

\end{lemma}

The following result is a direct corollary of Theorem \ref{th:compatibility}.

\begin{corollary} 
\label{cor:sigma_ii}
For any $A$ as above and any $\ii\in I^m$ the pair $\Sigma_\ii=({\bf X}_\ii,\tilde B_\ii)$ is a quantum seed in ${\mathcal L}_\ii$, where: 

$\bullet$ ${\bf X}_\ii:=\{t^{\varphi_\ii^{-1}(\varepsilon_1)},\ldots,t^{\varphi_\ii^{-1}(\varepsilon_m)}\}\subset {\mathcal L}_\ii$ is a cluster in ${\mathcal L}_\ii$,

$\bullet$ $\tilde B_\ii$ is the $m\times\ex$ matrix given by \eqref{eq:Bii} (and $\ex=\ex_\ii=\{k\in [1,m]\,|\, k^+\le m\}$).

\end{corollary}

In what follows, we denote by $\UU_\ii$ the corresponding upper cluster algebra $\UU(\Sigma_\ii)$. 
We will mostly deal with the case when $A$ is a Cartan matrix and $\ii$ is reduced.

\begin{proposition} For any sequence $\ii=(i_1,\ldots,i_m)\in I^m$ and any homogeneous $x\in U_+^*$ one has (in the notation of Lemma \ref{le:action}):
\begin{equation}
\label{eq:concrete_Feigin via right action U_+^*}
\Psi_\ii(x)=\sum (K_{i_m}^{-\half}E_{i_m})^{[a_m]}\cdots  (K_{i_1}^{-\half}E_{i_1})^{[a_1]}(x) \cdot t^{\bf a}\ ,
\end{equation}
where the summation is over all ${\bf a}=(a_1,\ldots,\gamma_m)\in \ZZ_{\ge 0}^m$ such that 
$a_1\alpha_{i_1}+\cdots+a_m\alpha_{i_m}=-|x|$.
\end{proposition}

\begin{proof} Indeed, Lemmas \ref{le:primitive actions proofs} and \ref{le:action} imply that $K_{S_i}|_{U_+^*}=K_i$ and  $\underline \partial_{S_i}^{op}|_{U_+^*}=K_i^{-\half}E_i$. 
Therefore, \eqref{eq:concrete_Feigin via right action U_+^*} directly follows from the second equation \eqref{eq:concrete_Feigin via left action}.  
\end{proof}

In the notation \eqref{eq:alambda} for each $\ii\in I^m$ and $\lambda\in \cP$ define the monomial $t_\lambda=t_{\lambda,\ii}\in \cL_\ii$  by 
\begin{equation}
\label{eq:tlambda}
t_\lambda:=t^{{\bf a}_\lambda}\ .
\end{equation} 
Lemma \ref{le:alambda} and Proposition \ref{pr:commute alambda} imply the following corollary.
 
\begin{lemma} 
\label{le:tlambda properties}
In the notation of Lemma \ref{le:alambda}, for each $\ii\in I^m$  one has:

(a) For each $\ell\in [1,m]\setminus \ex$ the coefficient $X_\ell$ of the quantum cluster ${\bf X}_\ii$ equals $t_{\omega_{i_\ell}}$.

(b) $t^{\bf a}\cdot t_\lambda=v^{(|{\bf a}|,w_m\lambda+\lambda)}t_\lambda\cdot t^{\bf a}$ for all $\lambda\in \cP$ and ${\bf a}=(a_1,\ldots,a_m)\in \ZZ^m$.
 
(c) $t_\mu t_\lambda=v^{(w_m\mu-w_m^{-1}\mu,\lambda)}t_\lambda t_\mu=v^{\half (w_m\mu-w_m^{-1}\mu,\lambda)}t_{\lambda+\mu}$ for all $\lambda,\mu\in \cP$.

\end{lemma}

In view of \eqref{eq:concrete_Feigin via left action} and \eqref{eq:concrete_Feigin via right action U_+^*},  the following result essentially coincides with \text{\cite[Theorem 3.1]{ber}}.

\begin{proposition}  
\label{pr:psi tlambda} $\Psi_\ii(\Delta_{w\lambda})=t_\lambda^{-1}$
for any reduced word $\ii$ of $w\in W$ and any $\lambda\in \cP^+$.
\end{proposition}

This and Lemma \ref{le:tlambda properties}(b) imply that $\{t_\lambda,\lambda\in \cP^+\}\subset \Psi_\ii(U_+^*)$ and $\{t_\lambda^{-1},\lambda\in \cP^+\}$ is an Ore set in $\Psi_\ii(U_+^*)$. It is also clear from Lemma \ref{le:tlambda properties}(c) that $t_\lambda\in \cU_\ii$ for all $\lambda\in \cP$.

This and the following result provide some partial evidence to Conjecture \ref{conj:isomorphism Uii}(a).

\begin{proposition} 
\label{pr:X_1 X_n}
For each reduced word $\ii$ of $w\in W$ such that $m^-\ne 0$ the cluster variables $X_1$ and $X_{m^-}$ of the seed $\Sigma_\ii$ belong to the localization of $\Psi_\ii(U_+^*)$ by the Ore set $\{t_\lambda^{-1},\lambda\in \cP^+\}$.

\end{proposition}

\begin{proof} The following version of  \text{\cite[Proposition 2.5]{ber}} is immediate.

\begin{lemma}  
\label{le:psi tlambda refined}
Let $\ii\in I^m$ and  $x\in U_+^*$ be such that $\Psi_\ii(x)\in \kk^\times \cdot t_1^{a_1}\cdots  t_m^{a_m}$
for some $a_1,\ldots,a_m\ge 0$. 

$\bullet$ If $a_1>0$, then (in the notation of \eqref{eq:partial derivative}) $\Psi_\ii(\partial_{S_{i_1}}^{op}(x))\in \kk^\times \cdot t_1^{a_1-1}\cdots  t_m^{a_m}$.

$\bullet$ If $a_m> 0$, then $\Psi_\ii(\partial_{S_{i_m}}(x))\in \kk^\times \cdot t_1^{a_1}\cdots  t_m^{a_m-1}$.
\end{lemma}

Proposition \ref{pr:psi tlambda} guarantees that the lemma is applicable to each $x=\Delta_{w\lambda}$ for $w\in W$, $\lambda\in \cP^+$ so that 
$(a_1,\ldots,a_m)=-a_\lambda$, e.g., $a_1=-(\alpha_{i_1}^\vee,w\lambda)=-(w^{-1}\alpha_{i_1}^\vee,\lambda)$ and $a_m=(\alpha_{i_m}^\vee,\lambda)$. Clearly, one can choose $\lambda$ in such a way that $a_1>0$ and $a_m>0$. Then, multiplying both $\Psi_\ii(\partial_{S_{i_m}}^{op}(\Delta_{w\lambda}))$ and 
$\Psi_\ii(\partial_{S_{i_m}}(\Delta_{w\lambda}))$ by $t_\lambda=t^{a_\lambda}$, we obtain
$$t_1^{-1},t_m^{-1}\in \Psi_\ii(U_+^*)[t_\lambda,\lambda\in \cP^+]$$
This, in particular, proves the first assertion of the proposition. To prove the second one, note that
$$\varphi_\ii^{-1}(\varepsilon_{m^-})=-\sum\limits_{k=1}^{m^-}(w_k\alpha_{i_k}^\vee,w_{m^-}\omega_{i_m})\varepsilon_k =a_{w_{m^-}\omega_{i_m}}+\sum_{k=m^-+1}^m
(\alpha_{i_k}^\vee,w_k^{-1}w_{m-1}\omega_{i_m})\varepsilon_k=a_{ws_{i_m}\omega_{i_m}}-\varepsilon_m$$
by \eqref{eq:phi_inv} and \eqref{eq:alambda}, where we used that $w_{m^-}\omega_{i_m}=w_{m-1}\omega_{i_m}=ws_{i_m}\omega_{i_m}$ and $w_k^{-1}w_{m-1}\omega_{i_m}=\omega_{i_m}$ for $m^-<k<m$. Therefore 
$$X_{m^-}=t^{\varphi_\ii^{-1}(\varepsilon_{m^-})}=t^{a_{w\omega_{i_m}}-\varepsilon_m}=q^r t_m^{-1}\cdot t_{ws_{i_m}\omega_{i_m}}$$
for some $r\in \ZZ$ hence $X_{m^-}\in \Psi_\ii(U_+^*)[t_\lambda,\lambda\in \cP]$. 

The proposition is proved.
\end{proof}

We conclude the section by constructing another important seed $\hat \Sigma_\ii=(\hat {\bf X}_\bfi,\hat B_\ii)$ in $\cF=Frac(\cL_\ii)$. 

In the notation of Section \ref{subsec:Quantum groups, representations, and generalized minors} for any $\ii\in I^m$ define the subset  $\hat {\bf X}_\ii=\{\hat X_1,\ldots,\hat X_m\}$ of $P_\ii$ by 
\begin{equation}
\label{eq:psi of minors}
\hat X_k:=\Psi_\ii(\Delta_{s_{i_1}\cdots s_{i_k}\omega_{i_k}})
\end{equation}
for $k=1,\ldots,m$.  It follows from Proposition~\ref{pr:psi tlambda} and Lemma~\ref{le:alambda}(e) that the coefficients $\hat X_\ell$ ($\ell\in[1,m]\setminus\ex$) are inverse to the coefficients of the initial seed $\Sigma_\bfi$.

\begin{lemma}
\label{le:hat X commutation}
If $\ii$ is a reduced word for $w\in W$, then $\hat {\bf X}_\ii$ is a quantum cluster for $Frac(\cL_\ii)$ with 
\begin{equation}
\label{eq:hat X commutation}
\hat X_k \hat X_\ell=q^{\Lambda_\ii(\varphi_\ii^{-1}(\varepsilon_k),\varphi_\ii^{-1}(\varepsilon_\ell))}\hat X_\ell \hat X_k
\end{equation}
for all $1\le k,\ell \le m$. In particular,  the assignment $X_k\mapsto \hat X_k$, $k=1,\ldots,m$ defines a homomorphism 
\begin{equation}
\label{eq:hat eta}
\hat \eta_\ii:\cL_\ii\to Frac(\cL_\ii)
\end{equation}

\end{lemma}

\begin{proof} It follows from Lemma \ref{le:v_wlambda commute} that 
for any  $w,w'\in W$ such that $\ell(ww')=\ell(w)+\ell(w')$ and for any $\lambda,\mu\in \cP$ one has: 
$$\Delta_{w\mu}\cdot v_{w w'\lambda}=q^{(w'\lambda-\lambda,\mu)-(ww'\lambda-\lambda,\mu)}v_{w w'\lambda}\cdot \Delta_{w\mu}\ .$$
Since $|\Delta_{w\mu}|=w\mu-\mu$, multiplying both sides on the right by $q^{-\half(ww'\lambda-\lambda,\lambda)}v_\lambda^{-1}$ and applying \eqref{eq:basis relations Uqb* general} we obtain:
$$\Delta_{w\mu}\cdot \Delta_{w w'\lambda}=q^{(w'\lambda-\lambda,\mu)-(ww'\lambda-\lambda,\mu)}q^{(w\mu-\mu,\lambda)}\Delta_{w w'\lambda}\cdot \Delta_{w\mu}=q^{(w\mu-\mu,ww'\lambda+\lambda)}\Delta_{w w'\lambda}\cdot \Delta_{w\mu} \ .$$
For $w=w_k$, $ww'=w_\ell$, $\mu=\omega_k$, and $\lambda=\omega_\ell$, $k\le \ell$, this, taken together with \eqref{eq:pairing phi} gives \eqref{eq:hat X commutation}.  The second assertion follows from the fact that $\cL_\ii$ is generated by ${\bf X}_\ii=\{X_1,\ldots,X_m\}$ subject to the relations \eqref{eq:hat X commutation}.  The lemma is proved. 
\end{proof}

The following is a  particular case  of \cite[Conjecture 10.10]{bz5}.  

\begin{conjecture}  Under the hypotheses of Lemma \ref{le:hat X commutation} the restriction of $\hat \eta_\ii$ to $\cU_\ii$ is an isomorphism of algebras
$\eta_\ii:\cU_\ii\to \cU_\ii$. 

\end{conjecture}

The following obvious fact gives  a possible line of  attack on the conjecture.

\begin{lemma} 
\label{le:upper cluster isomorphism}
Let $\Sigma=({\bf X},\tilde B)$ be a quantum seed in a skew-field $\cF$  and let $\hat \Sigma=(\hat \bfX,{\tilde B})$ be a quantum seed in a skew-field $\hat \cF$. 
Assume that the corresponding unitary bicharacters $\chi,\hat \chi$ are equal. Then:

(a)  The assignment $X_k\mapsto \hat X_k$, $k=1,\ldots,m$ defines an isomorphism of algebras $\eta:\UU(\Sigma)\widetilde \to \UU(\hat \Sigma)$.

(b) For any other seed $\Sigma'=({\bf X}',\tilde B')$ the seed $\eta(\Sigma')=(\eta({\bf X}'),\tilde B')$ satisfies 
$\UU(\eta(\Sigma'))=\UU(\hat \Sigma)$.


\end{lemma}

We finish the section with the brief listing of properties of $\Psi_\ii$ from  \cite{ber}, see also \cite[Section 9.3]{bz5}.

\begin{lemma} 
\label{le:feigin reduced}
For any reduced word $\ii=(i_1,\ldots,i_m)$ for $w\in W$ one has:

(a) the kernel of $\Psi_\ii$ is the orthogonal complement (with respect to the pairing between $\cH(\cC)$ and $\cH^*(\cC)$) 
of the $\kk$-linear span ${\mathcal E}_\ii$ in $U_+$ of all monomials $[S_{i_1}]^{a_1}\cdots [S_{i_m}]^{a_m}$, $a_1,\ldots,a_m\in \ZZ_{\ge 0}$.

(b) ${\mathcal E}_\ii={\mathcal E}_{\ii'}$ for any other reduced word $\ii'$ for $w$.

(c) The images $\underline \Delta_{w\lambda}$ of $\Delta_{w\lambda}$, $\lambda\in \cP^+$ in $U_+^*/\ker\Psi_\ii$ form an Ore set.

(d) The localization $(U_+^*/\ker\Psi_\ii)[\underline \Delta_{w\lambda}^{-1},\lambda\in \cP^+]$ is $\kk_q[N^w]$.

(e) $\Psi_\ii(\Delta_{w\lambda})\in \kk^\times\cdot  t_\lambda^{-1}$ for $\lambda\in \cP^+$, where $t_\lambda\in \cL_\ii$ is defined by \eqref{eq:tlambda}.

\end{lemma} 

This implies that $\Psi_\ii$ defines an injective homomorphism of algebras:
\begin{equation}
\label{eq:Feigin localized}
\underline \Psi_\ii:\kk_q[N^w]\hookrightarrow \cL_\ii.
\end{equation}








\section{
Proof of Theorems \ref{th:cluster isomorphism}, \ref{th:twist double coxeter}, and \ref{th:corollary bz}}
\label{sec:unip}

\subsection{Proof of Theorem \ref{th:cluster isomorphism}} 
\label{subsect: proof of Theorem cluster isomorphism}
Let $A$ be an $n\times n$ symmetrizable Cartan matrix with symmetrizing matrix $D$.  Let $(Q,\bfd)$ be any {\it acyclic} valued quiver (i.e. having no oriented cycles) 
such that:  

$\bullet$ $A=A_{\bf S}$ as in \eqref{eq:cartan categorical}, where ${\bf S}=\{S_1,\ldots,S_n\}$ are simple representations of $(Q,\bfd)$;

$\bullet$ $\ii_0=(1,2,\ldots,n)$ is a repetition-free source adapted sequence for $(Q,\bfd)$;

$\bullet$ $|\FF|=q=v^2$.
 
\noindent This, in particular, implies that $A=2\cdot I_n-[B_Q]_+-[-B_Q]_+$.

It is convenient to slightly modify the initial seed $\Sigma_\ii=({\bf X}_\ii,\tilde B_\ii)$ (defined in Corollary \ref{cor:sigma_ii}) for $\ii=(\ii_0,\ii_0)$ as follows.

Denote $\Sigma_\ii^\pm =({\bf X}_\ii^\pm, {\tilde B}_\ii^\pm)$, where ${\bf X}_\ii^\pm =\rho_\ii^\pm({\bf X}_\ii)=\{t^{\varphi_\ii^{-1}\rho_\ii^\pm (\varepsilon_1)},\ldots,t^{\varphi_\ii^{-1}\rho_\ii^\pm(\varepsilon_{2n})}\}$ and ${\tilde B}_\ii^\pm 
=\begin{pmatrix}
B_Q \\
\pm I_n\\
\end{pmatrix}$,  that is, it consists of the first $n$ columns of the matrix $B_{\tilde Q}^\pm$ given by  \eqref{eq:default tilde Q}.  By Lemma \ref{le:rho of b truncated},  
$${\tilde B}_\ii^\pm =(\rho_\ii^\pm)^{-1}(\tilde B_\ii)\ ,$$ 
therefore, Lemma \ref{le:coeff_trans} guarantees that both $\Sigma_\ii^+$ and $\Sigma_\ii^-$ are quantum seeds for $\UU_\ii$, i.e., $\UU_\ii=\UU(\Sigma_\ii^+)=\UU(\Sigma_\ii^-)$.  



We will now prove Theorem \ref{th:cluster isomorphism}(a), i.e., show that $\Psi_{(\ii_0,\ii_0)}(U_+^*)\subset \UU_\ii$. 

Recall that $U_+^*$ is generated by $[S_i]^*$, $i\in [1,n]$ and it follows from \eqref{eq:Feigin homomorphism} that $\Psi_{(\ii_0,\ii_0)}([S_k]^*)=t_k+t_{k+n}$ for $k=1,\ldots,n$.  Thus it suffices to prove the following result. 

\begin{lemma}
Let $\ii=(\ii_0,\ii_0)$.  Then  $t_k+t_{n+k}\in\cU_\ii$ for $1\le k\le n$.
\end{lemma}
\begin{proof} Since each coefficient $X_{n+k}^{\pm 1}$ belongs to $\UU_\ii$ by Lemma \ref{le:alambda}(e) and Proposition \ref{pr:psi tlambda}, it suffices to show that the $k$-th mutation of the quantum seed $\Sigma_\ii^-$ results in:
\begin{equation}
\label{eq:chi t+t}
X'_k=t_k+t_{k+n}
\end{equation}
for $k=1,\ldots,n$.

Indeed, since ${\tilde B}_\ii^-=({\bf b}^1_-,\ldots,{\bf b}^n_-)$ is the $[1,2n]\times [1,n]$ submatrix of 
$B_{\tilde Q}^-$, the equation \eqref{eq:star e} under the identification $|S_k|:=\varepsilon_k$ implies that
$${}^*\varepsilon_k=\varepsilon_k-[{\bf b}^k_-]_+,~\varepsilon_k^*=\varepsilon_k-[-{\bf b}^k_-]_+ $$
for $k=1,\ldots,n$.  Combining this with  Lemma  \ref{le:phii0i0}, we obtain:
$$\rho_\ii^{-1}\varphi_\ii(\varepsilon_k)=-{}^*\varepsilon_k=
[{\bf b}^k_-]_+-\varepsilon_k,~\rho_\ii^{-1}\varphi_\ii(\varepsilon_{n+k})=-\varepsilon_k^*=[-{\bf b}^k_-]_+-\varepsilon_k\ .$$
Hence
$$(\rho_\ii^{-1}\varphi_\ii)^{-1}([{\bf b}^k_-]_+-\varepsilon_k)=\varepsilon_k,~(\rho_\ii^{-1}\varphi_\ii)^{-1}([-{\bf b}^k_-]_+-\varepsilon_k)=\varepsilon_{n+k}$$
for $k=1,\ldots,n$. 
In turn, this, implies
$$X'_k=t^{(\rho_\ii^{-1}\varphi_\ii)^{-1}([{\bf b}^k_-]_+-\varepsilon_k)}+t^{(\rho_\ii^{-1}\varphi_\ii)^{-1}([-{\bf b}^k_-]_+-\varepsilon_k)}=t_k+t_{n+k}\ .$$
This proves \eqref{eq:chi t+t} and the lemma.
\end{proof}
 
Therefore, Theorem \ref{th:cluster isomorphism}(a) is proved.


Prove Theorem~\ref{th:cluster isomorphism}(b) now. Indeed,  $\Sigma_\ii^-=\Sigma_{\tilde \cC}$ in the notation of Lemma \ref{le:seed categorical}, where $\tilde \cC=\rep_\FF(\tilde Q,\tilde \bfd)$ and $(\tilde Q,\tilde \bfd)$ is the valued quiver on $2n$ vertices as in \eqref{eq:default tilde Q}. 

Since $(Q,\bfd)$ is the (acyclic) restriction of $(\tilde Q,\tilde \bfd)$ to $\ex=\{1,\ldots,n\}$, 
both  Theorem~\ref{th:character=cluster character} and Theorem~\ref{th:qcc} are applicable and combining them we obtain that the
assignment 
$$V\mapsto \Psi_\ii([V]^*)$$ 
is a bijection between all exceptional representations of $(Q,\bfd)$ and all non-initial quantum cluster variables in $\UU_\ii=\UU(\Sigma_{\tilde \cC})$.
Therefore, Theorem \ref{th:cluster isomorphism}(b) is proved. 

It remains to prove Theorem~\ref{th:cluster isomorphism}(c). 

We need the following result.

\begin{proposition}
\label{pr:feigin pullback} 
Assume that $\ii=(\ii_0,\ii_0)$ is reduced. Then there exists an acyclic quantum seed $\Sigma'$ such that all cluster variables of  $\Sigma'$ and $\mu_j(\Sigma')$, $j=1,\ldots,n$ belong to $\Psi_\ii(U_+^*)$.
\end{proposition}

\begin{proof} 
The essential ingredient is contained in the following result.

\begin{lemma}\label{le:degrees of X} In the assumptions of Proposition \ref{pr:feigin pullback}, for $k\in [1,n-1]$ let 
$$\Sigma^{(k)}=\mu_k\cdots \mu_1\mu_n\cdots \mu_1(\Sigma_\ii^-)$$ 

Then for $j\in [1,n]$ one has:

(a) The $j$-th cluster variable $X_j^{(k)}$, $j\in [1,2n]$ of $\Sigma^{(k)}$ is homogeneous and
$$|X_j^{(k)}|=\begin{cases} 
cs_{i_1}\cdots s_{i_{j-1}}(\alpha_{i_j}) & \text{ if $1\le j\le k$}\\
s_{i_1}\cdots s_{i_{j-1}}(\alpha_{i_j}) & \text{ if $k+1\le j\le n$}\\ 
\alpha_{i_j}+c\alpha_{i_j} & \text{ if $j>n$}
\end{cases}.$$
where $c=s_{i_1}\cdots s_{i_n}$ is the corresponding Coxeter element of $W$. 

(b) The $j$-th cluster variable ${X'_j}^{(k)}$, $j\in [1,n]$ of $\mu_j(\Sigma^{(k)})$ is homogeneous and
\begin{equation}
\label{eq:|Xpjk|}
|{X'}^{(k)}_j|=\begin{cases} 
-s_{i_1}\cdots s_{i_k}(\alpha_{i_j}) & \text{if $s_{i_1}\cdots s_{i_k}(\alpha_{i_j})<0$}\\ 
cs_{i_1}\cdots s_{i_k}(\alpha_{i_j}) & \text{if $s_{i_1}\cdots s_{i_k}(\alpha_{i_j})>0$}
\end{cases}.
\end{equation}

\end{lemma}

\begin{proof} 
The first two cases of part (a) are a direct application of the main result from \cite{rupel1}.  To see the last case we note that  
$X_{n+j}=t^{\bfa_{\omega_{i_j}}}$ for $1\le j\le n$ (see Lemma~\ref{le:tlambda properties}), where $X_{n+j}$ is the coefficient of $\Sigma_\ii$ hence 
$|X_{n+j}|=|\bfa_{\omega_{i_j}}|=c^2\omega_{i_j}-\omega_{i_j}$ according to Lemma~\ref{le:alambda}(f).  Since $X^{(k)}_{n+j}=X_{n+j}^-$ for all $k$, we obtain:
$$|X_{n+j}^-| =|\varphi_\bfi^{-1}\rho_\bfi^-(\varepsilon_{n+j})|=c^2\omega_{i_j}-\omega_{i_j}-\sum\limits_{\ell=1}^ja_{i_\ell i_j}(c^2\omega_{i_\ell}-\omega_{i_\ell})=c^2\mu-\mu=(1+c)(c-1)\mu\ ,$$
where $\mu=\omega_{i_j}-\sum\limits_{\ell=1}^ja_{i_\ell i_j}\omega_{i_\ell}$. Thus, it suffices to show that $(c-1)\mu=\alpha_{i_j}$. 

Indeed, $(c-1)\mu=(s_{i_1}\cdots s_{i_j}-1)\mu$. Furthermore, let $\theta_j=-\alpha_{i_j}+\sum\limits_{\ell=1}^n a_{i_\ell,i_j}\omega_{i_\ell}$. Clearly, $\theta_j$ is $W$-invariant and 
$\mu+\theta_j=s_{i_j}\omega_{i_j}-\sum\limits_{\ell>j} a_{i_\ell,i_j}\omega_{i_\ell}$. Combining these, we obtain
$$(c-1)\mu= (s_{i_1}\cdots s_{i_j}-1)\mu=(s_{i_1}\cdots s_{i_j}-1)(\mu+\theta_j)=(s_{i_1}\cdots s_{i_j}-1)(s_{i_j}\omega_{i_j})=\omega_{i_j}-s_{i_j}\omega_{i_j}=\alpha_{i_j} \ .$$
This finishes the proof of part (a).

Prove (b) now. Since ${X'}^{(k)}_k=X_k^{(k-1)}$, in view of part (a) 
it suffices to verify \eqref{eq:|Xpjk|} only for $j\in [1,n]\setminus \{k\}$. 

Since ${\tilde B}_\bfi^-=\left(\begin{array}{c}B_Q\\-I_n\end{array}\right)$, it is easy to see that 
\begin{equation}
\label{Bii- Bii+}
\mu_n\cdots \mu_1({\tilde B}_\bfi^-)=\left(\begin{array}{c}B_Q\\I_n\end{array}\right)={\tilde B}_\bfi^+\ .
\end{equation}

Without loss of generality we assume for the rest of the proof that $\ii_0=(1,\ldots,n)$. Since $\tilde B^{(k)}=\mu_k\cdots \mu_1({\tilde B}_\bfi^+)$, it is easy to show that its $j$-th column 
${\bf b}_j^{(k)}$, $j\ne k$ is given by:
$${\bf b}_j^{(k)}=\sum_{i=1}^n \sgn(j-k)\delta_i^{j,k}a_{ij}\varepsilon_i+d_i^{j,k}\varepsilon_{n+i}\ ,$$
where $s_1\cdots s_k\alpha_j=\sum\limits_{i=1}^n d_i^{j,k}\alpha_i$ and 
$\delta_i^{j,k}=\begin{cases} 
0 & \text{if $i=j$}\\
-1 & \text{if $i\in [\min(j,k)+1,\max(j-1,k)]$}\\
1 & \text{otherwise}
\end{cases}.$ 

In particular, $[\sgn(j-k)\cdot {\bf b}_j^{(k)}]_+=\sum_{i=1}^n [\delta_i^{j,k}a_{ij}]_+\cdot\varepsilon_i+[\sgn(j-k)d_i^{j,k}]_+\cdot\varepsilon_{n+i}
$. 
Using this, we can compute $|{X'_j}^{(k)}|$, $j\in [1,n]\setminus\{k\}$ by:
\begin{equation}
\label{eq:|Xpjk| via mutation}
 |{X'_j}^{(k)}|=|{X^{(k)}}^{[\sgn(j-k){\bf b}_j^{(k)}]_+-\varepsilon_j}|=
-|X^{(k)}_j|+\sum_{i=1}^n [\delta_i^{j,k}a_{ij}]_+\cdot|X_j^{(k)}|+[\sgn(j-k) d_i^{j,k}]_+\cdot|X_{n+i}^{(k)}| 
\ .
\end{equation}


Furthermore, if $\alpha:=s_1\cdots s_k\alpha_j<0$, i.e. all $d_i^{jk}\le 0$, then necessarily $j\le k$, $a_{j+1,j}=\cdots a_{kj}=0$ hence $s_1\cdots s_j\alpha_j=\alpha$ and we have:
$$|{X'_j}^{(k)}|=-|X^{(k)}_j|-\sum_{i=j+1}^k a_{ij}|{X_i}^{(k)}|-\sum_{i=1}^n d_i^{j,k}|X_{n+i}^{(k)}|
=cs_1\cdots s_j\alpha_j-(\alpha+c\alpha)=-\alpha$$
where we used the identity 
\begin{equation}
\label{eq:sum of coeffs}
\sum\limits_{i=1}^n d_i^{j,k}|X_{n+i}^{(k)}|=s_1\cdots s_k\alpha_j+cs_1\cdots s_k\alpha_j \ .
\end{equation}

It remains to consider the case when $\alpha:=s_1\cdots s_k\alpha_j>0$, then i.e. all $d_i^{jk}\ge 0$. If $j<k$, then 
$$|{X'_j}^{(k)}|=-|X^{(k)}_j|-\sum\limits_{i\in [j+1,k]} a_{ij}|{X_i}^{(k)}|=cs_1\cdots s_j\alpha_j-\sum\limits_{i\in [j+1,k]} a_{ij}c s_1\cdots s_{i-1}\alpha_i$$
$$=cs_1\cdots s_j\alpha_j-\sum\limits_{i\in [j+1,k]} c s_1\cdots s_{i-1}(\alpha_j-s_i\alpha_j)=cs_1\cdots s_j\alpha_j-(s_1\cdots s_j\alpha_i-c\alpha)=c\alpha
$$
by the telescopic summation argument since $a_{ij}\alpha_i=\alpha_j-s_i\alpha_j$. 

Finally, if $\alpha=s_1\cdots s_k\alpha_j>0$ and $j>k$, then 
$$|{X'_j}^{(k)}|=-|X^{(k)}_j|-\sum\limits_{i\in [k+1,j-1]} a_{ij}|{X_i}^{(k)}|+\sum_{i=1}^n d_i^{j,k}|X_{n+i}^{(k)}|$$
$$=cs_1\cdots s_j\alpha_j-\sum\limits_{i\in [k+1,j-1]} a_{ij} s_1\cdots s_{i-1}\alpha_i +\alpha+c\alpha
=cs_1\cdots s_j\alpha_j-(\alpha-s_i\cdots s_{j-1}\alpha_j) +(c\alpha+\alpha)=c\alpha$$
again by the telescopic summation argument and \eqref{eq:sum of coeffs}.

The lemma is proved. 
\end{proof}



Theorem \ref{th:cluster isomorphism}(b) implies that a given cluster variable $X$ of $\cU_\ii$ equals $\tilde X_V$ for some exceptional $V\in\rep_\FF(Q,\bfd)$ if and only if $|X|$ is a positive root. Moreover, \cite[Theorem 5.1]{cx} asserts that for each exceptional object $V\in\rep_\FF(Q,\bf d)$, the corresponding element $[V]^*\in\cH^*(\cC)$ belongs to $U_+^*$ and hence $\tilde X_V=\Psi_\bfi([V]^*)$ belongs to $\Psi_\ii(U_+^*)$.  
This and Lemma \ref{le:degrees of X}(a) imply that $X_j^{(k)}\in\Psi_\ii(U_+^*)$ for all $k\in [1,n-1]$, $j\in [1,n]$ because $|X_j^{(k)}|>0$. 
In view of the above and Lemma~\ref{le:degrees of X}(b), it remains to find $k\in [1,n-1]$ such that $|{X'_j}^{(k)}|>0$ for all $j\in [1,n]\setminus\{k,k+1\}$.

To do so we need some more notation.

Denote
$$I_0=\{k\in [1,n-1]\,|\,  c\alpha_{i_k}>0\}$$
and fix a source adapted sequence $\ii'_0$ for $Q$ such that that its prefix is any appropriate ordering of $I_0$ and the remainder -- 
any appropriate ordering of $[1,n]\setminus I_0$. 

Without loss of generality, we may relabel $Q$ and all cluster variables in such a way that $\ii'_0=(1,\ldots,n)$. Then, clearly, $I_0=[1,k]$, i.e., 
$$c\alpha_1>0,\ldots,c\alpha_k>0,~c\alpha_{k+1}<0,\ldots,c\alpha_n<0\ ,$$
e.g., $s_r\alpha_\ell=\alpha_\ell$ for all $r>\ell>k$.

Therefore, Lemma \ref{le:degrees of X}(b) implies that for $j\in [1,n]\setminus \{k,k+1\}$,  one has
$$|{X'_j}^{(k)}|=\begin{cases} 
-s_1\cdots s_k(\alpha_j) & \text{if $s_1\cdots s_k(\alpha_j) <0$}\\ 
cs_1\cdots s_k(\alpha_j) & \text{if $s_1\cdots s_k(\alpha_j) >0$}
\end{cases}.$$

If $j>k$, then $s_1\cdots s_k(\alpha_j)=-c\alpha_j>0$ and $cs_1\cdots s_k(\alpha_j)=cs_1\cdots s_{j-1}(\alpha_j)>0$ hence $|{X'_j}^{(k)}|>0$. Suppose that 
$j<k$ now. If $s_1\cdots s_k(\alpha_j)<0$, then $|{X'_j}^{(k)}|=
cs_1\cdots s_k(\alpha_j)>0$ and  we are done. In the remaining case, we have $s_1\cdots s_k(\alpha_j)>0$ and $|{X'_j}^{(k)}|=
cs_1\cdots s_k(\alpha_j)$. Taking into account that
$$ s_1\cdots s_k(\alpha_j)=\sum_{i=1}^k d_i\alpha_i$$
where all $d_i\in \ZZ_{\ge 0}$, we obtain:
$$|{X'_j}^{(k)}|=c(\sum_{i=1}^k d_i\alpha_i)=\sum_{i=1}^k d_ic\alpha_i>0\ .$$

The proposition is proved.
\end{proof}

Let $\Sigma'$ be as in Proposition \ref{pr:feigin pullback}. Since $\cU_\ii=\cU(\Sigma')$ and  $\Sigma'$ is acyclic,  Theorem \ref{th:lower} guarantees that $\UU_\ii$ is generated by ${\bf X}'\cup\bigcup\limits_{j=1}^n \mu_j({\bf X}')$  and the inverses of coefficients $\{X_{n+1}^{-1},\ldots,X_{2n}^{-1}\}$. On the other hand, Proposition \ref{pr:feigin pullback} guarantees that ${\bf X}'\subset \Psi_\ii(U_+^*)$ and $\mu_j({\bf X}')\subset \Psi_\ii(U_+^*)$ for $j=1\ldots,n$. In turn, this  implies the containment 
$$\UU_\ii\subseteq \Psi_\ii(U^*_+)[X_{n+1}^{-1},\ldots,X_{2n}^{-1}]$$
in ${\mathcal L}_\ii$.
But  Theorem~\ref{th:cluster isomorphism}(a) implies the opposite containment, which proves Theorem~\ref{th:cluster isomorphism}(c).

Therefore, Theorem~\ref{th:cluster isomorphism} is proved. \endproof

\subsection{Quantum twist and Proof of Theorem \ref{th:twist double coxeter}}
\label{subsec: proof of theorem twist double coxeter}

Let $(Q,\bfd)$ be an acyclic valued quiver on vertices $Q_0=\{1,\ldots,n\}$. Without loss of generality, we assume that  $\ii_0=(1,\ldots,n)$ is a complete source adapted sequence.  

\begin{definition}
\label{def:Vij}
For $1\le i\le j\le n$ denote by $V_{ij}$ the unique (up to isomorphism) indecomposable 
representation of $(Q,\bfd)$ with $|V_{ij}|=s_i\cdots s_{j-1} \alpha_j$. 
\end{definition}
It is well-known that each $V_{ij}$ is exceptional.

\begin{proposition} 
\label{pr:from Vij to Delta}
For $1\le i\le j\le n$ one has, under specialization $q:=|\FF|^{\half}$:
\begin{equation}
\label{eq:Vij* extremal}
[V_{ij}]^*=\Delta_{s_i\cdots s_j\omega_j} \ .
\end{equation}
\end{proposition}

\begin{proof} We start with  the following corollary of \cite[Proposition 4.3.3]{cx}.

\begin{lemma}
\label{le:Vij* recursion}
For $1\le i<j\le n$, one has the recursion in $\cH^*(\rep_\FF(Q,\bfd))$:
$$[V_{ij}]^*=\frac{v_i^{\frac{N}{2}}}{(v_i-v_i^{-1})^N}\sum\limits_{k=0}^N x_i^{[N-k]}\cdot [V_{i+1,j}]^*\cdot (-v_i^{-1}x_i)^{[k]} \ ,$$
where $N:=(\alpha^\vee_i,|[V_{i+1,j}]^*|)=-(\alpha^\vee_i,|V_{i+1,j}|)$.
\end{lemma}

\begin{proof} 
Indeed, since  $(\alpha^\vee_i,|V_{i+1,j}|)\le 0$,   \cite[Proposition 4.3.3]{cx} implies that 
each $[V_{ij}]$ belongs to $U_+$ or, more precisely, one has the following recursion in $\cH(\rep_\FF(Q,\bfd))$ for $1\le i<j\le n$:
\begin{equation}
\label{eq:exceptional_recursion}
[V_{ij}]=v_i^N\sum\limits_{k=0}^N [S_i]^{[N-k]}\cdot [V_{i+1,j}]\cdot (-v_i^{-1}[S_i])^{[k]}\ ,
\end{equation}
where $N=-(\alpha^\vee_i,|V_{i+1,j}|)$ and $v_i=\langle S_i,S_i\rangle$. 

Using the isomorphism $\cH(\rep_\FF(Q,\bfd))\widetilde \to \cH^*(\rep_\FF(Q,\bfd))$ given by $[V]\mapsto |Aut(V)|\cdot \delta_{[V]}$ for $V\in \rep_\FF(Q,\bfd)$, we obtain a similar recursion for $\delta_{[V_{ij}]}\in \cH^*(\rep_\FF(Q,\bfd))$:
$$
\delta_{[V_{ij}]}=\frac{v_i^N|Aut(V_{i+1,j})|}{|Aut(V_{ij})|\cdot |Aut(S_i)|^N}\sum\limits_{k=0}^N x_i^{[N-k]}\cdot \delta_{[V_{i+1,j}]}\cdot (-v_i^{-1}x_i)^{[k]}
$$
hence
\begin{equation}
\label{eq:delta Vij recursion} 
[V_{ij}]^*=c'_{ij}\sum\limits_{k=0}^N x_i^{[N-k]}\cdot [V_{i+1,j}]^*\cdot (-v_i^{-1}x_i)^{[k]}\ ,
\end{equation}
where $c'_{ij}=\frac{v_i^N |Aut(V_{i+1,j})|}{|Aut(V_{ij})||Aut(S_i)|^N}\cdot\frac{{\langle V_{ij},V_{ij}\rangle^{-\half} f(|V_{ij}|)}}{{\langle V_{i+1,j},V_{i+1,j}\rangle^{-\half} f(|V_{i+1,j}|)}}$.  Since $|V_{ij}|=s_i|V_{i+1,j}|=|V_{i+1,j}|+N\alpha_i$ we have
$$f(|V_{ij}|)=f(|V_{i+1,j}|)\langle S_i,S_i\rangle^{\frac{N}{2}}\ .$$ 
Then using the fact that $|Aut(V_{i+1,j})|=|Aut(V_{ij})|$ (see e.g., \cite[Proposition 2.1]{dr}) and $\langle V_{ij},V_{ij}\rangle=\langle V_{i+1,j},V_{i+1,j}\rangle$, we obtain:
$$c'_{ij}=\frac{v_i^N \langle S_i,S_i\rangle^{\frac{N}{2}}}{|Aut(S_i)|^N}=\frac{v_i^{\frac{N}{2}}}{(v_i-v_i^{-1})^N}\ ,$$
because $\langle S_i,S_i\rangle=v_i$ and $|Aut(S_i)|=v_i^2-1$. 

This proves the lemma.
\end{proof}


We are now in a position to prove \eqref{eq:Vij* extremal} by induction in $j-i$. Indeed, for $j-i=0$ we will set $w=s_i=s_j$ and $\lambda=s_iw\lambda=\omega_i$ in \eqref{eq:recursion minor explicit} where $N=(\alpha_i^\vee,\omega_i)=1$.  Taking into account that $\Delta_{\omega_j}=1$, we have
$$\Delta_{s_j\omega_j}=\frac{q_i^\frac{1}{2}}{q_i-q_i^{-1}}\sum_{k=0}^1 (q_i^{\frac{1}{2}}x_i)^{[1-k]} \cdot (-q_i^{-1-\frac{1}{2}}x_i)^{[k]}=\frac{q_i^\frac{1}{2}}{(q_i-q_i^{-1})}
(q_i^{\frac{1}{2}}x_i-q_i^{-1-\frac{1}{2}}x_i)=x_j=[V_{jj}]^*\ .$$

Now assume that $j-i>0$, then specializing \eqref{eq:recursion minor explicit} at $q_i=v_i$, $\lambda=\omega_j$, $w=s_i\cdots s_j$ and using the fact that $(\alpha_i^\vee,\omega_j)=0$ if $i\ne j$, we obtain the following recursive formula for $\Delta_{s_i\cdots s_j\omega_j}$:
\begin{equation}
\label{eq:recursion minor explicit special}
\Delta_{s_i\cdots s_j\lambda}=\frac{v_i^\frac{N}{2}}{(v_i-v_i^{-1})^N}\sum_{k=0}^N x_i^{[N-k]} \Delta_{s_{i+1}\cdots s_j\omega_j}\cdot (-v_i^{-1}x_i)^{[k]}\ .
\end{equation}
Combining Lemma \ref{le:Vij* recursion} with the inductive hypotheses for $[V_{i+1,j}]^*$, we obtain by \eqref{eq:recursion minor explicit special}:
$$[V_{ij}]^*=\frac{v_i^{\frac{N}{2}}}{(v_i-v_i^{-1})^N}\sum\limits_{k=0}^N x_i^{[N-k]}\cdot \Delta_{s_{i+1}\cdots s_j\omega_j}
\cdot (-v_i^{-1}x_i)^{[k]}=\Delta_{s_i\cdots s_j\omega_j}\ .$$
The proposition is proved.
\end{proof}

To complete the proof of Theorem~\ref{th:twist double coxeter} we need to show that $\hat \eta_\ii(\cU_\ii)= \cU_\ii$, where  $\hat \eta_\ii:\cL_\ii\to Frac(\cL_\ii)$ is defined in Lemma~\ref{le:hat X commutation}.

First note that Lemma~\ref{le:hat X commutation} implies that $\hat\Sigma_\bfi=(\hat\bfX_\bfi,\tilde B_\bfi)$ is a quantum seed in $\cL_\bfi$.  Denote ${\hat\Sigma}_\bfi^+=({\hat\bfX}_\bfi^+,{\tilde B}_\bfi^+)$, where ${\hat\bfX}_\bfi^+=\rho_\bfi^+(\hat\bfX_\bfi)$ in the notation of Lemma~\ref{le:coeff_trans}. By Lemma \ref{le:coeff_trans}(a) it is a quantum seed for $\cL_\ii$.

\begin{lemma}
\label{le:hat seeds coincide}
${\hat\Sigma}_\bfi^+=\mu_n\cdots \mu_1(\Sigma_\ii^-)$.
\end{lemma}
\begin{proof}
In the notation of Section \ref{subsect: proof of Theorem cluster isomorphism}, denote 
$$\Sigma'=({\bf X}',\tilde B'):=\mu_n\cdots \mu_1(\Sigma_\ii^-)\ .$$

Clearly, $\tilde B'={\tilde B}_\bfi^+$ by \eqref{Bii- Bii+}. 

Furthermore, repeating the argument from the proof of Proposition \ref{pr:feigin pullback}, we obtain
$$X'_j=\Psi_\ii([V_{1j}]^*)$$
for $1\le j\le n$, where $V_{1j}$ is  defined in Definition \ref{def:Vij}.

It follows from Proposition \ref{pr:from Vij to Delta} that 
$$X'_j=\Psi_\ii([V_{1j}]^*)=\Psi_\ii(\Delta_{s_1\cdots s_j\omega_j})=\hat X_j$$
for $j=1,\ldots,n$. Clearly, the coefficients 
$X'_{n+1},\ldots,X'_{2n}$ of ${\bf X}'$ are inherited from the  cluster ${\bfX}_\ii^-=\{X_1^-,\ldots,X_{2n}^-\}$ therefore 
\begin{equation}\label{eq:X' t}
 X_{n+k}'=X_{n+k}^-=t^{\varphi_\bfi^{-1}\rho_\bfi^-(\varepsilon_{n+k})}=t^{-\varphi_\bfi^{-1}\rho_\bfi^+(\varepsilon_{n+k})},
\end{equation}
where we use that $\rho_\bfi^-(\varepsilon_{n+k})=-\rho_\bfi^+(\varepsilon_{n+k})$.

Following Proposition~\ref{pr:psi tlambda} and Lemma~\ref{le:alambda}(e) we have
\begin{equation}
\label{eq:X hat t}
\hat X_{n+k}=\Psi_\bfi(\Delta_{c^2\omega_{i_{n+k}}})=t_{\omega_{i_{n+k}}}^{-1}=t^{-\varphi_\bfi^{-1}(\varepsilon_{n+k})}
\end{equation}
for $1\le k\le n$.  Combining \eqref{eq:X' t} and \eqref{eq:X hat t} we see that 
$$X_{n+k}'=\rho_\bfi^+(\hat\bfX_\bfi)_{n+k}=t^{-\varphi_\bfi^{-1}\rho_\bfi^+(\varepsilon_{n+k})}={\hat X}_{n+k}^+\ .$$
This proves that  $\bfX'={\hat\bfX}_\bfi^+$.  The proposition is proved.

\end{proof}

Now we can finish the proof of Theorem~\ref{th:twist double coxeter}. Indeed, by Lemma~\ref{le:coeff_trans}(b), $\cU({\hat\Sigma}_\bfi^+)=\cU({\hat\Sigma}_\bfi)$. On the other hand, $\cU({\hat\Sigma}_\bfi^+)=\cU({\Sigma}_\bfi^+)=\cU_\ii$ by Lemma \ref{le:hat seeds coincide}.  Therefore, Theorem~\ref{th:twist double coxeter} is proved. 

\endproof

\subsection{Proof of Theorem \ref{th:corollary bz}}
\label{subsec:proof corollary bz}

For simplicity, we set $\ii_0:=(1,\ldots,n)$ as above. Then Theorem \ref{th:cluster isomorphism} and \eqref{eq:Feigin localized} guarantee that if  $\ii=(\ii_0,\ii_0)$ is a reduced word for an element $c^2$, where $c=s_1\cdots s_n$, then $\underline \Psi_\ii$ is an isomorphism 
$$\underline \Psi_\ii:\kk_q[N^{c^2}]\widetilde \to \UU_\ii \ .$$
On the other hand, by Theorem \ref{th:twist double coxeter},  we obtain: 
$$\eta(X_j)=\Psi_\ii(\Delta_{s_1\cdots s_j\omega_j})=\underline \Psi_\ii(\underline \Delta_{s_1\cdots s_j\omega_j})$$
for $j=1,\ldots,n$.

Therefore, combining this with Lemma \ref{le:feigin reduced}(e) 
we see that if $\ii=(\ii_0,\ii_0)$ is a reduced word for $c^2$, the assignment 
$$X_j\mapsto \underline\Psi_\ii^{-1}(\eta(X_j))=
\begin{cases}
\underline \Delta_{s_1\cdots s_j\omega_j} & \text{if $j\le n$}\\
\underline \Delta_{s_1\cdots s_n s_1\cdots s_{j-n}\omega_{j-n}}^{-1} & \text{if $j> n$}\\
\end{cases}$$
for $j=1,\ldots,2n$ defines an injective homomorphism   
$$\cL_\ii\hookrightarrow Frac(\kk_q[N^{c^2}])$$
whose restriction to $\UU_\ii\subset \cL_\ii$ is an isomorphism $\UU_\ii\widetilde \to \kk_q[N^{c^2}]$. This proves Theorem  \ref{th:corollary bz}.

\section{Example}
 In this section we compute a complete example to illustrate our main results.  Consider the Cartan matrix $A=\begin{pmatrix} 2 & -3\\ -1 & 2\end{pmatrix}$ with symmetrizing matrix $D=\begin{pmatrix} 1 & 0\\ 0 & 3\end{pmatrix}$.  Our example will be built on the word $\bfi=(1,2,1,2)$.  In this case $\cL_\bfi$ is the algebra over $\ZZ[v^{\pm\half}]$ generated by $t_1^{\pm1}$, $t_2^{\pm1}$, $t_3^{\pm1}$, and $t_4^{\pm1}$ subject to the commutation relations:
 \[t_2t_1=v^{-3}t_1t_2,\quad t_3t_1=v^2t_1t_3,\quad t_4t_1=v^{-3}t_1t_4,\quad t_3t_2=v^{-3}t_2t_3,\quad t_4t_2=v^6t_2t_4,\quad t_4t_3=v^{-3}t_3t_4.\]

 We may compute the initial exchange matrix as $\tilde B_\bfi=\left(\begin{array}{cc}0 & 3\\ -1 & 0\\ 1 & -3\\ 0 & 1\end{array}\right)$ and the initial cluster $\bfX_\bfi=(X_1,X_2,X_3,X_4)\subset\cL_\bfi$ given by
 \begin{equation}
  \label{eq:mon_change_ex1}
  X_1=t^{-\varepsilon_1},\quad X_2=t^{-3\varepsilon_1-\varepsilon_2},\quad X_3=t^{-2\varepsilon_1-\varepsilon_2-\varepsilon_3},\quad X_4=t^{-3\varepsilon_1-2\varepsilon_2-3\varepsilon_3-\varepsilon_4}.
 \end{equation}
 We will also need another choice of coefficients $X_3^-$ and $X_4^-$ computed by:
 \begin{equation}
  \label{eq:mon_change_ex2}
  X_3^-=X^{-\varepsilon_3}=t^{2\varepsilon_1+\varepsilon_2+\varepsilon_3},\quad X_4^-=X^{3\varepsilon_3-\varepsilon_4}=t^{-3\varepsilon_1-\varepsilon_2+\varepsilon_4}.
 \end{equation}
 The variables of $\bfX_\bfi$ or $\bfX_\bfi^-=(X_1,X_2,X_3^-,X_4^-)$ (and their inverses) form another generating set for $\cL_\bfi$ and we may use them to express the generators $t_1$, $t_2$, $t_3$, and $t_4$ as follows:
 \begin{equation}
  \label{eq:mon_change_ex3}
  t_1=X^{-\varepsilon_1},\quad t_2=X^{3\varepsilon_1-\varepsilon_2},\quad t_3=X^{-\varepsilon_1+\varepsilon_2-\varepsilon_3}=X_-^{-\varepsilon_1+\varepsilon_2+\varepsilon_3},\quad t_4=X^{-\varepsilon_2+3\varepsilon_3-\varepsilon_4}=X_-^{-\varepsilon_2+\varepsilon_4},
 \end{equation}
 where we write $X_-^\bfa$ for bar-invariant monomials in the cluster $\bfX_\bfi^-$.  It is easy to see that the commutation matrix of $\bfX_\bfi$ is given by $\left(\begin{array}{cccc} 0 & 3 & 1 & 3\\ -3 & 0 & 0 & 3\\ -1 & 0 & 0 & 3\\ -3 & -3 & -3 & 0\end{array}\right)$ and this is compatible with $\tilde B_\bfi$, verifying Theorem~\ref{th:compatibility}.  Similarly the commutation matrix of $\bfX_\bfi^-$ is $\left(\begin{array}{cccc} 0 & 3 & -1 & 0\\ -3 & 0 & 0 & -3\\ 1 & 0 & 0 & 3\\ 0 & 3 & -3 & 0\end{array}\right)$ which is compatible with the exchange matrix $\tilde B_\bfi^-=\left(\begin{array}{cc}0 & 3\\ -1 & 0\\ -1 & 0\\ 0 & -1\end{array}\right)$.

 Inside $\kk_{v^2}[N]$ we have the generalized quantum minors $\Delta_{s_1\omega_1}$, $\Delta_{s_1s_2\omega_2}$, $\Delta_{s_1s_2s_1\omega_1}$, and $\Delta_{s_1s_2s_1s_2\omega_2}$ which provide the cluster variables and coefficients of the cluster $(\hat\bfX_\bfi,\tilde B_\bfi)$.  These generalized quantum minors can easily be computed using \eqref{eq:recursion minor explicit} as follows:
 \begin{align*}
  \Delta_{s_1\omega_1}&=x_1,\\
  \Delta_{s_1s_2\omega_2}&=\frac{v^{\frac{3}{2}}x_1^3x_2-v^{\half}[3]_vx_1^2x_2x_1+v^{-\half}[3]_vx_1x_2x_1^2-v^{-\frac{3}{2}}x_2x_1^3}{(v^3-v^{-3})(v^2-v^{-2})(v-v^{-1})},\\
  \Delta_{s_1s_2s_1\omega_1}&=\frac{-v^{\half}x_1^3x_2+(v^{\frac{7}{2}}+v^{-\half}+v^{-\frac{5}{2}})x_1^2x_2x_1-(v^{\frac{5}{2}}+v^{\half}+v^{-\frac{7}{2}})x_1x_2x_1^2+v^{-\half}x_2x_1^3}{(v^3-v^{-3})(v^2-v^{-2})(v-v^{-1})},\\
  \Delta_{s_1s_2s_1s_2\omega_2}&=\frac{v^{\frac{3}{2}}x_1^3\Delta_{s_2s_1s_2\omega_2}-v^{\half}[3]_vx_1^2\Delta_{s_2s_1s_2\omega_2}x_1+v^{-\half}[3]_vx_1\Delta_{s_2s_1s_2\omega_2}x_1^2-v^{-\frac{3}{2}}\Delta_{s_2s_1s_2\omega_2}x_1^3}{(v^3-v^{-3})(v^2-v^{-2})(v-v^{-1})},\\
 \end{align*}
 where we have used the notation $[3]_v=v^2+1+v^{-2}$.  

 Applying the Feigin homomorphism $\Psi_{(1,2,1,2)}$ we obtain the cluster variables as follows:
 \begin{align*}
  \hat X_1&=\Psi_{(1,2,1,2)}(\Delta_{s_1\omega_1})=t^{\varepsilon_1}+t^{\varepsilon_3},\\
  \hat X_2&=\Psi_{(1,2,1,2)}(\Delta_{s_1s_2\omega_2})=t^{3\varepsilon_1+\varepsilon_2}+t^{3\varepsilon_1+\varepsilon_4}+[3]_v t^{2\varepsilon_1+\varepsilon_3+\varepsilon_4}+[3]_v t^{\varepsilon_1+2\varepsilon_3+\varepsilon_4}+t^{3\varepsilon_3+\varepsilon_4},\\
  \hat X_3&=\Psi_{(1,2,1,2)}(\Delta_{s_1s_2s_1\omega_1})=t^{2\varepsilon_1+\varepsilon_2+\varepsilon_3}=X^{-\varepsilon_3},\\
  \hat X_4&=\Psi_{(1,2,1,2)}(\Delta_{s_1s_2s_1s_2\omega_2})=t^{3\varepsilon_1+2\varepsilon_2+3\varepsilon_3+\varepsilon_4}=X^{-\varepsilon_4}.
 \end{align*}
 Applying the monomial change \eqref{eq:mon_change_ex3} we see that 
 $$\hat X_1=\Psi_{(1,2,1,2)}(\Delta_{s_1\omega_1})=X_-^{-\varepsilon_1}+X_-^{-\varepsilon_1+\varepsilon_2+\varepsilon_3}$$
 is the new variable obtained by mutating the seed $(\bfX_\bfi^-,\tilde B_\bfi^-)$ in direction 1 and that mutating further in direction 2 produces the new cluster variable:
 \begin{align*}
  \hat X_2=\Psi_{(1,2,1,2)}(\Delta_{s_1s_2\omega_2})
  &=X_-^{-\varepsilon_2}+X_-^{-3\varepsilon_1-\varepsilon_2+\varepsilon_4}+(v^2+1+v^{-2}) X_-^{-3\varepsilon_1+\varepsilon_3+\varepsilon_4}\\
  &\quad+(v^2+1+v^{-2})X_-^{-3\varepsilon_1+\varepsilon_2+2\varepsilon_3+\varepsilon_4} +X_-^{-3\varepsilon_1+2\varepsilon_2+3\varepsilon_3+\varepsilon_4}.
 \end{align*}

\section{Appendix: Twisted Bialgebras in Braided Monoidal Categories}\label{sec:appendix}

Let $\kk$ be a field and $\Gamma$ an additive monoid. For any unitary bicharacter $\chi\st \Gamma\times \Gamma\to \kk^\times$ let ${\mathcal C}_\chi$ be the tensor category of $\Gamma$-graded vector spaces $V=\bigoplus\limits_{\gamma\in \Gamma} V(\gamma)$ such that each component  $V(\gamma)$ is finite-dimensional. Clearly, this category is braided via $\Psi_{U,V}\st U\otimes V\to V\otimes U$ given by 
$$\Psi_{U,V}(u_\gamma\otimes v_{\gamma'})=\chi(\gamma,\gamma') \cdot v_{\gamma'}\otimes u_\gamma$$
for any $u_\gamma\in U(\gamma)$, $v_{\gamma'}\in V(\gamma')$.  

Let ${\mathcal U}=\bigoplus\limits_{\gamma\in \Gamma}{\mathcal U}(\gamma)$ be a bialgebra in ${\mathcal C}_\chi$.  Denote by $\hat {\mathcal U}$ the completion of ${\mathcal U}$ with respect to the grading, 
that is, the space of all formal series $\tilde u=\sum\limits_{\gamma\in \Gamma} u_\gamma$, where $u_\gamma\in {\mathcal U}(\gamma)$. For each such $\tilde u$ denote by $Supp(\tilde u)$ the submonoid 
of $\Gamma$ generated by $\{\gamma\st u_\gamma\ne 0\}$.

From now on we assume that for any $\gamma\in \Gamma$ the set
$$A_\gamma=\{(\gamma',\gamma'')\st \gamma'+\gamma''=\gamma\}$$ 
of two-part partitions of $\gamma$ is finite.  It is easy to see that, under this assumption, $\hat {\mathcal U}$ has a well-defined multiplication.  The coproduct on ${\mathcal U}$ extends to $\hat \Delta\st  \hat {\mathcal U}\to {\mathcal U}\hat {\bigotimes} {\mathcal U}$ so that $\hat {\mathcal U}$ becomes a {\it complete bialgebra}.  The following fact is obvious. 

\begin{lemma}\label{le:grouplike}
Let $E=\sum\limits_{\gamma\in \Gamma} E^{(\gamma)}$ be a formal series, where each $ E^{(\gamma)}\in {\mathcal U}(\gamma)$. Then $E$ is grouplike in $\hat {\mathcal U}$ (i.e., $\hat\Delta(E)=E\otimes E$) if and only if 
$$\Delta(E^{(\gamma)})=\sum_{(\gamma',\gamma'')\in A_\gamma} E^{(\gamma')}\otimes E^{(\gamma'')}$$
for each $\gamma\in\Gamma$.
\end{lemma}

As a corollary of Lemma~\ref{le:grouplike} we have the following well-known result.
\begin{lemma}
\label{le:primitive}
If $x\in\cU$ is primitive, i.e. $\Delta(x)=x\otimes1+1\otimes x$ and $\Psi_{\cU,\cU}(x\otimes x)=qx\otimes x$ for some non-root of unity $q\in \kk^\times$, then the braided exponential 
$$exp_q(x)=\sum_{k=0}^\infty \frac{1}{(k)_q!} x^k$$
of $x$ is grouplike in $\hat{\cU}$, where $(k)_q!=(1)_q\cdots (k)_q$ and $(\ell)_q=\frac{q^\ell-1}{q-1}$.
\end{lemma}

However, the product of grouplike elements is not always grouplike. We can sometimes restore the grouplike property of a product by twisting the factors with elements of an appropriate noncommutative algebra $\PP$ in $\cC_\chi$.  This, contained in Proposition~\ref{pr:product grouplike}, is the main idea behind the forthcoming theorem.

Now we define the {\it restricted dual} algebra ${\mathcal A}$ of ${\mathcal U}$. As a vector space, ${\mathcal A}$ is the set of all $\kk$-linear forms $x\st {\mathcal U}\to \kk$ such that $x$ vanishes on ${\mathcal U}(\gamma)$ for all but finitely many $\gamma\in\Gamma$. In other words, ${\mathcal A}\cong \bigoplus\limits_{\gamma\in\Gamma} {\mathcal A}(\gamma)$ where ${\mathcal A}(\gamma)={\Hom}_\kk({\mathcal U}(\gamma),\kk)$.  Clearly, ${\mathcal A}$ is an algebra in ${\mathcal C}_\chi$ with the product (resp. unit) adjoint of the coproduct $\Delta$ (resp. counit) on ${\mathcal U}$. 

Let ${\bf E}=(E_1,\ldots,E_m)$ be a family of grouplike elements 
$$E_k=\sum_{\gamma\in \Gamma} E_k^{(\gamma)}$$
in $\hat {\mathcal U}$. We say that ${\bf E}$ is $\PP$-adapted if for each $k=1,\ldots,m$ there exists a homomorphism $\tau_k$ from the monoid $Supp(E_k)$ to the multiplicative monoid of $\PP$ such that 
\begin{equation}\label{eq:tau commute}
\tau_\ell(\gamma_\ell)\tau_k(\gamma_k)=\chi(\gamma_k,\gamma_\ell)\cdot\tau_k(\gamma_k) \tau_\ell(\gamma_\ell)
\end{equation}
for all $k<\ell$ and $\gamma_k\in Supp(E_k)$.  For every $\PP$-adapted family ${\bf E}$ we define a map $\Psi_{\bf E}\st {\mathcal A} \to \PP$ by the formula 
\begin{equation}
\label{eq:abstract_Feigin}
\Psi_{\bf E}(x)=\sum_{\gamma_1\in Supp(E_1),\ldots,\gamma_m\in Supp(E_m)}x\big(E_1^{(\gamma_1)}\cdots E_m^{(\gamma_m)}\big)\tau_1(\gamma_1)\cdots \tau_m(\gamma_m)\ ,
\end{equation}
where we denote by $(x,u)\mapsto x(u)$ the natural non-degenerate evaluation pairing ${\mathcal A}\times {\mathcal U}\to \kk$.  Note that the sum in \eqref{eq:abstract_Feigin} is always finite because $x$ vanishes on all but finitely many homogeneous components of ${\mathcal U}$.  

\begin{theorem}
\label{th:abstract_Feigin} 
For any $\PP$-adapted family ${\bf E}$ of grouplike elements the map $\Psi_{\bf E}\st {\mathcal A} \to \PP$ defined by \eqref{eq:abstract_Feigin} is a homomorphism of $\Gamma$-graded algebras.
\end{theorem}
 
\begin{proof} 
For any $\kk$-algebra $\PP$ denote ${\mathcal U}_\PP:={\mathcal U} \bigotimes \PP$ and view it as an algebra with the standard (NOT braided!) algebra structure.  We will often abbreviate $u\cdot t:=u\otimes t$ for $u\in {\mathcal U}$, $t\in \PP$. 

Denote by $\hat {\mathcal U}_\PP$ the completion of ${\mathcal U}_\PP$, i.e., $\hat {\mathcal U}_\PP=\hat {\mathcal U} \bigotimes \PP$  is the space of formal series of the form $\sum_{\gamma\in \Gamma} u_\gamma\cdot t_\gamma $, where $u_\gamma\in {\mathcal U}(\gamma)$ and $t_\gamma\in \PP$.  Consider the  tensor square ${\mathcal V}_\PP={\mathcal U}_\PP\bigotimes\limits_\PP {\mathcal U}_\PP$ where the left factor is regarded as a right $\PP$-module and the right factor as a left $\PP$-module. Note that ${\mathcal V}_\PP$ is a $\PP$-bimodule in which we can write $t(u\otimes v)=(tu)\otimes v=u\otimes (tv)=(u\otimes v)t$ for any $u,v\in {\mathcal U}, t\in \PP$. Under the standard identification 
$${\mathcal V}_\PP\cong({\mathcal U} \bigotimes  {\mathcal U})\bigotimes \PP\ ,$$ 
this bimodule ${\mathcal V}_\PP$ becomes an algebra.

We will also need the completed tensor square $\hat {\mathcal V}_\PP={\mathcal U}_\PP\hat {\bigotimes\limits_\PP} {\mathcal U}_\PP$.  There is a natural morphism of $\PP$-bimodules 
$$\hat \Delta_\PP\st \hat {\mathcal U}_\PP\to \hat {\mathcal V}_\PP$$
which is the $\PP$-linear extension of the coproduct $\hat \Delta$ on  $\hat {\mathcal U}$. Clearly, $\hat \Delta_\PP$ is an algebra homomorphism. 

For each $\PP$-adapted family ${\bf E}$ define an element $\tilde E \in \hat {\mathcal U}_\PP$ as follows:
$$\tilde E=\tilde E_1\cdots \tilde E_m\ ,$$  
where 
$$\tilde E_k=\sum_{\gamma\in Supp(E_k)} E_k^{(\gamma)} \cdot\tau_k(\gamma)\ .$$

\begin{proposition}
\label{pr:product grouplike}
For any $\PP$-adapted family of grouplike elements ${\bf E}$ the element $\tilde E\in \hat {\mathcal U}_\PP$ is grouplike, i.e., 
$$\Delta_\PP(\tilde E)=\tilde E\otimes \tilde E\ .$$
\end{proposition}

\begin{proof} We need the  following fact.  
\begin{lemma}
\label{le:Ek}
In the assumptions of Proposition \ref{pr:product grouplike} one has:

(a) each $\tilde E_k$ is a grouplike element in $\hat {\mathcal U}_\PP$. 

(b) $(1\otimes \tilde E_k)(\tilde E_\ell\otimes 1)=(\tilde E_\ell\otimes 1) (1 \otimes\tilde E_k)$ for any $1\le k<l\le m$.

\end{lemma}
   
\begin{proof} To prove (a), note that by Lemma \ref{le:grouplike} we have
$$
\hat \Delta_\PP(\tilde E_k)= \sum_{\gamma\in Supp(E_k)}\Delta (E_k^{(\gamma)})\cdot \tau_k(\gamma)=\sum_{\gamma',\gamma''\in Supp(E_k)}E_k^{(\gamma')}\otimes E_k^{(\gamma'')}\tau_k(\gamma'+\gamma'')$$
$$=\sum_{\gamma',\gamma''\in Supp(E_k)}E_k^{(\gamma')}\cdot \tau_k(\gamma')\otimes E_k^{(\gamma')}\tau_k(\gamma'')=\tilde E_k\otimes \tilde E_k \ ,$$
where we have used the multiplicativity of $\tau_k$: $\tau_k(\gamma'+\gamma'')=\tau_k(\gamma')\tau_k(\gamma'')$. 

To prove (b), abbreviate $\tilde E_k^{(\gamma)}:=E_k^{(\gamma)}\cdot \tau_k(\gamma)$ for $k=1,\ldots,m$, $\gamma\in Supp(E_k)$.  Then for $k<\ell$ and $\gamma_k\in Supp(E_k)$, $\gamma_\ell\in Supp(E_\ell)$  we deduce the following commutation relation using \eqref{eq:tau commute}:
$$(1\otimes \tilde E_k^{(\gamma_k)})(\tilde E_\ell^{(\gamma_\ell)}\otimes 1)=(1\otimes E_k^{(\gamma_k)})(E_\ell^{(\gamma_\ell)}\otimes 1)\cdot\tau_k(\gamma_k)\tau_\ell(\gamma_\ell)$$
$$=(E_\ell^{(\gamma_\ell)}\otimes 1)(1\otimes E_k^{(\gamma_k)})\cdot\chi(\gamma_k,\gamma_\ell)\tau_k(\gamma_k)\tau_\ell(\gamma_\ell)=(E_\ell^{(\gamma_\ell)}\otimes 1)(1\otimes E_k^{(\gamma_k)})\cdot \tau_\ell(\gamma_\ell)\tau_k(\gamma_k)$$
$$=(\tilde E_\ell^{(\gamma_\ell)}\otimes 1)(1\otimes \tilde E_k^{(\gamma_k)})\ .$$
Since $\tilde E_k=\sum\limits_{\gamma_k\in Supp(E_k)}\tilde E_k^{(\gamma_k)} $ and $\tilde E_\ell=\sum\limits_{\gamma_\ell\in Supp(E_\ell)} \tilde E_\ell^{(\gamma_\ell)}$, we obtain the desired result:
$$(1\otimes \tilde E_k)(\tilde E_\ell \otimes 1)=\sum\limits_{\gamma_k\in Supp(E_k),\gamma_\ell\in Supp(E_\ell)}(1\otimes \tilde E_k^{(\gamma_k)}) (\tilde E_\ell^{(\gamma_\ell)}\otimes 1)$$
$$=\sum\limits_{\gamma_k\in Supp(E_k),\gamma_\ell\in Supp(E_\ell)}(\tilde E_\ell^{(\gamma_\ell)}\otimes 1)(1\otimes \tilde E_k^{(\gamma_k)})=(\tilde E_\ell\otimes 1) (1 \otimes\tilde E_k) \ .$$
Lemma~\ref{le:Ek} is proved.
\end{proof}

Now we are ready to finish the proof of Proposition \ref{pr:product grouplike}.  Using  Lemma~\ref{le:Ek} and the identities $\tilde u\otimes  \tilde v=(\tilde u\otimes 1)(1\otimes \tilde v)$, $(\tilde u\otimes 1)(\tilde v\otimes 1)=\tilde u\tilde v\otimes 1$ and $(1\otimes \tilde u)(1\otimes \tilde v)=1\otimes \tilde u\tilde v$, for any $\tilde u,\tilde v\in \hat {\mathcal U}_\PP$, we compute
$$\hat \Delta_\PP(\tilde E) =\hat \Delta_\PP(\tilde E_1\cdots \tilde E_m)=\hat \Delta_\PP(\tilde E_1)\cdots \hat\Delta_\PP(\tilde E_m)=(\tilde E_1\otimes \tilde E_1)\cdots (\tilde E_m\otimes \tilde E_m)$$
$$=(\tilde E_1\otimes 1)(1\otimes \tilde E_1)\cdots (\tilde E_m\otimes 1)(1\otimes \tilde E_m)=\big((\tilde E_1\otimes 1)\cdots(\tilde E_m\otimes 1)\big)\big((1\otimes \tilde E_1)\cdots(1\otimes \tilde E_m)\big)$$
$$=(\tilde E\otimes 1)(1\otimes \tilde E)=\tilde E\otimes \tilde E\ .$$
Proposition  \ref{pr:product grouplike} is proved. 
\end{proof}

Finally, we define the pairing ${\mathcal A}\times \hat {\mathcal U}_\PP\to \PP$ by: 
$$x(\sum u_\gamma\cdot t_\gamma)=\sum x(u_\gamma)t_\gamma\ .$$  
Clearly, the pairing is well-defined because all but finitely many terms $x(u_\gamma)$ are $0$ for each $u\in {\mathcal A}$.  For every $\PP$-adapted family ${\bf E}$ we see that the map $\Psi_{\bf E}\st {\mathcal A} \to P$ defined in \eqref{eq:abstract_Feigin} is given by the formula $\Psi_{\bf E}(x):=x(\tilde E)$.

The definition of the multiplication in $\cA$ implies that $(xy)(\tilde u)=(x\otimes y)(\hat \Delta_\PP(\tilde u))$ for all $x,y\in {\mathcal A}$ and $\tilde u\in \hat {\mathcal U}_\PP$, where 
$$(x\otimes y)(\tilde u_1\otimes \tilde u_2):=x(\tilde u_1)y(\tilde u_2)$$ 
for any $\tilde u_1,\tilde u_2\in \hat{\mathcal U}_\PP$. Thus, we have 
$$\Psi_{\bf E}(xy)=(xy)(\tilde E)=(x\otimes y)(\hat \Delta_\PP (\tilde E))=(x\otimes y)(\tilde E\otimes \tilde E)=x(\tilde E)y(\tilde E)=\Psi_{\bf E}(x)\Psi_{\bf E}(y),$$
which finishes the proof of Theorem~\ref{th:abstract_Feigin}.
 \end{proof}

For each $u\in \UU$ define the linear operators $x\mapsto u(x)$ and $x\mapsto u^{op}(x)$ on ${\mathcal A}$ by:
$$u(x)(u')=x(u'u),~
u^{op}(x)(u')=x(uu')$$
for all $u'\in \UU$, $x\in {\mathcal A}$.

Clearly, the operators $x\mapsto u(x)$ and $x\mapsto u^{op}(x)$ define respectively the left and the right $\UU$-action on ${\mathcal A}$ and $u(x),u^{op}(x)\in \AA_{\gamma'-\gamma}$ for each homogeneous $u\in \UU_\gamma$ and $x\in \AA_{\gamma'}$.

Using this in the form $x(u_1\cdots u_m)=(u_1\cdots u_m(x))(1)=u_m^{op}\cdots u_1^{op}(x)(1)$, we rewrite \eqref{eq:abstract_Feigin} for any homogeneous $x\in \AA_\gamma$ as:
\begin{equation}
\label{eq:abstract_Feigin via left action}
\Psi_{\bf E}(x)=\sum E_1^{(\gamma_1)}\cdots E_m^{(\gamma_m)}(x)\cdot \tau_1(\gamma_1)\cdots \tau_m(\gamma_m)
\ ,
\end{equation}
\begin{equation}
\label{eq:abstract_Feigin via right action}
\Psi_{\bf E}(x)=\sum {E_m^{(\gamma_m)}}^{op}\cdots {E_1^{(\gamma_1)}}^{op}(x)\cdot \tau_1(\gamma_1)\cdots \tau_m(\gamma_m)\ ,
\end{equation}
where the summation is over all $(\gamma_1,\ldots,\gamma_m)\in Supp(E_1)\times \cdots \times  Supp(E_m)$ such that 
$\gamma_1+\cdots+\gamma_m=\gamma$.

We finish with the following obvious, however, useful fact. 

\begin{lemma}
\label{le:primitive actions} Let $E\in \UU_\alpha$ be any homogeneous primitive element. Then for any $x\in \AA_\gamma$ and $y\in \AA$ one has 
$$E(yx)=\chi(\gamma,\alpha)\cdot E(y)x+ yE(x),~E^{op}(xy)=E^{op}(x)y+ \chi(\alpha,\gamma)\cdot xE^{op}(y)\ .$$
\end{lemma}

\end{document}